\documentclass[reqno,10pt,a4paper]{amsart}

\usepackage{amssymb,eucal,mathrsfs,array,setspace,geometry,enumitem,cite,tensor}
\usepackage[dvipsnames]{xcolor}
\usepackage[centertableaux]{ytableau}
\usepackage{txfonts}            % arXiv does not support newtexmath  yet. Falling back on txfonts
\usepackage{mathtools} % needed for cases* and the bra/ket constructions
\usepackage{stmaryrd} % adds llbracket etc...

\usepackage[colorlinks=true,citecolor=red,linkcolor=blue]{hyperref} % load _before_ cleveref!
\usepackage[capitalise,noabbrev]{cleveref}

\DeclareSymbolFont{largesymbols}{OMX}{zplm}{m}{n} % Replaces summation/product symbols in txmath by the palatino ones...

\setitemize{leftmargin=*}   % Removes margin from itemized lists (enumitem package)
\setenumerate{leftmargin=*, % Removes margin for enumerate lists (enumitem package)
              label=\textup{(\arabic*)}} % and makes labels upright numbers

\let\originalleft\left     % removes spurious spacing around \left and \right brackets
\let\originalright\right
\renewcommand{\left}{\mathopen{}\mathclose\bgroup\originalleft}
\renewcommand{\right}{\aftergroup\egroup\originalright}

\newcolumntype{C}{>{$}c<{$}} %Defines math mode in tabular (array package)

\numberwithin{equation}{section}
\allowdisplaybreaks

\newcommand{\sfrac}[2]{#1/#2}%{\frac{#1}{#2}}

\renewcommand{\ge}{\geq}%slant}
\renewcommand{\le}{\leq}%slant}

\DeclarePairedDelimiter{\brac}{\lparen}{\rparen} % use \brac for (...) and \brac* to automatically scale the ( and )
\DeclarePairedDelimiter{\sqbrac}{\lbrack}{\rbrack} % use \sqbrac[\big] for \bigl(...\bigr) etc...
\DeclarePairedDelimiter{\set}{\lbrace}{\rbrace}
\newcommand{\st}{\mspace{5mu} : \mspace{5mu}} % "such that" in sets

\DeclarePairedDelimiter{\norm}{\lVert}{\rVert}
\DeclarePairedDelimiter{\ang}{\langle}{\rangle}
\DeclarePairedDelimiter{\normord}{{} :}{: {}} % normal ordering ({} necessary to prevent := or =:)
\DeclarePairedDelimiter{\powser}{\llbracket}{\rrbracket} % [[ ... ]] using stmaryrd

\DeclarePairedDelimiterX{\comm}[2]{\lbrack}{\rbrack}{#1 , #2}  % commutators
\DeclarePairedDelimiterX{\acomm}[2]{\lbrace}{\rbrace}{#1 , #2} % anticommutators
\DeclarePairedDelimiterX{\super}[2]{\lparen}{\rparen}{#1 \delimsize\vert \mathopen{} #2} % for super args (m|n)

\newcommand{\killing}[2]{\kappa \brac[\big]{#1 , #2}} % the Killing form or the next best thing

\DeclareMathOperator{\pf}{pf}

\DeclareMathOperator{\id}{id}

\newcommand{\dd}{\mathrm{d}}   % d in derivatives and integrals
\newcommand{\ee}{\mathsf{e}}   % ln e = 1

\newcommand{\wun}{\mathbf{1}}  % the unit of all sorts of things

\newcommand{\poch}[2]{\brac*{#1}_{#2}}

\DeclareMathOperator{\cspn}{span}
\newcommand{\spn}[1]{\cspn\set*{#1}}                    % span of vectors

\newcommand{\lra}{\longrightarrow}
\newcommand{\ira}{\hookrightarrow}    % for injections
\newcommand{\dses}[5]{0 \lra #1 \overset{#2}{\lra} #3 \overset{#4}{\lra} #5 \lra 0} % displayed ses

\DeclareMathOperator{\ind}{Ind}
\newcommand{\Ind}[3]{\ind^{#1}_{#2} #3}
\DeclareMathOperator{\ext}{Ext}
\newcommand{\Extgrp}[3]{\ext^{#1}\brac*{#2,#3}}
\DeclareMathOperator{\cent}{C}
\newcommand{\Cent}[2]{\cent\brac*{#1,#2}}

\newcommand{\fld}[1]{\mathbb{#1}}    % for fields and related things
\newcommand{\alg}[1]{\mathfrak{#1}}  % for Lie algebras
\newcommand{\Mod}[1]{\mathcal{#1}}   % modules
\newcommand{\VOA}[1]{\mathsf{#1}}    % VOAs
\newcommand{\categ}[1]{\mathscr{#1}} % categories (requires mathrsfs package)

\newcommand{\ZZ}{\fld{Z}}

\newcommand{\CC}{\fld{C}}

\newcommand{\even}[1]{{#1}_{\overline{0}}}
\newcommand{\odd}[1]{{#1}_{\overline{1}}}
\newcommand{\rpar}[1]{\Pi{#1}}

\newcommand{\affine}[1]{\widehat{#1}}
\newcommand{\SLA}[2]{\alg{#1} \brac*{#2}}                 % Lie algebras like sl(2)
\newcommand{\SLSA}[3]{\alg{#1} \super*{#2}{#3}}           % Lie superalgebras like gl(1|1)
\newcommand{\AKMA}[2]{\affine{\alg{#1}} \brac*{#2}}       % Kac-Moody algebras %%\brac[\big]{#2}
\newcommand{\AKMSA}[3]{\affine{\alg{#1}} \super*{#2}{#3}} % Kac-Moody superalgebras %%\super[\big]
\newcommand{\osp}{\SLSA{osp}{1}{2}}                            % osp12 Lie super algebra
\newcommand{\slt}{\SLA{sl}{2}}
\newcommand{\aosp}[1]{\AKMSA{osp}{1}{2}_{#1}}%{\affine{\alg{osp}}_{#1}}%{\AKMSA{osp}{1}{2}}
\newcommand{\NSosp}[1]{\aosp{0}^{#1}}
\newcommand{\Rosp}[1]{\aosp{1/2}^{#1}}
\newcommand{\UEA}[1]{\mathsf{U}\brac*{#1}}
\newcommand{\cart}{\affine{\alg{h}}}                           % Cartan subalgebra
\newcommand{\heisvs}[1]{\Mod{H}_{#1}}                                  % vector space of free bosons
\newcommand{\hlie}[1]{\alg{H}_{#1}}                           % rank #1 Heisenberg Lie algebra

\newcommand{\bglie}{\alg{G}_{\beta\gamma}}                % beta gamma ghosts
\newcommand{\bclie}[1]{\alg{bc}_{#1}}                         % bc ghosts

\newcommand{\MinMod}[2]{\VOA{B}_{0|1}\left( #1 , #2 \right)}                % osp minimal model VOA
\newcommand{\vosp}[1]{\VOA{V}_{#1}}                                           % universal osp VOA
\newcommand{\hvoa}[1]{\VOA{H}_{#1}}                                          % Heisenberg VOA
\newcommand{\ffa}{\VOA{F}}%{\VOA{HGbc}}                                                  % free field symbol
\newcommand{\ffb}{\VOA{B}}%{\VOA{HLbc}}                                                  % free field symbol
\newcommand{\lva}[1]{\VOA{L}_{\ideal{#1}}}                                  % lattice va
\newcommand{\bgva}{\VOA{G}}
\newcommand{\bcva}{\VOA{bc}}

\newcommand{\ideal}[1]{\ang*{#1}}
\newcommand{\vsv}[1]{\chi_{#1}}                                             % Vacuum singular vector
\newcommand{\Vac}{\Omega}

\newcommand{\gtzhu}[2]{\mathsf{Zhu}^{\!#2}\sqbrac*{#1}} %A                      % g twisted Zhu's algebra
\newcommand{\zhu}[1]{\gtzhu{#1}{}}                                          % Zhu's algebra
\newcommand{\parity}{\tau}                                                  % parity automorphism
\newcommand{\tzhu}[1]{\gtzhu{#1}{\parity}}                                  % Twisted Zhu

\newcommand{\finite}[1]{\overline{#1}}
\newcommand{\sfaut}{\sigma}                          % spectral flow
\newcommand{\NS}{\mathrm{NS}}
\newcommand{\Ra}{\mathrm{R}}

\newcommand{\Ver}[1]{\Mod{V}_{#1}}                   % Verma module
\newcommand{\Irr}[1]{\Mod{L}_{#1}}                   % irreducible module
\newcommand{\NSf}[1]{\Mod{A}_{#1}}       % NS module induced from finite module
\newcommand{\fNSf}[1]{\finite{\Mod{A}}_{#1}}
\newcommand{\NSh}[1]{\Mod{B}^+_{#1}}       % NS module induced from hw module
\newcommand{\fNSh}[1]{\finite{\Mod{B}}^+_{#1}}
\newcommand{\NSl}[1]{\Mod{B}^-_{#1}}       % NS module induced from lw module
\newcommand{\fNSl}[1]{\finite{\Mod{B}}^-_{#1}}
\newcommand{\NSr}[1]{\Mod{C}_{#1}}       % NS module induced from dense module
\newcommand{\fNSr}[1]{\finite{\Mod{C}}_{#1}}
\newcommand{\Rf}[1]{\Mod{F}_{#1}}       % R module induced from finite module
\newcommand{\fRf}[1]{\finite{\Mod{F}}_{#1}}
\newcommand{\Rh}[1]{\Mod{D}^+_{#1}}       % R module induced from hw module
\newcommand{\fRh}[1]{\finite{\Mod{D}}^+_{#1}}
\newcommand{\Rl}[1]{\Mod{D}^-_{#1}}       % R module induced from lw module
\newcommand{\fRl}[1]{\finite{\Mod{D}}^-_{#1}}
\newcommand{\Rr}[1]{\Mod{E}_{#1}}       % R module induced from dense module
\newcommand{\fRr}[1]{\finite{\Mod{E}}_{#1}}
\newcommand{\bgh}{\Mod{G}^+}            % beta gamma hw
\newcommand{\dbgh}{\Mod{G}^{+\,\ast}}            % beta gamma hw
\newcommand{\bgl}{\Mod{G}^-}            % beta gamma lw
\newcommand{\bgd}[1]{\Mod{W}_{#1}}      % beta gamma dense
\newcommand{\fbgh}{\finite{\Mod{G}}^+}            % finite versions
\newcommand{\fbgl}{\finite{\Mod{G}}^-}
\newcommand{\fbgd}[1]{\finite{\Mod{W}}_{#1}}
\newcommand{\catO}{\categ{O}}           % Category O
\newcommand{\catR}{\categ{R}}           % Relaxed analogue of category O

\newcommand{\Fock}[1]{\mathbb{F}_{#1}}                       % Heisenberg Fock space
\newcommand{\NSFock}{\Fock{}^{\NS}}          % NS Fock space
\newcommand{\dNSFock}{\Fock{}^{\NS\,\ast}}
\newcommand{\RFock}{\Fock{}^{\Ra}}            % R Fock space
\newcommand{\dRFock}{\Fock{}^{\Ra\,\ast}}
\newcommand{\NSFFaFock}[1]{\tensor*[^\ffa]{\mathbb{F}}{_{#1}^{\NS}}}
\newcommand{\RFFaFock}[1]{\tensor*[^\ffa]{\mathbb{F}}{_{#1}^{\Ra}}}
\newcommand{\NSFFbFock}[1]{\tensor*[^\ffb]{\mathbb{F}}{_{#1}^{\NS}}}
\newcommand{\RFFbFock}[1]{\tensor*[^\ffb]{\mathbb{F}}{_{#1}^{\Ra}}}

\DeclarePairedDelimiter{\bra}{\langle}{\rvert}
\DeclarePairedDelimiter{\ket}{\lvert}{\rangle}
\DeclarePairedDelimiterX{\braket}[2]{\langle}{\rangle}{#1 \delimsize\vert \mathopen{} #2}
\DeclarePairedDelimiterX{\bracket}[3]{\langle}{\rangle}{#1 \delimsize\vert \mathopen{} #2 \delimsize\vert \mathopen{} #3}
\newcommand{\brab}[1]{\bra*{#1}}
\newcommand{\ketb}[1]{\ket*{#1}}
\newcommand{\braketb}[2]{\braket*{#1}{#2}}
\newcommand{\bracketb}[3]{\bracket*{#1}{#2}{#3}}
\newcommand{\NSfbra}[1]{\brab{#1;\NS}}          % bra of a NS Fock space tensor Heisenberg Fock space
\newcommand{\NSfket}[1]{\ketb{#1;\NS}}          % ket of a NS Fock space tensor Heisenberg Fock space
\newcommand{\Rfket}[1]{\ketb{#1;\Ra}}            % ket of a R Fock space tensor Heisenberg Fock space

\newcommand{\NSffbracket}[3]{\bracketb{#1;\NS}{#2}{#3;\NS}}

\newcommand{\Rffbracket}[3]{\bracketb{#1;\Ra}{#2}{#3;\Ra}}
\DeclarePairedDelimiter{\corrfn}{\langle}{\rangle}   %correlation function
\newcommand{\corrfnb}[1]{\corrfn*{#1}} %[\big]{#1}}
\newcommand{\NScorrfn}[2]{\corrfnb{#1}_{\NS}^{#2}}        %correlation function in NS sector
\newcommand{\Rcorrfn}[1]{\corrfnb{#1}_{\Ra}}          %correlation function in R sector
\newcommand{\Rbgcorrfn}[2]{\corrfnb{#1}_{\Ra}^{#2}}
\newcommand{\FNScorrfn}[3]{\corrfnb{#1}_{\NS}^{#2;#3}}

\newcommand{\kacsymbol}{\mathsf{K}}
\newcommand{\kac}[1]{\kacsymbol_{#1}}                % Kac table without quotienting by reflection symmetry
\newcommand{\NSkac}[1]{\kac{#1}^{\NS}}
\newcommand{\Rkac}[1]{\kac{#1}^{\Ra}}
\newcommand{\rkac}[1]{\overline{\kacsymbol}_{#1}}    % Kac table with quotienting by reflection symmetry
\newcommand{\rNSkac}[1]{\rkac{#1}^{\NS}}
\newcommand{\rRkac}[1]{\rkac{#1}^{\Ra}}

\newcommand{\admp}[2]{\delta^{(#1)}\brac*{#2}}      % minimal admissible partition of length #1 with last two parts #2
\newcommand{\uniqp}[2]{\epsilon^{(#1)}\brac*{#2}}   % unique admissible partition of length #1 and weight |\admp{#1}{0,0}| + #2

\DeclareMathOperator{\tr}{tr}

\newcommand{\chmap}{\mathrm{ch}}
\newcommand{\Gr}[1]{\sqbrac[\big]{#1}}          % element of a Grothendieck group/ring
\newcommand{\ch}[1]{\chmap \Gr{#1}}             % characters
\newcommand{\vop}[2]{\mathrm{V}_{#1}\brac*{#2}}
\newcommand{\SCR}{\mathcal{Q}}
\newcommand{\scr}[1]{\SCR_{#1}}

\newcommand{\scrf}[2]{\SCR_{#1}(#2)}

\newcommand{\scrs}[2]{\SCR_{#1}^{[#2]}}

\newcommand{\cyc}[2]{\Gamma(#1,#2)}

\newcommand{\Van}{\Delta}
\newcommand{\van}[1]{\Van\brac*{#1}}

\newcommand{\poly}[1]{\mathsf{#1}}
\newcommand{\powsum}[1]{\poly{p}_{#1}}                           % power sum

\newcommand{\fpowsum}[2]{\poly{p}_{#1} \brac[\big]{#2}}

\newcommand{\jack}[2]{\poly{P}_{#1}^{#2}}                        % Jack symmetric function
\newcommand{\fjack}[3]{\poly{P}_{#1}^{#2} \brac[\big]{#3}}
\newcommand{\fdjack}[3]{\poly{Q}_{#1}^{#2} \brac[\big]{#3}}

\newcommand{\jprod}[3]{\ang*{#1}_{#2}^{#3}}    % inner product for finitely many variables
\newcommand{\cjprod}[2]{\ang*{#1}^{#2}}        % inner product for infinitely many variables

\newcommand{\symiso}[2]{\rho^{#1}_{#2}}                  %Iso sym fn -> Heisenberg

\newcommand{\uea}{universal enveloping algebra}

\newcommand{\lw}{lowest-weight}
\newcommand{\lwv}{\lw{} vector}

\newcommand{\hw}{highest-weight}
\newcommand{\hwv}{\hw{} vector}

\newcommand{\sv}{singular vector}
\newcommand{\svs}{singular vectors}
\newcommand{\va}{vertex algebra}
\newcommand{\vsa}{vertex superalgebra}
\newcommand{\voa}{vertex operator algebra}
\newcommand{\vosa}{vertex operator superalgebra}
\newcommand{\ope}{operator product expansion}

\newcommand{\PBW}{Poincar\'{e}-Birkhoff-Witt}

\newcommand{\ns}{Neveu-Schwarz}

\newcommand{\rhs}{right-hand side}

\theoremstyle{plain}
\newtheorem{thm}{Theorem}[section]
\newtheorem{prop}[thm]{Proposition}
\newtheorem{lem}[thm]{Lemma}
\newtheorem{cor}[thm]{Corollary}
\newtheorem*{thm*}{Theorem}

\theoremstyle{definition} % Non-italicised text

\newtheorem{defn}[thm]{Definition}
\newtheorem*{rmk}{Remark}

\Crefname{thm}{Theorem}{Theorems}
\Crefname{prop}{Proposition}{Propositions}
\Crefname{lem}{Lemma}{Lemmas}
\Crefname{cor}{Corollary}{Corollaries}
\Crefname{defn}{Definition}{Definitions}

\begin{document}

\title[]{Admissible level \(\osp\) minimal models and their relaxed highest weight modules}

\author[S Wood]{Simon Wood}

\address[Simon Wood]{
School of Mathematics \\
Cardiff University \\
Cardiff, United Kingdom, CF24 4AG.
}

\email{woodsi@cardiff.ac.uk}


\begin{abstract}
  The minimal model \(\osp\) \vosa{s} are the simple quotients of affine
  \vosa{s} constructed from the
  affine Lie super algebra \(\aosp{}\) at certain rational values of the
  level \(k\). 
  We classify all isomorphism classes of \(\ZZ_2\)-graded simple relaxed highest weight modules over the minimal model
  \(\osp\) \vosa{s} in both the \ns{} and Ramond sectors. To this end, we
  combine free field realisations, screening operators and the theory of symmetric functions in the
  Jack basis to compute explicit presentations for the Zhu
  algebras in both the \ns{} and Ramond sectors.
  Two different free field realisations are used depending on the level. For
  \(k<-1\), the free field realisation resembles the Wakimoto free field
  realisation of affine \(\slt\) and is originally due to Bershadsky and Ooguri. It involves 1 free boson (or rank 1
  Heisenberg \va{}), one \(\beta\gamma\) bosonic ghost system and one \(bc\)
  fermionic ghost system. For \(k>-1\), the argument presented here requires
  the bosonisation of the \(\beta\gamma\) system by embedding it into an indefinite rank 2 lattice \va{}.
\end{abstract}

\maketitle

\onehalfspacing

\section{Introduction}
\label{sec:intro}

The orthosymplectic Lie superalgebra \(\osp\) is the Lie superalgebra of
endomorphisms of the vector superspace \(\CC^{1|2}\) that preserves the
standard supersymmetric bilinear form on \(\CC^{1|2}\). It is arguably the
easiest example of a finite-dimensional simple complex Lie superalgebra in
Kac's classification \cite{KacLsa77}. 
The purpose of this article is to classify the simple relaxed highest weight
modules over the minimal model \(\osp\) \vosa{s}, that is, the simple quotient
\vosa{s} constructed from the affinisation of \(\osp\) at certain
rational levels, called admissible levels.

Let \(u,v\) be integers satisfying \(u\ge 2,\ v\ge1\), \(u-v\in 2\ZZ\) and
\(\gcd\brac*{u,\frac{u-v}{2}}=1\) and let
\begin{equation}
  k_{u,v}=\frac{u-3v}{2v},\quad
  \lambda_{i,j}=\frac{i-1}{2}-\frac{1+(-1)^{i+j}}{4}-\frac{u}{2v}j,\quad
  s_{i,j}=\frac{i}{2}-\frac{u}{2v}j,\quad
  q_{i,j}=\frac{(uj-vi)^2-4v^2}{8v^2},
\end{equation}
where \(i\) and \(j\) are integers. We denote the simple relaxed highest
weight \(\aosp{}\) modules as follows (see \cref{sec:affineosp} for the precise
definitions):
\begin{itemize}
\item In the \ns{} sector, \(\ZZ_2\)-graded simple relaxed highest weight \(\aosp{}\)
  modules at level \(k\in\CC\) are characterised by the simple \(\osp\) weight module formed by the
  vectors of least conformal weight. Let \(\alpha\) denote the simple root of \(\osp\).
  \begin{enumerate}
  \item \(\NSf{\lambda}\), \(\lambda\in \ZZ_{\ge0}\), denotes the 
    simple module whose space of
    vectors of least conformal weight is a finite dimensional highest (and
    lowest) weight \(\osp\) module of highest weight \(\lambda\alpha\) with a
    highest weight vector of even parity.
  \item \(\NSh{\lambda}\), \(\lambda\in \CC\setminus\ZZ_{\ge0}\), denotes the 
    simple module whose space of
    vectors of least conformal weight is an infinite dimensional highest weight \(\osp\) module
    of highest weight \(\lambda\alpha\) with a
    highest weight vector of even parity.
  \item \(\NSl{\lambda}\), \(\lambda\in \CC\setminus\ZZ_{\le0}\), denotes the 
    simple module whose space of
    vectors of least conformal weight is an infinite dimensional lowest weight \(\osp\) module
    of lowest weight \(\lambda\alpha\) with a
    lowest weight vector of even parity.
  \item \(\NSr{\sqbrac*{\lambda},s}\), \(\sqbrac*{\lambda}\in \CC /2\ZZ,\
    s\in\CC\) satisfying \(s^2\neq \brac*{\mu+\frac{1}{2}}^2\) for all
    \(\mu\in\sqbrac*{\lambda}\cup \sqbrac*{\lambda+1}\), denotes the 
    simple module whose space of
    vectors of least conformal weight is a dense \(\osp\) module, which is
    characterised by its weight support and the action of the \(\osp\) super-Casimir operator.
    The weight support of even vectors is \(\sqbrac*{\lambda}\alpha\), while the
    weight
    support of odd vectors is \(\sqbrac*{\lambda+1}\alpha\). The
    super-Casimir acts as multiplication by \(s\) on even vectors and \(-s\) on odd vectors.
  \end{enumerate}
\item In the Ramond sector, \(\ZZ_2\)-graded simple relaxed highest weight
  \(\aosp{}\) modules at level \(k\in\CC\) 
  are characterised by the simple \(\slt\) weight module formed by the vectors of
  least conformal weight. Since the even subalgebra of \(\osp\) is isomorphic to
  \(\slt\), the simple root of this subalgebra is \(2\alpha\) (equivalently, \(\alpha\) is
  the fundamental weight).
  \begin{enumerate}[resume]
  \item \(\Rf{\lambda}\), \(\lambda\in \ZZ_{\ge0}\), denotes the 
    simple module whose space of
    vectors of least conformal weight is even and is the finite dimensional highest (and
    lowest) weight \(\slt\) module of highest weight \(\lambda\alpha\).
  \item \(\Rh{\lambda}\), \(\lambda\in \CC\setminus\ZZ_{\ge0}\), denotes the 
    simple module whose space of
    vectors of least conformal weight is even and is an infinite dimensional highest weight \(\slt\) module
    of highest weight \(\lambda\alpha\).
  \item \(\Rl{\lambda}\), \(\lambda\in \CC\setminus\ZZ_{\le0}\), denotes the 
    simple module whose space of
    vectors of least conformal weight is even and is an infinite dimensional lowest weight \(\slt\) module
    of lowest weight \(\lambda\alpha\).
  \item \(\Rr{\sqbrac*{\lambda},q}\), \(\sqbrac*{\lambda}\in \CC /2\ZZ,\
    q\in\CC\) satisfying \(q\neq \mu\brac*{\mu+2}\) for all
    \(\mu\in\sqbrac*{\lambda}\), denotes the 
    simple module whose space of
    vectors of least conformal weight is even and is a dense \(\slt\) module,
    which is characterised by its weight support and the action of the
    \(\slt\) quadratic Casimir operator.
    Its weight support is \(\sqbrac*{\lambda}\alpha\)
    and the \(\slt\) Casimir operator acts as multiplication by \(q\).
  \end{enumerate}
\end{itemize}

We denote the simple quotient of the universal \vosa{} constructed from
\(\osp\) at level \(k_{u,v}\) by \(\MinMod{u}{v}\).
The main result of this paper can then be stated as follows.
\begin{thm*}
  Every \(\ZZ_2\)-graded simple relaxed highest weight module over the minimal
  model \(\osp\) \vosa{}, \(\MinMod{u}{v}\),
  at level \(k_{u,v}\) is isomorphic to one of the
  following or their parity reversals.
  \begin{trivlist}
  \item In the \ns{} sector:
    \begin{enumerate}
    \item\label{itm:nsfintro} \(\NSf{\lambda_{i,0}}\), where \(1\le i\le u-1\) and \(i\) is odd.
    \item\label{itm:nshwintro} \(\NSh{\lambda_{i,j}}\), where \(1\le i\le u-1\), \(1\le j\le v-1\)
      and \(i+j\) is odd.
    \item \(\NSl{-\lambda_{i,j}}\), where \(1\le i\le u-1\), \(1\le j\le v-1\)
      and \(i+j\) is odd.
    \item \(\NSr{\sqbrac*{\lambda},s_{i,j}}\), where \(\sqbrac*{\lambda}\in
      \CC /2\ZZ\), \(1\le i\le u-1\), \(1\le j\le v-1\)
      and \(i+j\) is odd.
    \end{enumerate}
  \item In the Ramond sector:
    \begin{enumerate}[resume]
    \item\label{itm:rfintro} \(\Rf{\lambda_{i,0}}\), where \(1\le i\le u-1\) and \(i\) is even.
    \item\label{itm:rhwintro} \(\Rh{\lambda_{i,j}}\), where \(1\le i\le u-1\), \(1\le j\le v-1\)
      and \(i+j\) is even.
    \item \(\Rl{-\lambda_{i,j}}\), where \(1\le i\le u-1\), \(1\le j\le v-1\)
      and \(i+j\) is even.
    \item \(\Rr{\sqbrac*{\lambda},q_{i,j}}\), where \(\sqbrac*{\lambda}\in
      \CC /2\ZZ\), \(1\le i\le u-1\), \(1\le j\le v-1\)
      and \(i+j\) is even.
    \end{enumerate}
  \end{trivlist}
\end{thm*}

For the special case of \(u=2\) and \(v=4\), the above classification was
proved in \cite{SnaOsp17}.
The (non-rigorous) classification of simple highest weight \(\MinMod{u}{v}\)
modules, that is, those listed as items \ref{itm:nsfintro},
\ref{itm:nshwintro}, \ref{itm:rfintro} and \ref{itm:rhwintro} in the theorem above,
has already been established in the physics literature
by Fan and Yu, and Ennes and Ramallo \cite{Fanlsa93,EnnOsp98}.
 
In the mathematics literature, the simple modules over \(\MinMod{u}{1}\) in
the \ns{} sector were classified by Kac and Wang in \cite{KacVer94} and those
in the Ramond sector were later classified by Creutzig, Frohlich and Kanade \cite{CreOsp17}. In this case all simple
modules are of the type listed in part \ref{itm:nsfintro} of the above
classification theorem. For general \(u,v\), the \ns{} sector admissible simple highest
weight \(\aosp{}\) modules (modules whose span of analytic continuations of
characters admit a closed action of the modular group) were classified by
Kac and Wakimoto \cite{KacAdm88} and this classification matches the modules of
points \ref{itm:nsfintro} and \ref{itm:nshwintro}.
However, as far as the author is aware, this is the first 
rigorous classification of simple relaxed highest weight \(\ZZ_2\)-graded
\(\MinMod{u}{v}\) modules for general \(u,v\).

The module classification presented above could now be combined with a number
of recent developments.
For example, recent results on
character formulae for relaxed highest weight modules \cite{kawrel18} applied
to the classification could be
used to repeat the Verlinde formula calculations of
\cite{SnaOsp17} using the standard module formalism \cite{RidSL208,CreMod12,CreWZW13}. 
However, an alternative approach to computing fusion rules and Verlinde formulae, realising the \(\osp\) minimal models as an
extension of the tensor product of the Virasoro and \(\slt\) minimal modules
is currently in preparation \cite{KanOsp18}. 
A further application could be the classification of simple Whittaker
\(\aosp{}\) modules in analogy to the \(\slt\) classification
\cite{Adawhit18}.
\smallskip

The minimal model \(\osp\) \vosa{} \(\MinMod{u}{v}\) is a simple quotient of
the universal \(\osp\) \vosa{} by an ideal generated by a singular vector. 
A natural strategy for classifying simple modules over \(\MinMod{u}{v}\) is,
therefore, to identify modules over the universal \vosa{} on which the ideal
acts trivially. This annihilating ideals approach to module classification was
first applied by Feigin, Nakanishi and Ooguri \cite{FeiAnn92} to certain
Virasoro minimal models. In the context of \(\osp\)
this strategy was first used by Kac and Wang
\cite{KacVer94}, as an application of their generalisation of Zhu algebras
\cite{ZhuMod96} to \vosa{s}, to classify simple \ns{} modules when \(v=1\).
These kinds of calculations require explicit formulae for the singular vector
generating the ideal. Unfortunately, for \(v>1\), the only known general formulae are formal
expressions involving non-integer powers of affine \(\osp\) generators
\cite{Fanlsa93,IoEn12} resembling those of Malikov, Feigin and Fuks 
\cite{MalSing84} for
affine Lie algebra \svs{}. Converting these formulae into expressions with
non-negative integer powers is prohibitively laborious, so
instead of using these formulae, the proof of the above
classification theorem presented in this paper exploits a deep connection between free field
realisations and symmetric functions to derive expressions which are more
tractable. Finding such tractable singular vector formulae in terms of symmetric
functions has a long history starting with the Virasoro algebra
\cite{WakSchu86,MimJa95} with later generalisations to other algebras, such as, 
affine \(\slt\) \cite{KatMis92,RidRel15}, the \(N=1\) super Virasoro algebra
\cite{DesSup01,YanUgl15,OPDS16} and the \(W_N\) algebras
\cite{AwaCS95,RidWN18}.
The presence of fermions (or odd fields) in \vosa{s} necessitates the
consideration of skew symmetric functions as well as symmetric
functions. Fortunately, this skew symmetry can be compensated for, so that,
symmetric function methods can still be applied after
considering certain ideals, first considered by Feigin, Jimbo, Miwa and Mukhin
\cite{Feidif02}, spanned by Jack symmetric polynomials whose
parameter is negative rational.

The methods presented in this article were developed for and 
applied to the classification of simple modules of the triplet algebras
\cite{TsuExt13}, the Virasoro minimal models \cite{RidJac14}, the affine \(\slt\)
minimal models \cite{RidRel15} and the \(N=1\) superconformal minimal models
\cite{BloSVir16}. In particular, \cite{RidJac14,RidRel15,BloSVir16} form a
series aimed at systematising the classification of modules over
\vosa{s} which are non-trivial simple quotients of universal \vosa{s}. 
A convenient property of these methods is that they
work in essentially the same way not only in both the \ns{} and Ramond sectors
of a given algebra, but also across all examples of \vosa{s} considered.
Since all the algebras considered so far are rank 1, it will be
interesting to see if these methods can be generalised to higher ranks as was
recently done for the singular vector formulae of the \(W_N\)-algebras \cite{RidWN18}.
\medskip

This article is organised as follows. In \cref{sec:osp}, we give a brief
overview of \(\osp\), its affinisation \(\aosp{}\) and 
modules over both of these algebras. The section concludes with the construction
of the universal \(\osp\) \vosa{} \(\vosp{k}\) at level \(k\) and the minimal model
\vosa{} \(\MinMod{u}{v}\) as a simple quotient of \(\vosp{k}\) at certain
rational levels, termed admissible. In \cref{sec:zhu}, 
we state explicit presentations of the untwisted and twisted Zhu algebras of
\(\MinMod{u}{v}\), that is, the Zhu algebras in the \ns{} and Ramond
sectors. The proof of these presentations is postponed to 
\cref{sec:0modecalc}. These presentations are then used to prove the main
result of the article, \cref{thm:moduleclassification}, that is, the classification of
\(\ZZ_2\)-graded simple relaxed highest weight modules over \(\MinMod{u}{v}\) and the rationality of
\(\MinMod{u}{v}\) in category \(\catO\). 
In \cref{sec:ffrandscr}, we define the free field algebras and screening operators
necessary for the free field realisation of \(\vosp{k}\), and derive
identities for correlation functions in preparation for proving the
presentations of the untwisted and twisted Zhu algebras. In
\cref{sec:0modecalc}, the presentations of the Zhu algebras are proved by
evaluating the action of the zero modes of singular vectors of
\(\vosp{k_{u,v}}\) on candidate relaxed highest weight vectors. This action of
the zero modes depends polynomially on free field weights and, with the methods
used here, can only be evaluated indirectly
by determining certain zeros in the free field weights and showing that the zeros
found saturate certain bounds. This necessitates splitting the calculation
into two cases depending on whether \(k<-1\) or \(k>-1\) in order to assure
that these bounds are indeed saturated.

\subsection*{Acknowledgements}

The author would like to thank the following people for interesting
discussions: Andreas Aaserud on polynomial rings, Tomoyuki Arakawa
on universal enveloping algebras, Pierre Mathieu and Jorgen Rasmussen on fractional powers of screening
operators and bosonising \(\beta\gamma\) systems, and John Snadden on
super-Casimir operators. 
Additionally, the author would like to thank David Ridout for interesting discussions
on too many topics to list, and for the careful reading of a previous version of this
article and giving helpful feedback. The author's research is supported by the Australian Research 
Council Discovery Early Career Researcher Award DE140101825 and the Discovery
Project DP160101520.

\section{The \(\osp\) \vosa{}}
\label{sec:osp}

In this section we settle notation regarding the Lie superalgebra \(\osp\),
its affinisation and its associated \vosa{}, and recall known results. 
Since this article studies Lie superalgebras and \vosa{s}, all vector
spaces will be assumed to be complex vector superspaces, that is, graded by
\(\ZZ_2\). For any vector superspace \(V\), we denote the subspace of even vectors by
\(\even{V}\), the subspace of odd vectors by \(\odd{V}\) and the parity
reversal of \(V\) by \(\rpar{V}\). When considering vector spaces without any
obvious superspace structure, then the grading will be assumed to be trivial
unless stated otherwise, that is, the entire vector space will be assumed to
be even (examples include \(\AKMA{sl}{2}\) or the \(\beta\gamma\) ghost \va{} considered
below). We refer readers unfamiliar with Lie superalgebras to \cite{WangLSA12}
for an exhaustive discussion of the subject and to \cite[Section
2.1]{SnaOsp17} for a summary of \(\osp\) beyond that given below.

\subsection{The finite dimensional simple Lie superalgebra \(\osp\)}
\label{sec:finiteosp}

The orthosymplectic Lie superalgebra \(\osp\) is the Lie superalgebra
preserving the standard supersymmetric bilinear form of \(\CC^{1|2}\), where
supersymmetric means that the bilinear form is symmetric on the even
subspace and skew-symmetric on the odd subspace. This Lie superalgebra is 5
dimensional and we choose the following basis:
\(\{e,x,h,y,f\}\), where \(\spn{e,h,f}=\even{\osp}\) and \(\spn{x,y}=\odd{\osp}\). The
non-vanishing (anti-) commutation relations in this basis are
\begin{align}
  \label{eq:comrels}
  \comm{h}{e}&=2e,\quad \comm{e}{f}=h,\quad\comm{h}{f}=-2f,\nonumber\\
  \comm{e}{y}&=-x,\quad\comm{h}{x}=x,\quad\comm{h}{y}=-y,\quad
  \comm{f}{x}=-y,\nonumber\\
  \acomm{x}{x}&=2e,\quad\acomm{x}{y}=h,\quad\acomm{y}{y}=-2f.
\end{align}
In their standard normalisation, the non-vanishing pairings of the invariant
bilinear form on \(\osp\) are
\begin{align}
  \label{eq:killing}
  \killing{h}{h}=2,\quad\killing{e}{f}=\killing{f}{e}=1,\quad\killing{x}{y}=-\killing{y}{x}=2.
\end{align}

The (anti-)commutation relations \eqref{eq:comrels} imply that the even
subspace of \(\osp\) is isomorphic to \(\slt\) and that \(h\) spans a choice of Cartan subalgebra
for both \(\osp\) and \(\slt\). We denote by \(\alpha\) the simple root corresponding to the root
vector \(x\), whose length with respect to the norm induced from \eqref{eq:killing}
is \(\norm{\alpha}^2=\sfrac{1}{2}\). Note that, for the copy of \(\slt\) sitting
inside \(\osp\), the simple root is \(2\alpha\) and thus \(\alpha\) is the
fundamental weight of \(\slt\).

\begin{defn}
  A \emph{weight module} over a Lie superalgebra is a module over that
  algebra which decomposes into a direct sum of simultaneous eigenspaces of the
  Cartan subalgebra. The \emph{weight support} of a weight module is the set
  of all weights for which the corresponding weight space is non-trivial.
\end{defn}

In this article we shall focus exclusively on weight modules and hence the
subalgebras of the universal enveloping algebras of \(\osp\) and \(\slt\)
which preserve weight spaces will prove vital. These subalgebras are
just the centralisers of the Cartan subalgebra.
\begin{lem}\label{thm:centralgebras}
  Let \(\alg{g}\) be a Lie superalgebra with choice of Cartan subalgebra
  \(\alg{h}\) and let
  \begin{equation}
    \Cent{\alg{h}}{\alg{g}}=\set*{w\in \UEA{g}\st [h,w]=0,\ \forall h\in\alg{h}}
  \end{equation}
  be the centraliser of \(\alg{h}\) in the universal enveloping algebra of \(\alg{g}\).
  \begin{enumerate}
  \item\label{itm:cent1} As an associative algebra
    \begin{equation}
      \Cent{\alg{h}}{\osp}\cong\CC[h,\Sigma],
    \end{equation}
    where \(\Sigma=xy-yx+1/2\) is the \emph{super-Casimir} of \(\osp\).
  \item\label{itm:cent2} As an associative algebra
    \begin{equation}
      \Cent{\alg{h}}{\slt}\cong\CC[h,Q],
    \end{equation}
    where \(Q=h^2/2+ef+fe\) is the \emph{quadratic Casimir} of \(\slt\).
  \end{enumerate}
\end{lem}
Part \ref{itm:cent2} of the above lemma is well known and part \ref{itm:cent1} 
is an immediate consequence
of Pinczon's work \cite{PinsUEA90} on the universal enveloping algebra of
\(\osp\) and Le\'sniewski's discovery \cite{LessCas95} of the super-Casimir of
\(\osp\). While part \ref{itm:cent1} is surely also known, the author was not able
to find a source and so a proof has been given for completeness.
\begin{proof}
  It was shown in \cite{PinsUEA90} that the centre of the universal enveloping
  algebra is isomorphic to \(\CC[Q^{\osp}]\), where
  \begin{equation}
    Q^{\osp}=\frac{1}{2}h^2+ef+fe-\frac{1}{2}xy+\frac{1}{2}yx=\frac{1}{2}\Sigma^2-\frac{1}{8}
  \end{equation}
  is the quadratic Casimir of \(\osp\). Further,
  it was shown in \cite{LessCas95} that the super-Casimir \(\Sigma\) commutes
  with all even elements of \(\osp\) and anticommutes with all odd ones.

  To prove the proposition,
  fix any \PBW{} ordering of the bases of \(\osp\) and \(\slt\) and note that
  since \(x^2=e\) and \(y^2=-f\), the exponents of \(x\) and \(y\) in an
  element of the \PBW{} basis of \(\osp\) are at most 1. Every
  element of the \PBW{} basis is an eigenvector of the adjoint action of the
  Cartan subalgebra. Hence, the centraliser algebras \(\Cent{\alg{h}}{\osp}\)
  and \(\Cent{\alg{h}}{\slt}\) are spanned by \PBW{} basis vectors whose
  eigenvalues are 0. Within such basis vectors, every occurrence of the \(f\)
  generator must be countered by an \(e\) and vice versa for both \(\osp\) and
  \(\slt\). Similarly, for \(\osp\), every occurrence of \(y\)
  must be countered by an \(x\) and vice versa.
  Using the commutation relations, the generators in these basis vectors can
  be reordered, so that that basis vectors become sums of monomials in \(h,
  fe, yx\) for \(\osp\) and \(h, fe\) for \(\slt\). It then follows by direct
  computation that
  \begin{equation}
    yx=\frac{1}{2}\brac*{h-\Sigma+\frac{1}{2}},\quad
    fe=-\frac{1}{4}\brac*{h+\Sigma+\frac{3}{2}}\brac*{h-\Sigma+\frac{1}{2}},
  \end{equation}
  for \(\osp\) and for \(\slt\) that
  \begin{equation}
    fe=\frac{1}{2}\brac*{Q-\frac{1}{2}h^2-h}.
  \end{equation}
  The algebraic independence of \(h\) and \(\Sigma\) then follow from the
  algebraic independence of \(h\) and \(Q\), since any algebraic relation of
  \(h\) and \(\Sigma\) would also imply a relation for \(h\) and \(Q\).
\end{proof}

As mentioned in the proof above, the quadratic Casimir operator of \(\osp\) has the nice property of
freely generating the centre of the universal enveloping algebra of
\(\osp\), however, the super-Casimir operator \(\Sigma\) shall prove to be far more
convenient in the considerations that follow.
For example, isomorphism classes of simple \(\osp\) weight modules can be
characterised in terms of their weight support and the action of the
super-Casimir operator \(\Sigma\).
Similarly, isomorphism classes of simple \(\slt\) weight modules can be
characterised in terms of their weight support and the action of the
Casimir operator \(Q\).

\begin{thm}[Block \cite{Blosl279,BloIrr79}]
  \label{thm:finitesl}
  Any simple \(\slt\) weight module is isomorphic to one of the following:
  \begin{enumerate}
  \item The simple \((\lambda+1)\) dimensional module
    \(\fRf{\lambda}\) which is both highest and lowest weight, where
    \(\lambda\in\ZZ_{\ge0}\). The weights of the highest and lowest weight vectors
    is \(\lambda \alpha\) and \(-\lambda\alpha\), respectively, and the eigenvalue
    of \(Q\) is \(\frac{\lambda}{2}\brac*{\lambda+2}\).
  \item The simple infinite dimensional highest weight module \(\fRh{\lambda}\),
    where \(\lambda\in \CC\setminus \ZZ_{\ge0}\). This module is generated
    by a highest weight vector \(v_\lambda\) of weight \(\lambda\alpha\),
    which therefore satisfies \(e v_\lambda=0\) and
    \(hv_\lambda=\lambda v_\lambda\). A basis
    is given by the weight vectors
    \(\set*{f^n v_\lambda\st n\in\ZZ_{\ge0}}\) and the eigenvalue of \(Q\) is \(\frac{\lambda}{2}\brac*{\lambda+2}\).
    This is an example of a Verma module.
  \item The simple infinite dimensional lowest weight module \(\fRl{\lambda}\),
    where \(\lambda\in \CC\setminus \ZZ_{\le0}\). This module is generated
    by a lowest weight vector \(v_\lambda\) of weight \(\lambda\alpha\), which
    therefore satisfies \(f v_\lambda=0\) and
    \(hv_\lambda=\lambda v_\lambda\). A basis is given by the
    weight vectors \(\set*{e^n v_\lambda\st
      n\in\ZZ_{\ge0}}\) and the eigenvalue of \(Q\) is \(\frac{\lambda}{2}\brac*{\lambda-2}\). This is an example of a Verma module for which \(f\)
    (as opposed to \(e\)) has been chosen as the simple root vector.
  \item The simple infinite dimensional dense weight module \(\fRr{\lambda,q}\),
    with \(\sqbrac*{\lambda}\in\CC/2\ZZ\) and \(q\in\CC\)
    satisfying \(q\neq \mu(\mu+2)\) for all \(\mu\in
    \sqbrac*{\lambda}\). This module is generated by a weight vector
    \(v_{\lambda,q}\) that is neither highest nor lowest weight, satisfying
    \(hv_{\lambda,q}=\lambda v_{\lambda,q}\) and \(Q v_{\lambda,q}=q
    v_{\lambda,q}\), and has a basis of weight vectors  
    \(\set*{v_{\lambda,q}, e^n v_{\lambda,q}, f^n v_{\lambda,q}\st
      n\in\ZZ_{\ge1}}\).
  \end{enumerate}
\end{thm}

\begin{thm}
  \label{thm:finiteosp}
  Any simple \(\ZZ_2\)-graded \(\osp\) weight module is isomorphic to one of the following or its parity reversal:
  \begin{enumerate}
  \item The simple \((2\lambda+1)\) dimensional module \(\fNSf{\lambda}\)
    which is both highest and lowest weight, where
    \(\lambda\in\ZZ_{\ge0}\).
    The highest and lowest weight vectors both have even parity and their
    respective weights are \(\lambda \alpha\) and \(-\lambda\alpha\). The
    eigenvalue of \(\Sigma\) is \(\frac{\lambda}{2}\brac*{\lambda+1}\) on even
    vectors and \(-\frac{\lambda}{2}\brac*{\lambda+1}\) on odd vectors.
  \item The simple infinite dimensional highest weight module \(\fNSh{\lambda}\),
    where \(\lambda\in \CC\setminus \ZZ_{\ge0}\).
    This module is generated by an even parity highest weight vector
    \(v_\lambda\) of weight \(\lambda\alpha\), which therefore satisfies \(x v_\lambda=0\) and
    \(hv_\lambda=\lambda v_\lambda\). A basis of weight vectors of this module
    is given by
    \(\set*{y^n v_\lambda \st n\in\ZZ_{\ge0}}\).
    The
    eigenvalue of \(\Sigma\) is \(\frac{\lambda}{2}\brac*{\lambda+1}\) on even
    vectors and \(-\frac{\lambda}{2}\brac*{\lambda+1}\) on odd vectors.
    This is an example of a Verma module.
  \item The simple infinite dimensional lowest weight module \(\fNSl{\lambda}\),
    where \(\lambda\in \CC\setminus \ZZ_{\le0}\).
    This module is generated
    by an even parity lowest weight vector \(v_\lambda\), which therefore
    satisfies \(y v_\lambda=0\) and \(hv_\lambda=\lambda v_\lambda\). A 
    basis of this module is given by the weight vectors \(\set*{x^n v_\lambda \st
      n\in\ZZ_{\ge0}}\).
    The
    eigenvalue of \(\Sigma\) is \(\frac{\lambda}{2}\brac*{\lambda-1}\) on even
    vectors and \(-\frac{\lambda}{2}\brac*{\lambda-1}\) on odd vectors.
    This is an example of a Verma module for which \(y\) has
    been chosen as the simple root vector.
  \item The simple infinite dimensional dense weight module \(\fNSr{\sqbrac*{\lambda},s}\),
    with \(\sqbrac*{\lambda}\in\CC/2\ZZ\) and \(s\in\CC\) satisfying \(s^2\neq
    \brac*{\mu+\frac{1}{2}}^2\)
    for any \(\mu\in \sqbrac*{\lambda}\cup\sqbrac*{\lambda+1}\).
    This module is generated by an even parity weight vector
    \(v_{\lambda,s}\) that is neither highest nor lowest weight, satisfying
    \(hv_{\lambda,s}=\lambda v_{\lambda,s}\) and \(\Sigma v_{\lambda,s}=s
    v_{\lambda,s}\), and has a basis of weight vectors  
    \(\set*{v_{\lambda,s}, x^n v_{\lambda,s}, y^n v_{\lambda,s}\st
      n\in\ZZ_{\ge1}}\).
  \end{enumerate}
  Note that the parity reversal of \(\fNSr{\sqbrac*{\lambda},s}\) is
  \(\rpar{\fNSr{\sqbrac*{\lambda},s}}\cong\fNSr{\sqbrac*{\lambda+1},-s}\). 
\end{thm}

See \cite{MazLec10} for a proof of \cref{thm:finitesl} and a comprehensive
discussion of \(\slt\) module theory.
The proof of the classification of simple \(\osp\) weight modules is similar to that of \(\slt\) and
was given in \cite{SnaOsp17}.

\subsection{The affinisation of \(\osp\)}
\label{sec:affineosp}

The Lie superalgebra \(\osp\) can be affinised by constructing its
loop algebra and centrally extending it.  It is common to add a degree
operator to ensure that the induced invariant bilinear form on the newly
constructed affine Lie superalgebra is non-degenerate. Here,
however, we shall always identify the
action of such a degree operator on modules with that of the negative of the Virasoro
\(L_0\) operator of the Sugawara construction. We therefore omit the degree
operator from the affinisation of \(\osp\). We consider two
affinisations \(\NSosp{}\) and \(\Rosp{}\),
identifiable by \emph{spectral flow isomorphisms}, which we shall respectively refer to as
the \ns{} and Ramond affinisation. We shall suppress superscripts and just
write \(\aosp{}\) whenever the distinction between the two
affinisations is not important. Though these two affinisations are
isomorphic, their conformal gradings differ and thus so do the module categories
which one considers for each algebra.

As vector spaces the \ns{} and Ramond affinisations of \(\osp\) are given by
\begin{align}
  \NSosp{}&=\osp\otimes\CC[t,t^{-1}]\oplus \CC K \nonumber\\
  \Rosp{}&=
  \spn{e,h,f}\otimes\CC[t,t^{-1}]\oplus
  \spn{x,y}\otimes\CC[t,t^{-1}]t^{\sfrac{1}{2}}\oplus
  \CC K.
\end{align}
The (anti-) commutation relations of these two algebras are constructed from those
of \(\osp\) and are given by
\begin{align}
  \comm{a_m}{b_n}&=\comm{a}{b}\otimes
  t^{m+n}+m\killing{a}{b}\delta_{m+n,0}\,, \quad
  a_m=a\otimes t^m,\ b_n=b\otimes t^n,
\end{align}
where \(a,b\in \osp\) are homogeneous vectors of definite parity, \(\comm{\ }{\ }\) denotes either the commutator
or the anti-commutator depending on the parities of \(a\) and \(b\) and \(m,n\)
lie in \(\ZZ\) or \(\ZZ+\sfrac{1}{2}\) depending on the parities of \(a,b\)
and the algebra considered. Finally, \(K\) is even and central. 
Here and hereafter we shall also suppress the
\(t\) variables and write \(a_m\) instead of \(a\otimes t^m\). 

The \ns{} and Ramond affinisations of \(\osp\) admit triangular decompositions
as well as \emph{relaxed}
triangular decompositions. 
For the \ns{} affinisation we choose the decompositions
\begin{align} 
  \NSosp{}&=\NSosp{-}\oplus \cart\oplus \NSosp{+},&
  \NSosp{}&=\NSosp{<}\oplus \NSosp{0}\oplus \NSosp{>},
  \nonumber\\
  \NSosp{-}&=\spn{e_{-n},x_{-n},h_{-n},y_{1-n},f_{1-n}\st n\in\ZZ^+},&
  \NSosp{<}&= \spn{e_{-n},x_{-n},h_{-n},y_{-n},f_{-n}\st n\in\ZZ^+}
  \nonumber\\
  \NSosp{+}&=\spn{e_{n-1},x_{n-1},h_{n},y_{n},f_{n}\st n\in\ZZ^+},&
  \NSosp{>}&= \spn{e_{n},x_{n},h_{n},y_{n},f_{n}\st n\in\ZZ^+},
  \nonumber\\
  \cart&= \spn{h_0,K},& 
  \NSosp{0}&= \osp \oplus\CC K,
  \label{eq:nstri}
\end{align}
where the decomposition on the left is the usual triangular decomposition into
negative roots, a Cartan subalgebra and positive roots, and the decomposition
on the right is a relaxed triangular decomposition.
For the Ramond affinisation we choose the 
 decompositions
\begin{align}
  \Rosp{}&=\Rosp{-}\oplus \cart\oplus \Rosp{+},&
  \Rosp{}&=\Rosp{<}\oplus \Rosp{0}\oplus \Rosp{>},
  \nonumber\\
  \Rosp{-}&=\spn{e_{-n},x_{\frac{1}{2}-n},h_{-n},y_{\frac{1}{2}-n},f_{1-n}\st n\in\ZZ^+},&
  \Rosp{<}&=\spn{e_{-n},x_{\frac{1}{2}-n},h_{-n},y_{\frac{1}{2}-n},f_{-n}\st n\in\ZZ^+}
  \nonumber\\
  \Rosp{+}&=\spn{e_{n-1},x_{n-\frac{1}{2}},h_{n},y_{n-\frac{1}{2}},f_{n}\st n\in\ZZ^+},&
  \Rosp{>}&=\spn{e_{n},x_{n-\frac{1}{2}},h_{n},y_{n-\frac{1}{2}},f_{n}\st n\in\ZZ^+}
  \nonumber\\
  \cart&=\spn{h_0,K},&
  \Rosp{0}&= \slt \oplus\CC K,
  \label{eq:rtri}
\end{align}
where the decomposition on the left is the usual triangular decomposition 
and the decomposition on the right is a relaxed triangular decomposition.
Finally, the \emph{parabolic subalgebras} are denoted by
\begin{equation}
  \aosp{\epsilon}^{\ge}=\aosp{\epsilon}^{0}\oplus \aosp{\epsilon}^{>}.
\end{equation}

\begin{defn}
  Let \(\Mod{M}\) be an \(\aosp{\epsilon}\) module. A \emph{relaxed highest weight
    vector} \(m\in\Mod{M}\) is a simultaneous eigenvector of the Cartan
  subalgebra \(\cart\) which is annihilated by
  \(\aosp{\epsilon}^{>}\). Further, \(\Mod{M}\) is said to be a \emph{relaxed highest
    weight module} 
  if it is generated by a relaxed highest weight vector and it said to be
  a \emph{relaxed Verma module} if it is isomorphic to
  \begin{equation}
    \Ind{\aosp{\epsilon}}{\aosp{\epsilon}^{\ge}}{\finite{\Mod{M}}}=\UEA{\aosp{\epsilon}}\otimes_{\UEA{\aosp{\epsilon}^{\ge}}}\finite{\Mod{M}},
  \end{equation}
  where \(\UEA{\aosp{\epsilon}}\) denotes the universal enveloping algebra and
  \(\finite{\Mod{M}}\) is some simple \(\aosp{\epsilon}^{0}\) weight module upon which \(\aosp{\epsilon}^{>}\)
  acts trivially.
\end{defn}
Since in the \ns{} and Ramond sectors, respectively, we have
\begin{align}
  \NSosp{0}&\cong\osp\oplus \CC K,\nonumber\\ 
  \Rosp{0}&\cong\slt\oplus \CC K, 
\end{align}
\ns{} and Ramond Verma modules are respectively induced from \(\osp\) and
\(\slt\) modules on which the central element
\(K\) to act as \(k\cdot\id, k\in\CC\). 
For Ramond Verma modules, the
\(\slt\) modules they are induced from will be assigned even parity.

The above triangular decompositions suggest certain natural module categories
within which to consider weight modules.

\begin{defn}
  Category \(\catR\) is the category of \(\aosp{\epsilon}\) modules \(\Mod{M}\) which satisfy the
  following axioms:
  \begin{itemize}
  \item \(\Mod{M}\) is \(\ZZ_2\)-graded.
  \item \(\Mod{M}\) is finitely generated.
  \item \(\Mod{M}\) is a weight module (the action of the Cartan subalgebra is semisimple).
  \item The action of \(\aosp{\epsilon}^{<}\) is locally nilpotent: For any
    \(m\in\Mod{M}\), the space \(\UEA{\aosp{\epsilon}^{>}}\cdot m\) is finite dimensional.
  \end{itemize}
  The morphisms are the \(\aosp{\epsilon}\) module homomorphism between these
  modules.
  Category \(\catO\) is the full subcategory of category \(\catR\) whose
  modules satisfy the additional property
  \begin{itemize}
  \item The action of \(\aosp{\epsilon}^+\) is locally nilpotent.
  \end{itemize}
\end{defn}
All relaxed highest weight modules belong to category \(\catR\) and 
highest weight modules belong to category \(\catO\) and thus also
\(\catR\). Additionally, since \(\aosp{\epsilon}\) has finite dimensional 
root spaces (if one grades by both \(\osp\) and conformal weight), finite generation
implies that each module in \(\catR\) has finite dimensional
weight spaces (again, if one grades by both \(\osp\) and conformal weight). 
Finally, the axioms also
imply that every module in category \(\catR\) contains a relaxed highest
weight vector, thus the simple objects of category \(\catR\) are simple relaxed highest weight
modules.

\begin{defn}
  We denote the unique simple quotients by maximal proper submodules of the
  inductions of the modules of Theorems 
  \cref{thm:finitesl,thm:finiteosp} 
  by the same symbols with the overline removed. To simplify notation, we
  suppress the algebras appearing in the super- and subscript of \(\Ind{\aosp{\epsilon}}{\aosp{\epsilon}^{\ge}}{}\).
  \begin{itemize}
  \item In the \ns{} sector:
    \begin{enumerate}
    \item \(\NSf{\lambda}\), where \(\lambda\in\ZZ_{\ge0}\),
      is the simple quotient of
      \(\Ind{}{}{\fNSf{\lambda}}\). 
    \item \(\NSh{\lambda}\), where \(\lambda\in \CC\setminus \ZZ_{\ge0}\),
      is the simple quotient of
      \(\Ind{}{}{\fNSh{\lambda}}\). 
    \item \(\NSl{\lambda}\), where \(\lambda\in \CC\setminus \ZZ_{\le0}\),
      is the simple quotient of
      \(\Ind{}{}{\fNSl{\lambda}}\).
    \item \(\NSr{\sqbrac*{\lambda},s}\), with
      \(\sqbrac*{\lambda}\in\CC/2\ZZ\), \(s\in\CC\) satisfying
      \(s^2\neq \brac*{\mu+\frac{1}{2}}^2\)
      for any \(\mu\in \sqbrac*{\lambda}\cup\sqbrac*{\lambda+1}\),
      is the simple quotient of
      \(\Ind{}{}{\fNSr{\sqbrac*{\lambda},s}}\).
    \end{enumerate}
  \item In the Ramond sector:
    \begin{enumerate}[resume]
    \item \(\Rf{\lambda}\), where \(\lambda\in\ZZ_{\ge0}\), is the simple
      quotient of \(\Ind{}{}{\fRf{\lambda}}\). 
    \item \(\Rh{\lambda}\), where \(\lambda\in \CC\setminus \ZZ_{\ge0}\),
      is the simple quotient of
      \(\Ind{}{}{\fRh{\lambda}}\). 
    \item \(\Rl{\lambda}\), where \(\lambda\in \CC\setminus \ZZ_{\le0}\),
      is the simple quotient of
      \(\Ind{}{}{\fRl{\lambda}}\). 
    \item \(\Rr{\sqbrac*{\lambda},q}\), with \(\sqbrac*{\lambda}\in\CC/2\ZZ\),
      \(q\in\CC\) satisfying \(q\neq
      \mu(\mu+2)\) for any \(\mu\in \sqbrac*{\lambda}\),
      is the simple quotient of
      \(\Ind{}{}{\fRr{\sqbrac*{\lambda},q}}\). 
    \end{enumerate}
  \end{itemize}
\end{defn}

Finally, before considering the \vosa{s} that can be constructed from
\(\aosp{}\), we introduce a family of isomorphisms \(\sfaut^\ell,\ \ell\in\frac{1}{2}\ZZ\), called \emph{spectral flow}.
These relate \ns{} and Ramond affinisations of \(\osp\). 
The images of the basis vectors of \(\aosp{}\) under
\(\sfaut^\ell\) are
\begin{align}\label{eq:sflow}
  \sfaut^\ell\brac*{e_n}&=e_{n-2\ell},\quad
  \sfaut^\ell\brac*{x_n}=x_{n-\ell},\quad
  \sfaut^\ell\brac*{h_n}=h_{n}-2\ell\delta_{n,0} K,\nonumber\\
  \sfaut^\ell\brac*{K}&=K,\quad
  \sfaut^\ell\brac*{y_n}=y_{n+\ell},\quad
  \sfaut^\ell\brac*{f_n}=f_{n+2\ell}.
\end{align}
Clearly, the spectral flow isomorphism \(\sfaut^\ell\)
is an automorphism relating the \ns{} and Ramond affinisations to themselves
if \(\ell\in\ZZ\), otherwise it maps from the \ns{} to the Ramond affinisation
and vice versa.
Algebra isomorphisms such as the spectral flow isomorphism \(\sfaut^\ell\)
lead one naturally to consider modules whose algebra action has been twisted
by such an automorphism.
\begin{defn}\label{def:isotwists}
  Let \(A\) and \(B\) be Lie superalgebras, \(\phi:A\to B\) an isomorphism and
  let \(M\) be a module over \(B\), then we define, \(\phi^{-1}M\), the twist of
  \(M\) by \(\phi\), to be the \(A\)-module which as a vector superspace is
  just \(M\) but on which the action of \(A\) is defined to be
  \begin{equation}
    a\cdot_\phi m=\phi(a) m,\quad \forall a\in A,\ \forall m\in M.
  \end{equation}
\end{defn}

Because the spectral flow isomorphisms do not preserve the triangular
decompositions \eqref{eq:nstri} or \eqref{eq:rtri}, the twist of a module in
category \(\catR\) need not lie category \(\catR\). 
However, the algebra isomorphism \(\zeta\) defined by
\begin{align}\label{eq:catotwist}
  \zeta(e_n)&= -f_{n+1},\quad \zeta(x_n)= -y_{n+\frac{1}{2}},\quad
  \zeta(h_{n}) -h_n +2\delta_{n,0} K,\nonumber\\
  \zeta(K)&= K,\quad
  \zeta(y_n)= x_{n-\frac{1}{2}},\quad
  \zeta(f_n)= -e_{n-1}.
\end{align}
preserves the standard triangular decompositions of \(\aosp{}\), 
that is, \(\zeta\brac*{\aosp{\epsilon}^\pm}=\aosp{\frac{1}{2}-\epsilon}^\pm\)
and \(\zeta\brac*{\cart}=\cart\).
Twisting by \(\zeta\) therefore defines
functors mapping between the \ns{} and Ramond sectors of category
\(\catO\). In particular, the twist of a highest weight \ns{} module at level \(k\) of
highest weight \(\lambda\alpha\) is a highest weight Ramond module at level \(k\) of highest
weight \((k-\lambda)\alpha\).

\begin{rmk}
  Our choice of representation category is informed by physics
  considerations coming from conformal field theory as well as technical
  considerations coming from Zhu's algebra, an associative algebra to be
  discussed below. A necessary condition for the consistency of a conformal field
  theory is that the representation category be closed under fusion and
  conjugation. Additionally, one generally requires characters to be well defined
  and to behave well under modular transformations so that modular invariants
  can be identified and fusion rules at the level of the Grothendieck group can
  be computed from Verlinde like formulae.
  Neither closure under fusion nor conjugation appear to be satisfied for category \(\catO\) when
  considering non-integral admissible levels, see \cite{CreWZW13}
  for the case of \(\slt\), \cite{RidBos14} for the \(\beta\gamma\) ghosts or
  \cite{SnaOsp17} for \(\osp\) at level 
  \(k=-5/4\). While the generalisation to category \(\catR\) yields closure under
  conjugation, the Verlinde like formulae in \cite{SnaOsp17,CreWZW13,RidBos14},
  computed using the standard module formalism, indicate that
  category \(\catR\) does not close under fusion and indeed also that the action
  of the modular group does not close on the span of category \(\catR\)
  characters.
  
  The reason for focusing on category \(\catR\) is that modules in category
  \(\catR\) will always contain relaxed highest weight vectors and Zhu's algebra
  can be interpreted as the associative algebra of zero modes of fields acting on
  relaxed highest weight vectors. As Zhu algebra methods are blind to
  modules not containing relaxed highest weight vectors, category \(\catR\) is the
  largest category in which Zhu algebras can be used for module classification.
  
  The success of the standard module formalism as a conjectured generalisation
  of the Verlinde formula for rational theories suggests that the natural module
  category to work with is smallest category containing \(\catR\), all twists of
  \(\catR\) by the spectral flow isomorphisms \(\sfaut^\ell\) and all extensions
  of modules in these categories.
\end{rmk}

\subsection{The affine \(\osp\) \vosa{}}
\label{sec:ospvosa}

The affine \(\osp\) \vosa{s} are constructed by inducing from the trivial
\(\osp\) module. Let \(\CC_k, k\in\CC\setminus\set{-\frac{3}{2}}\) be the 1
dimensional \(\NSosp{0}\) module on which \(\osp\) 
acts trivially
and on which the central element \(K\) acts as \(k\cdot\id\). As an
\(\NSosp{}\) module, the \emph{universal \(\osp\) \vosa{}} \(\vosp{k}\) is
given by the relaxed Verma module induced from \(\CC_k\), that is,
\begin{equation}
  \vosp{k}=\Ind{\NSosp{}}{\NSosp{\ge}}{\CC_k}.
\end{equation}
As a \vosa{} \(\vosp{k}\) is freely generated under normal ordering and taking
derivatives by 5 fields,
labelled by the basis vectors of \(\osp\), subject to the operator product relations
\begin{equation}
  a(z) b(w)\sim \frac{\killing{a}{b}}{(z-w)^2}+\frac{\comm{a}{b}(w)}{z-w},\quad a,b\in\osp.
\end{equation}
The Virasoro field, whose Laurent expansion coefficients generate the Virasoro
algebra, is given by the standard Sugawara construction.
\begin{equation}
  \label{eq:emt}
  T(z)=\frac{1}{2k+3}\brac*{\frac{1}{2}\normord{h(z)^2}+\normord{e(z)f(z)}
    +\normord{f(z)e(z)}-\frac{1}{2}\normord{x(z)y(z)}+\frac{1}{2}\normord{y(z)x(z)}}.
\end{equation}

We follow the definition given by Frenkel and Ben-Zvi \cite[Chapter
5.1]{FreVer01}, augmented to include a \(\ZZ_2\) grading by parity, for
modules over a \vosa{}. Note that this implies that a module over \(\vosp{k}\)
is just a smooth \(\ZZ_2\)-graded level \(k\) \(\aosp{}\)-module and vice versa.

\begin{cor}
  All modules in the \ns{} sector of \(\catR\) are modules over \(\vosp{k}\),
  while all modules in the Ramond sector of \(\catR\) are modules twisted by
  the parity automorphism over \(\vosp{k}\), that is, modules
  for which the action of odd fields has half integer monodromy about 0.
\end{cor}

We call a non-zero vector of a \(\vosp{k}\) module \emph{singular} if it is a
simultaneous eigenvector of \(h_0\) and \(L_0\), and is annihilated by all
positive root vectors, that is, by
\(\NSosp{+}\). The highest weight vectors which generate Verma modules
are prominent examples of singular vectors.

\begin{prop}[Gorelik and Kac \cite{GorSim07}]
  \label{thm:admlevels}
  The \vosa{} \(\vosp{k},k\in\CC\setminus\set{-\frac{3}{2}}\) has a non-trivial
  proper ideal if and only 
  there exist integers \(u\ge2\),
  \(v\ge1\) satisfying \(u-v\in2\ZZ\) and \(\gcd\brac*{u,\frac{u-v}{2}}=1\) such that
  \begin{equation}
    \label{eq:admlevels}
    2k+3=\frac{u}{v}.
  \end{equation}
  Further, the ideal is unique (the
  only other ideals are the trivial one and the \vosa{} itself) and it is
  generated by a singular vector \(\vsv{u,v}\) of \(\osp\) weight \((u-1)\alpha\) and
  conformal weight \((u-1)\frac{v}{2}\).
\end{prop}

We refer to the levels satisfying the conditions of \cref{thm:admlevels} as
\emph{admissible} and parametrise them as
\begin{equation}
  k_{u,v}=\frac{u-3v}{2v}.
\end{equation}

\begin{defn}\label{def:minmodvosa}
  Let \(u\ge2\), \(v\ge1\) be integers satisfying \(u-v\in2\ZZ\) and
  \(\gcd\left(u,\frac{u-v}{2}\right)\).
  The \emph{minimal model \(\osp\) \vosa{} \(\MinMod{u}{v}\)} is
  defined to be the unique simple quotient of the universal \(\osp\) \vosa{}
  \(\vosp{k_{u,v}}\) 
  by its maximal proper ideal:
  \begin{equation}
    \MinMod{u}{v} = \frac{\vosp{k_{u,v}}}{\ideal{\vsv{u,v}}}.
  \end{equation}
\end{defn}

Our choice of notation for \(\MinMod{u}{v}\) mimics Kac's notation for \(\osp\) in his
classification of simple Lie superalgebras \cite{KacLsa77}.
Here and hereafter it will always be assumed that the variables \(u\) and
\(v\) satisfy the conditions of \cref{thm:admlevels} and \cref{def:minmodvosa}.
In the context of any formulae that \(u,v\) appear in, the level will always by
taken to satisfy \cref{eq:admlevels}.
Since \(\MinMod{u}{v}\) is a quotient of \(\vosp{k_{u,v}}\), a
module over \(\vosp{k_{u,v}}\) is also a module over \(\MinMod{u}{v}\)
if and only if the ideal \(\ideal{\vsv{u,v}}\) acts trivially. This clearly holds in
both the \ns{} and the Ramond sectors, that is, for both untwisted and twisted modules.

\section{Zhu algebras and module classification}
\label{sec:zhu}

Zhu algebras are associative algebras constructed from
\vosa{s}. They can be interpreted as the algebras of zero modes of \vosa{}
fields acting on relaxed highest weight vectors. They are invaluable aides to
module classification, because there is a one to one
correspondence between simple Zhu algebra modules and simple relaxed highest
weight \vosa{} modules. This is due to the fact that the space of relaxed
highest weight vectors of a simple relaxed highest weight \vosa{} module is
naturally a Zhu algebra module and conversely any Zhu algebra module can be
induced to a \vosa{} module whose space of relaxed highest weights is the Zhu
algebra module. As their name suggests, Zhu algebras were originally
considered by Zhu \cite{ZhuMod96} for \voa{s}. Zhu's work was then later
generalised to untwisted modules over \vosa{s} by Kac and Wang \cite{KacVer94} and to
twisted \vosa{} modules by Dong, Li and Mason \cite{DonTwi98}. We refer to
\cite[App.~A]{BloSVir16} for the definitions and conventions used here.

\begin{prop}
  For \(k\in\CC\setminus\set{-\frac{3}{2}}\), the untwisted and twisted Zhu
  algebras of \(\vosp{k}\) are
  \begin{equation}
    \zhu{\vosp{k}}\cong \UEA{\osp},\quad \tzhu{\vosp{k}}\cong \UEA{\slt}.
  \end{equation}
\end{prop}

The untwisted case was proved in \cite{KacVer94} and the twisted case in
\cite{SnaOsp17}, both follow from the same reasoning as in \cite{FreVer92}
where non-super affine Lie algebras were considered.

\begin{prop}\label{thm:zhupresentation}
  Let \(u\ge2\), \(v\ge1\) be integers satisfying \(u-v\in2\ZZ\) and
  \(\gcd\left(u,\frac{u-v}{2}\right)\).
  The untwisted and twisted Zhu algebras of \(\MinMod{u}{v}\) are
  \begin{equation}
    \zhu{\MinMod{u}{v}}\cong \frac{\UEA{\osp}}{\ideal{\left[\vsv{u,v}\right]}},\quad 
    \tzhu{\MinMod{u}{v}}\cong
    \begin{cases}
      \frac{\UEA{\slt}}{\ideal{\left[\vsv{u,v}\right]^\parity}},& u,v \text{ odd}\\
      \frac{\UEA{\slt}}{\ideal{\left[y_0\vsv{u,v}\right]^\parity}},& u,v \text{ even}
    \end{cases},
  \end{equation}
  where \(\ideal{\left[\vsv{u,v}\right]}\) and
  \(\ideal{\left[\vsv{u,v}\right]^\parity}\) (or \(\ideal{\left[y_0\vsv{u,v}\right]^\parity}\)) are the two sided ideals
  generated by the images of the
  singular vector \(\vsv{u,v}\) (or \(y_0\vsv{u,v}\)) in \(\zhu{\MinMod{u}{v}}\) and \(\tzhu{\MinMod{u}{v}}\), respectively.
\end{prop}
The above presentations for the untwisted and twisted Zhu algebras were proved
in \cite{SnaOsp17}.
We will be computing the image of the singular vector in the Zhu algebras by
evaluating certain polynomial constraints. These constraints are only meaningful
once one knows that the image of the singular vector is non-zero.

\begin{lem}\label{thm:svnonzeroimage}
  Let \(u\ge2\), \(v\ge1\) be integers satisfying \(u-v\in2\ZZ\) and
  \(\gcd\brac*{u,\frac{u-v}{2}}\). The image of the singular vector
  \(\vsv{u,v}\) of \(\vosp{k_{u,v}}\) in the untwisted and twisted Zhu
  algebras  is non-zero.
\end{lem}

\begin{proof}
  We prove the lemma by contradiction. If the image of the singular vector
  were zero in the untwisted or twisted Zhu algebras, then the two sided
  ideals of \cref{thm:zhupresentation} would be zero ideals. Hence the Zhu
  algebras of \(\vosp{k_{u,v}}\) and \(\MinMod{u}{v}\) would be isomorphic and every
  module over a \(\vosp{k}\) Zhu algebra would also be a module over the
  corresponding \(\MinMod{u}{v}\) Zhu algebra. This would imply that every
  simple \(\vosp{k_{u,v}}\) module would also be a simple \(\MinMod{u}{v}\)
  module and in particular that the field \(\vsv{u,v}(z)\) acts trivially on
  every simple \(\vosp{k_{u,v}}\) module.

  Consider the mode of \(\vsv{u,v}(z)\) of index \(-(u-1)\frac{v}{2}\), that
  is, the coefficient of \(z^0\). This mode acts non-trivially in
  \(\vosp{k_{u,v}}\) and therefore corresponds to a non-zero element in (a
  completion of) the universal enveloping algebras of \(\aosp{0}\), in
  particular its projection onto the universal enveloping algebra of the
  subalgebra \(\NSosp{-}\oplus\cart\) is non-zero.
  So applying this mode to the highest weight vector of the Verma module
  \(\Ind{}{}{\fNSh{\lambda}}\) evaluates to a non-zero linear combination of
  monomials in \(\aosp{0}\) generators with non-positive index, where all
  occurrences of \(h_0\) are replaced by \(\lambda\), that is, the coefficients of
  the monomials will be polynomials in \(\lambda\). These polynomials cannot
  vanish for every simple Verma module \(\Ind{}{}{\fNSh{\lambda}}\), since
  there are infinitely many values of \(\lambda\) for which
  \(\Ind{}{}{\fNSh{\lambda}}\) is simple. This contradicts \(\vsv{u,v}(z)\)
  acting trivially on every simple \(\vosp{k_{u,v}}\) module, thus the image of
  the singular vector in the untwisted Zhu algebra must be non-zero. 
  
  The twisted case follows from the untwisted case. Since there are
  simple \ns{} Verma modules \(\Ind{}{}{\fNSh{\lambda}}\) which are not
  modules over \(\MinMod{u}{v}\), the Ramond Verma module 
  \(\Ind{}{}{\fRh{k-\lambda}}\cong\zeta\brac*{\Ind{}{}{\fNSh{\lambda}}}\)
  is also not a twisted module over \(\MinMod{u}{v}\) and thus the image of
  the singular vector in the twisted Zhu algebras must be non-zero.
\end{proof}

Recall that the Zhu algebras are filtered by conformal weight and so vectors
in \(\osp\) such as \(e\) and \(h\) are assigned degree 1 in the Zhu algebras
because the are the images of the conformal weight 1 fields \(e(z)\) and
\(h(z)\). The conformal weight of the Virasoro field is 2 and its image in the
untwisted and twisted Zhu algebras is proportional to the quadratic Casimirs
\(Q^\osp\) and \(Q\), respectively, which are therefore assigned degree
2. Finally, since the \(\osp\) quadratic Casimir, \(Q^\osp\), is quadratic in
the super-Casimir, \(\Sigma\), the degree of \(\Sigma\) in the untwisted Zhu
algebra is 1.

\begin{lem}
  \label{thm:zhusvshape}
  Let \(u\ge2\), \(v\ge1\) be integers satisfying \(u-v\in2\ZZ\) and
  \(\gcd\left(u,\frac{u-v}{2}\right)\). Then in the untwisted 
  Zhu algebra of \(\vosp{k_{u,v}}\) the image of the singular vector is
  \begin{equation}\label{eq:nssvshape}
    \left[\vsv{u,v}\right]=
    \begin{cases}
      e^{\frac{u-1}{2}} g\brac*{h,\Sigma}&\text{if}\ u,v\ \text{odd}\\
      e^{\frac{u-2}{2}}x g\brac*{h,\Sigma}&\text{if}\ u,v\ \text{even} 
    \end{cases},
  \end{equation}
  while in the twisted Zhu algebra it is 
  \begin{align}\label{eq:rsvshape}
    \left[\vsv{u,v}\right]^\parity&=e^{\sfrac{(u-1)}{2}}
    g_\parity\brac*{h,Q},\quad
    \text{if}\ u,v\ \text{odd},\nonumber\\
    \left[y_0\vsv{u,v}\right]^\parity&=e^{\sfrac{(u-2)}{2}}
    g_\parity\brac*{h,Q},\quad
    \text{if}\ u,v\ \text{even},
  \end{align}
  where \(g\) and \(g_\parity\) are polynomials whose degrees satisfy
  \begin{equation}\label{eq:poldegbounds}
    \deg g\brac*{h,\Sigma}\le \frac{(u-1)(v-1)}{2},
    \qquad
    \deg g_\parity\brac*{h,Q}\le\frac{(u-1)(v-1)+1}{2}.
  \end{equation}
\end{lem}

\begin{proof}
  The \(\osp\) weight of the \sv{} is \((u-1)\alpha\) and so the images of
  the \sv{} in the untwisted and twisted Zhu algebras are elements
  of the universal enveloping algebras of \(\osp\) and \(\slt\), respectively,
  of the same weight. Since \(\osp\) and \(\slt\) are rank 1, it is clear that
  any homogeneous element of the universal enveloping algebras of positive
  weight can be written as the product of a monomial in the positive root
  vectors and an element of the centralisers of the Cartan subalgebra
  introduced in \cref{thm:centralgebras}. The images of the singular must
  therefore be as given in \eqref{eq:nssvshape} and \eqref{eq:rsvshape}.

  Since the conformal weight of the singular vectors is \(\frac{u-1}{2}v\), the
  total degree of the images of the singular vectors is at most \(\frac{u-1}{2}v\).
  The factors of \(e\) and \(x\) in the images of the singular vectors
  each contribute 1 unit of degree and the \rhs{s} of the inequalities of
  \eqref{eq:poldegbounds} are just the upper bounds on how many units of
  degree the polynomials
  \(g\) and \(g_\parity\) can contribute.
\end{proof}

\begin{defn}
  The \emph{Kac table} of \(\MinMod{u}{v}\) is the set of pairs of integers
  \begin{equation}
    \kac{u,v}=\left\{(i,j) \st 1\le i \le u-1,\ 1\le j\le v-1\right\}
  \end{equation}
  and the \emph{\ns{}} 
  and \emph{Ramond Kac tables}
  are the subsets
  \begin{align}
    \NSkac{u,v}&=\left\{ (i,j)\in\kac{u,v}\st i+j\ \text{is odd}\right\},\nonumber\\
    \Rkac{u,v}&=\left\{ (i,j)\in\kac{u,v}\st i+j\ \text{is
        even}\right\}, 
  \end{align}
  respectively. Let \(\sim\) denote the equivalence relation on \(\kac{u,v}\)
  given by \((i,j)\sim (i^\prime,j^\prime)\) if and only if
  \((i,j)=(i^\prime,j^\prime)\) or \((i,j)=(u-i^\prime,v-j^\prime)\). The
  \emph{reduced Kac table}, the \emph{reduced \ns{} Kac table} and the
  \emph{reduced Ramond Kac table} are then respectively defined to be
  \begin{equation}
    \rkac{u,v}=\kac{u,v}/\sim,\quad
    \rNSkac{u,v}=\NSkac{u,v}/\sim,\quad
    \rRkac{u,v}=\Rkac{u,v}/\sim.
  \end{equation}
\end{defn}
For \(i,j\in\ZZ\), let
\begin{equation}
  \lambda_{i,j}=\frac{i-1}{2}-\frac{1+(-1)^{i+j}}{4}-\frac{u}{2v}j,\quad
  s_{i,j}=\frac{i}{2}-\frac{u}{2v}j,\quad
  q_{i,j}=\frac{\brac*{uj-vi}^2-4v^2}{8v^2}.
\end{equation}
\begin{thm}
  \label{thm:svimage}
  Let \(u\ge2\), \(v\ge1\) be integers satisfying \(u-v\in2\ZZ\) and
  \(\gcd\left(u,\frac{u-v}{2}\right)=1\).
  Up to normalisation, the images of the singular vectors \(\vsv{u,v}\) in
  \(\vosp{k_{u,v}}\) are given by the following formulae: 
  In the untwisted Zhu algebra,
  \begin{equation}
    \left[\vsv{u,v}\right]=
    \begin{cases}
      e^{\frac{u-1}{2}} g\brac*{\Sigma}& u,v\ \text{odd}\\
      e^{\frac{u-2}{2}}x g\brac*{\Sigma}& u,v\ \text{even}
    \end{cases}
  \end{equation}
  and in the twisted Zhu algebra,
  \begin{align}
    e^{\sfrac{(u-2)}{2}} g_\parity\brac*{Q}=
    \begin{cases}
      \left[\vsv{u,v}\right]^\parity& u,v\ \text{odd}\\
      \left[y_0\vsv{u,v}\right]^\parity& u,v\ \text{even}
    \end{cases},
  \end{align}
    where
  \begin{align}
    \label{eq:centralfactor}
    g\brac*{\Sigma}&=
    \prod_{(i,j)\in\NSkac{u,v}}
    \brac*{\Sigma-s_{i,j}} \qquad \text{and}\qquad
    g_\parity\brac*{Q}=
    \prod_{[(i,j)]\in \rRkac{u,v}}
    \brac*{Q-q_{i,j}}.
  \end{align}
\end{thm}

The proof of the above polynomial formulae is where most of the effort of this
paper is spent. The main idea is to evaluate the action of the zero mode of
the singular vector on candidate relaxed highest
weight vectors. Since the images of the singular vector in the Zhu algebras
are polynomials in \(\osp\) or \(\slt\) generators, these polynomials can be
determined through sufficiently many
evaluations. Unfortunately, closed formulae for singular vectors are notoriously hard to find.
We sidestep this issue by constructing the
singular vector \(\vsv{u,v}\) within free field realisations as the image of
screening operators. 
However, these free field methods require significant preparation and so, for
greater clarity of presentation, we postpone 
the proof of \cref{thm:svimage}
to \cref{sec:0modecalc}  where it has been
split up into four cases:
\cref{thm:nszhupoly1,thm:rzhupoly1,thm:nszhupoly2,thm:rzhupoly2}.

\begin{thm}
  \label{thm:moduleclassification}
  Let \(u\ge2\), \(v\ge1\) be integers satisfying \(u-v\in2\ZZ\) and
  \(\gcd\left(u,\frac{u-v}{2}\right)\). 
  The minimal model \(\osp\) \vosa{} \(\MinMod{u}{v}\) is rational in category
  \(\catO\) in both the \ns{} and Ramond sectors, that is, both sectors admit
  only a finite number of isomorphism classes of simple modules and every
  module is semisimple.
  Any simple \(\MinMod{u}{v}\) module in category \(\catR\) 
  is isomorphic to one of the following mutually inequivalent modules.\\
  In the \ns{} sector:
  \begin{enumerate}
  \item\label{itm:nsfinite} \(\NSf{\lambda_{i,0}}\) or \(\rpar{\NSf{\lambda_{i,0}}}\),
    where \(1\le i\le u-1\) and \(i\) is odd.
  \item\label{itm:nshw} \(\NSh{\lambda_{i,j}}\) or \(\rpar{\NSh{\lambda_{i,j}}}\), where 
    \((i,j)\in \NSkac{u,v}\).
  \item \(\NSl{-\lambda_{i,j}}\) or \(\rpar{\NSl{-\lambda_{i,j}}}\), where 
    \((i,j)\in \NSkac{u,v}\).
  \item\label{itm:nsdense} \(\NSr{\sqbrac*{\lambda},s_{i,j}}\), 
    where \(\sqbrac*{\lambda}\in\CC/2\ZZ\), 
    \((i,j)\in\NSkac{u,v}\)
    and \(s_{i,j}^2\neq \brac*{\mu+\frac{1}{2}}^2\) for all \(\mu\in \sqbrac*{\lambda}\cup\sqbrac*{\lambda+1}\).
  \end{enumerate}
  In the Ramond sector:
  \begin{enumerate}[resume]
  \item\label{itm:rfinite} \(\Rf{\lambda_{i,0}}\) or
    \(\rpar{\Rf{\lambda_{i,0}}}\), where 
    \(1\le i\le u-1\) and \(i\) is even.
  \item\label{itm:rhw} \(\Rh{\lambda_{i,j}}\) or \(\rpar{\Rh{\lambda_{i,j}}}\), where 
    \((i,j)\in\Rkac{u,v}\).
  \item \(\Rl{-\lambda_{i,j}}\) or  \(\rpar{\Rl{-\lambda_{i,j}}}\), where
    \((i,j)\in \Rkac{u,v}\).
  \item\label{itm:rdense} \(\Rr{\sqbrac*{\lambda},q_{i,j}}\) or \(\rpar{\Rr{\sqbrac*{\lambda},q_{i,j}}}\), where \(\sqbrac*{\lambda}\in\CC/2\ZZ\),
    \(\sqbrac*{(i,j)}\in\rRkac{u,v}\) and 
    \(q_{i,j}\neq \mu(\mu+2)\) for all
    \(\mu\in \sqbrac*{\lambda}\).
  \end{enumerate}
\end{thm}
Note that the parity reversals of the dense modules
\(\NSr{\sqbrac*{\lambda},s_{i,j}}\) do not need to be included, as
\(\rpar{\NSr{\sqbrac*{\lambda},s_{i,j}}}\cong \NSr{\sqbrac*{\lambda+1},-s_{i,j}}=\NSr{\sqbrac*{\lambda+1},s_{u-i,v-j}}\).
\begin{proof}
  The theorem follows by evaluating the action of the image of the singular
  vector in the untwisted and twisted Zhu algebras on the simple modules of
  Theorems \ref{thm:finitesl} and \ref{thm:finiteosp}. Those simple \(\osp\)
  and \(\slt\) modules on which the image of the singular vector acts
  trivially can be induced to modules over \(\MinMod{u}{v}\).

  In both the \ns{} and Ramond sectors, the factors of \(e\) and \(x\) in the
  image of the singular vectors cannot act trivially on infinite
  dimensional simple weight modules. They do however act trivially on finite
  dimensional simple weight modules up to a certain dimension, that is, those
  of points \ref{itm:nsfinite} and \ref{itm:rfinite}.

  For the infinite dimensional simple weight modules, since the factors of
  \(e\) and \(x\) cannot act trivially, the polynomials \(g\) or \(g_\parity\)
  must. Thus, in the \ns{} sector the super-Casimir \(\Sigma\) must act as
  \(\pm s_{i,j}\cdot\id\) for one of the \(s_{i,j}\) in the factorisation
  \eqref{eq:centralfactor} of \(g\). The modules for which this is the case
  are precisely those of points \ref{itm:nshw} -- \ref{itm:nsdense}. Similarly,
  in the Ramond sector the \(\slt\) Casimir \(Q\) must act as \(q_{i,j}\cdot\id\)
  for one of the \(q_{i,j}\) in the factorisation \eqref{eq:centralfactor} of
  \(g_\parity\).The modules for which this is the case
  are precisely those of points \ref{itm:rhw} -- \ref{itm:rdense}.

  To conclude rationality in category \(\catO\) one must show that there exist
  no indecomposable extensions of modules in points \ref{itm:nsfinite} and
  \ref{itm:nshw} or modules in points \ref{itm:rfinite} and
  \ref{itm:rhw}. Consider first the \ns{} sector. An indecomposable
  self-extension would have as its space of relaxed highest weight vectors an
  indecomposable self extension \(\overline{\Mod{M}}\) of a simple highest
  weight Zhu algebra module of highest weight \(\lambda\alpha\).
  Further, the weight space of \(\overline{\Mod{M}}\) of weight
  \(\lambda\alpha\) would be an indecomposable module over \(\CC[h,\Sigma]\),
  the centraliser of the Cartan subalgebra. If we choose a \PBW{} ordering
  which places \(y\) to the left of \(x\), then \(\Sigma=h-2yx+1/2\).
  Since \(h\) acts
  semisimply in category \(\catO\) and \(yx\) acts trivially on the weight
  space of weight \(\lambda\alpha\), the formula for \(\Sigma\)
  implies that \(\Sigma\) acts semisimply on the weight
  space of weight \(\lambda\alpha\) in \(\overline{\Mod{M}}\), thereby contradicting
  indecomposability. To rule out indecomposable extensions of two inequivalent
  simple modules, \(\Mod{M}\) and \(\Mod{N}\), in category \(\catO\), note that
  either the indecomposable module or its contragredient dual would be a
  highest weight module and thus a quotient of a Verma module. In particular
  the highest weight of the submodule would need match the weight of a
  singular vector of the Verma module. However, from \cite[Thm.~3.1]{IoOsp01},
  it is easy to verify that Verma modules with highest weights as in points
  \ref{itm:nsfinite} or 
  \ref{itm:nshw} do not have singular vectors with weights as in points \ref{itm:nsfinite} or
  \ref{itm:nshw}.

  The rationality of the Ramond sector in category \(\catO\) follows from that of
  the \ns{} sector after twisting by the algebra isomorphism \(\zeta\) of
  \eqref{eq:catotwist}:
  Since \(\zeta\) preserves the triangular decomposition used to define
  category \(\catO\), the twist of an indecomposable extension of Ramond
  modules in category \(\catO\) would be an
  indecomposable extension of \ns{} modules in category \(\catO\), but no such
  extensions exist in the \ns{} sector of category \(\catO\) and so neither
  can they in the Ramond sector of category \(\catO\).
\end{proof}

\section{Free field algebras and screening}
\label{sec:ffrandscr}

In this section we shall set the stage for proving \cref{thm:svimage} by following a line of
reasoning that has already been proved to be very successful for similar Zhu
algebra related problems \cite{TsuExt13,RidJac14,RidRel15,BloSVir16}. The
basic idea is as follows. First we construct a free field
realisation of the universal \(\osp\) \vosa{} \(\vosp{k}\), that is, we
embed \(\vosp{k}\) into free field \vosa{s}. Free field \vosa{s} have the
convenient property of allowing one to easily construct certain module
homomorphisms called screening operators which in turn can be used to realise
the singular vector as the image of certain highest weight vectors. Armed with
these formulae for singular vectors, we then evaluate the action of the
zero mode of the singular vector on candidate relaxed highest weight vectors
to deduce the image of the singular vector in the Zhu algebras. 

\subsection{The Heisenberg free field algebra and lattice \va{s}}
\label{sec:heis}

The rank \(r\) Heisenberg algebra is the affinisation of the trivial \(r\)
dimensional complex Lie algebra. Though the only ranks required here will be \(r=1\)
and \(r=2\), we give the general definition for conceptual clarity. The
definition presented here is essentially the definition of ``free (super)
  bosons'' in \cite{KacVB98}.

Let \(\heisvs{r}\) be an \(r\) dimensional complex (purely even) vector space
together with a symmetric non-degenerate
bilinear form \((-,-)\) (recall that all such forms are equivalent over the
complex numbers). The Heisenberg algebra is given by
\begin{equation}
  \hlie{r}=\heisvs{r}\otimes \CC\sqbrac*{t,t^{-1}}\oplus\CC \wun,
\end{equation}
as a vector space, where \(\wun\) is the central element which will always be identified with the
identity when acting on modules. The Lie bracket is given by
\begin{equation}
  \label{eq:hcommrel}
  [a_m,b_n]=m\brac*{a,b}\delta_{m+n,0}\wun,\quad a_m=a\otimes t^m,
  \quad b_n=b\otimes t^n,\quad
  a,b\in\heisvs{r}.
\end{equation}
The usual triangular decomposition of the Heisenberg algebra is then given by
\begin{equation}
  \hlie{r}^\pm=\heisvs{r}\otimes\CC[t^\pm]t^\pm,\quad \hlie{r}^0=\heisvs{r}\oplus\CC\wun.
\end{equation}
The Verma modules with respect to this decomposition are called \emph{Fock
  spaces} and they are all simple. Writing
\(\hlie{r}^{\ge}=\hlie{r}^+\oplus\hlie{r}^0\), we define
\begin{equation}
  \Fock{b}=\Ind{\hlie{r}}{\hlie{r}^{\ge}}{\CC \ket{b}},\quad b\in\heisvs{r},
\end{equation}
to be the Verma module induced from the 1 dimensional \(\hlie{r}^\ge\) module
characterised by
\begin{equation}
  a_0\ket{b}=(a,b)\ket{b},\quad \wun\ket{b}=\ket{b},\quad \hlie{r}^+\ket{b}=0.
\end{equation}

\begin{defn}
  Let \(\heisvs{r}\) be an \(r\)
  dimensional complex vector space with any choice of symmetric non-degenerate
  bilinear form \((-,-)\) and any choice of basis \(\{a^1,\dots,a^r\}\). The
  \emph{rank \(r\) Heisenberg \va{}} is the unique \va{} that is strongly generated by
  fields \(a^1(z),\dots,a^r(z)\) (one for each basis vector), subject to the
  defining \ope{s}
  \begin{equation}
    \label{eq:hoperels}
    a^i(z)a^j(z)\sim \frac{(a^i,a^j)}{(z-w)^2},\quad 1\le i,j\le r,
  \end{equation}
  and satisfies no additional relations beyond those required by \va{} axioms.
\end{defn}

The \ope{s} \eqref{eq:hoperels} imply that the modes of the Laurent expansions
\begin{equation}
  a(z)=\sum_{n\in\ZZ} a_n z^{-n-1},\quad a\in\heisvs{r}
\end{equation}
satisfy the commutation relations \eqref{eq:hcommrel} of the Heisenberg
algebra \(\hlie{r}\). As a module over itself \(\hvoa{r}\) corresponds to \(\Fock{0}\).
The Heisenberg \va{s} admit continuous families of conformal structures, however, we
shall wait until later applications to choose a specific one.

\emph{Vertex operators} are maps between Fock spaces whose definition requires
us to extend the universal enveloping algebra of the Heisenberg Lie algebra
\(\UEA{\hlie{r}}\) by the group algebra \(\CC[\heisvs{r}]\), where
\(\heisvs{r}\) is viewed as an abelian group under vector space addition. 
For any vector \(a\in\heisvs{r}\),
we denote the corresponding group algebra element by \(\ee^a\). Then
\(\CC[\heisvs{r}]\otimes\UEA{\hlie{r}}\) is an associative algebra after imposing
the relations
\begin{equation}
  [a_n,\ee^b]=(a,b) \delta_{n,0} \ee^b,\quad a,b\in\heisvs, n\in\ZZ.
\end{equation}
Finally, we define the group algebra elements to act on the highest weight
vectors of Fock spaces in the following way:
\begin{equation}
  \ee^a\ket{b}=\ket{a+b}.
\end{equation}
For any \(a\in\heisvs{r}\), the
vertex operator \(\vop{a}{z}\), corresponding to \(a\), is given by
\begin{equation}
  \vop{a}{z}=\ee^az^{a_0}\prod_{m\ge1}\exp\brac*{\frac{a_{-m}}{m}z^m}\exp\brac*{-\frac{a_m}{m}z^{-m}}.
\end{equation}
On Fock spaces \(\vop{a}{z}\) defines a map
\begin{equation}
  \vop{a}{z}: \Fock{b}\mapsto \Fock{a+b}\powser*{z,z^{-1}}z^{(a,b)}
\end{equation}
and the composition of \(\ell\) vertex operators associated to vectors
\(a^1,\dots,a^\ell\in\heisvs{r}\) is given by
\begin{equation}\label{eq:vertope}
  \vop{a^1}{z_1}\cdots\vop{a^\ell}{z_\ell}=\prod_{1\le i<j\le
    \ell}(z_i-z_j)^{(a^i,a^j)}\ee^{a^1+\cdots+a^\ell}\prod_{i=1}^\ell
  z_i^{a^i_0}
  \prod_{m\ge1}\exp\brac*{\sum_{i=1}^\ell\frac{a^i_{-m}}{m}z_i^m}\exp\brac*{-\sum_{i=1}^\ell\frac{a^i_m}{m}z_i^{-m}}.
\end{equation}

\emph{Lattice \va{s}} are (infinite order) extensions of Heisenberg \va{s}
\(\hvoa{r}\) constructed by
picking linearly independent vectors \(a^1,\dots,a^m\) in \(\heisvs{r}\) such that the pairings of
these vectors are integral. Let \(\ang*{a^1,\dots,a^m}=\bigoplus_{i=1}^m
\ZZ a^{i}\) denote the lattice in \(\heisvs{r}\) generated by \(a^1,\dots,a^m\).
\begin{defn}
  Let \(\heisvs{r}\) be an \(r\)
  dimensional complex vector space with a choice of symmetric non-degenerate
  bilinear form \((-,-)\) and associated Heisenberg
  \va{} \(\hvoa{r}\), further let 
  \(\{a^1,\dots,a^m\}\), \(1\le m\le r\) be a choice of linearly independent
  vectors in \(\heisvs{r}\), such
  that the pairings of these vectors are integral. The \emph{lattice \va{}
    \(\lva{a^1,\dots,a^m}\)} is the extension of \(\hvoa{r}\cong\Fock{0}\) which as a
  vector space is given by
  \begin{equation}
    \lva{a^1,\dots,a^m}=\bigoplus_{\mu\in\ang*{a^1,\dots,a^m}}\Fock{\mu}
    =\bigoplus_{n_1,\dots,n_m\in\ZZ}\Fock{n_1 a^1+\cdots +n_m a^m},
  \end{equation}
  where the state-field correspondence is uniquely determined by assigning to each
  highest weight vector \(\ket{n_1a^1+\cdots +n_m a^m}\) the vertex operator \(\vop{n_1a^1+\cdots n_m a^m}{z}\).
\end{defn}

\begin{rmk}
  Note that we are not requiring the pairing
  on the lattice in the above definition to be even or positive definite, as
  it will be necessary to consider indefinite lattices for
  the free field realisation of the \(\beta\gamma\) ghost \va{} constructed below.
\end{rmk}

\begin{prop}\label{thm:lattmods}
  Let \(\ang*{a^1,\dots,a^m}^\ast=\set{\mu\in \heisvs{r}\st (\mu,\lambda)\in\ZZ\
  \forall \lambda\in \ang*{a^1,\dots,a^m}}\) be the dual of
  \(\ang*{a^1,\dots,a^m}\) in \(\heisvs{r}\). For each coset \([\lambda]\in
  \ang*{a^1,\dots,a^m}^\ast/\ang*{a^1,\dots,a^m}\)
  \begin{equation}
    \Fock{[\lambda]}=\bigoplus_{\mu\in \ang*{a^1,\dots,a^m}}\Fock{\lambda+\mu}
  \end{equation}
  is a module over \(\lva{a^1,\dots,a^m}\).
\end{prop}
\begin{proof}
  Clearly the action of \(\lva{a^1,\dots,a^m}\) closes on
  \(\Fock{[\lambda]}\), so the only potential obstruction to
  \(\Fock{[\lambda]}\) being a module (and not a twisted module) is the fields
  of \(\lva{a^1,\dots,a^m}\) being single valued, that is, that their series
  expansions only have integer exponents. This is true if and only if it is
  true for the vertex operators \(\vop{a^i}{z},\ 1\le i\le m\) and
  it is true for those by construction
  due to the Fock space weights all lying in \(\ang*{a^1,\dots,a^m}^\ast\).
\end{proof}

\subsection{The \(\beta\gamma\) ghosts}
\label{sec:betagamma}

We refer to \cite{RidBos14} for an in depth treatment of the \(\beta\gamma\)
ghosts.
The \(\beta\gamma\) ghost algebra is an infinite dimensional complex Lie algebra
\begin{equation}
  \bglie=\bigoplus_{n\in\ZZ}\CC \beta_n\oplus\bigoplus_{n\in\ZZ}\CC \gamma_n\oplus\CC\wun,
\end{equation}
whose Lie brackets are
\begin{equation}
  \label{eq:bgcommrel}
  \comm{\gamma_m}{\beta_n}=\delta_{m+n,0}\wun,\quad
  \comm{\beta_m}{\beta_n}=\comm{\gamma_m}{\gamma_n}=0,\quad m,n\in\ZZ.
\end{equation}

The \(\beta\gamma\) ghost algebra admits the relaxed triangular decomposition
\begin{equation}
  \label{eq:bgtridecomp}
  \bglie^{\pm}=\bigoplus_{n\ge1}\CC \beta_{\pm n}
  \oplus\bigoplus_{n\ge1}\CC\gamma_{\pm n},\quad \bglie^0=\CC\beta_0\oplus\CC\gamma_0\oplus\CC\wun.
\end{equation}
Conveniently, \(\bglie^0\) is just the \(A_1\) Weyl algebra for which Block
classified all simple modules \cite{BloIrr79}. Here we shall only need simple
weight modules, and to list them we introduce the operator
\(J=\gamma_0\beta_0\).
\begin{prop} \label{prop:FinGhReps}
Every simple \(\bglie^0\) weight module is isomorphic to one of the following.
\begin{enumerate}
\item The simple highest weight module $\fbgh = \CC[\gamma_0] \finite{\Vac}$ generated by a \hwv{} $\finite{\Vac}$ satisfying $\beta_0 \finite{\Vac} = 0$, hence $J \finite{\Vac} = 0$.
\item The simple lowest weight $\bglie^0$-module $\fbgl = \CC[\beta_0]
  \finite{\omega}$ generated by a \lwv{} $\finite{\omega}$ satisfying
  $\gamma_0 \finite{\omega} = 0$, hence $J \finite{\omega} =
  \finite{\omega}$. 
\item \label{it:ParaClass} 
  The simple dense module $\fbgd{\sqbrac*{\lambda}}$,
  \(\sqbrac*{\lambda}\in\CC/\ZZ\), \(\sqbrac*{\lambda}\neq\sqbrac*{0}\) with a basis of weight
  $\finite{u}_j$, 
  satisfying $J \finite{u}_j = j \finite{u}_j$, hence
  $\fbgd{\sqbrac*{\lambda}}=\CC[\beta_0] \finite{u}_{\mu} \oplus \CC[\gamma_0]
  \gamma_0   \finite{u}_{\mu}$.
\end{enumerate}
\end{prop}
In addition to the simple weight modules above, we also need to consider two
indecomposable weight modules whose \(J\)-weight support is \(\ZZ\) and whose isomorphism classes are determined by the following short exact sequences ($\Extgrp{1}{\fbgl}{\fbgh} = \Extgrp{1}{\fbgh}{\fbgl} = \CC$):
\begin{equation}
\dses{\fbgh}{}{\fbgd{0}^+}{}{\fbgl}, \qquad 
\dses{\fbgl}{}{\fbgd{0}^-}{}{\fbgh}.
\end{equation}
Both may be realised on the space $\CC[\gamma_0] \finite{u}_0 \oplus \CC[\beta_0] \finite{u}_1$, where $\beta_0 \finite{u}_0 = 0$ and $\gamma_0 \finite{u}_1 = a^+ \finite{u}_0$, for $\bgd{0}^+$, and $\beta_0 \finite{u}_0 = a^- \finite{u}_1$ and $\gamma_0 \finite{u}_1 = 0$, for $\bgd{0}^-$.  We may normalise the basis vectors so that $a^+ = a^- = 1$.

We denote the inductions of the above modules by the same
symbols but without the overlines:
\begin{equation}
  \bgh=\Ind{\bglie}{\bglie^\ge}{\fbgh},\quad
  \bgl=\Ind{\bglie}{\bglie^\ge}{\fbgl},\quad
  \bgd{\sqbrac*{\lambda}}=\Ind{\bglie}{\bglie^\ge}{\fbgd{\sqbrac*{\lambda}}},\quad
  \bgd{0}^\pm=\Ind{\bglie}{\bglie^\ge}{\fbgd{0}^\pm}.
\end{equation}
The first three of the above modules are simple while the last are characterised by
the non-split exact sequences
\begin{equation}
\dses{\bgh}{}{\bgd{0}^+}{}{\bgl}, \qquad 
\dses{\bgl}{}{\bgd{0}^-}{}{\bgh}.
\end{equation}

The \(\beta\gamma\) algebra admits an algebra automorphism \(\sigma\) called
\emph{spectral flow} whose action on generators is given by
\begin{equation}
  \sigma(\beta_n)=\beta_{n-1},\quad\sigma(\gamma_n)=\gamma_{n+1},\quad
  \sigma(\wun)=\wun.
\end{equation}

Since spectral flow is an algebra automorphism, it can be used to twist
modules (recall \cref{def:isotwists}).
In particular, \(\bgh\) and \(\bgl\) are related by spectral flow:
\begin{equation}
  \bgl\cong \sigma^{-1}\bgh.
\end{equation}
Note, however, that spectral flow does not preserve the relaxed triangular
decomposition \eqref{eq:bgtridecomp} and therefore the spectral flow of a
relaxed highest weight module will generally no longer be relaxed highest
weight (it will instead be relaxed highest weight with respect to a different triangular decomposition).

\begin{defn}
  The \emph{\(\beta\gamma\) ghost \va{}} \(\bgva\) is the unique \va{} that is strongly
  generated by two even fields \(\beta(z)\) and \(\gamma(z)\), has the \ope{s}
  \begin{equation}
    \label{eq:bgoperel}
    \gamma(z)\beta(w)\sim\beta(w)\gamma(z)\sim \frac{1}{z-w},\quad \beta(z)\beta(w)\sim\gamma(z)\gamma(w)\sim0.
  \end{equation}
  and satisfies no additional relations beyond those required by \va{} axioms.
\end{defn}
The \ope{} \eqref{eq:bgoperel} implies that the modes of the Laurent expansion
\begin{equation}
  \beta(z)=\sum_{n\in\ZZ}\beta_n z^{-n-1},
  \quad
  \gamma(z)=\sum_{n\in\ZZ}\gamma_n z^{-n}
\end{equation}
satisfy the commutation relations \eqref{eq:bgcommrel} of the \(\beta\gamma\)
ghost algebra.
As a \(\beta\gamma\) ghost algebra module, \(\bgva\) is isomorphic to
\(\bgh\). The \(\beta\gamma\) ghost \va{} admits a 1-parameter family of
Virasoro fields. Here we shall only consider the unique choice
\(T=-\normord{\beta(z)\partial \gamma(z)}\) which assigns conformal weight 1
to \(\beta\) and conformal weight 0 to \(\gamma\). Below it will also be
necessary to consider the field \(J(z)=\normord{\beta(z)\gamma(z)}\) which
generates a rank 1 Heisenberg \va{} and grades the \(\beta\gamma\) ghost \va{}
by assigning weight 1 to \(\beta\) and weight \(-1\) to \(\gamma\). The
central charge for this choice of Virasoro field is \(c=2\).

The \(\beta\gamma\) ghost \va{} admits a free field realisation in terms of a
lattice algebra, that is, there exists an embedding from \(\bgva\) to a lattice algebra.
\begin{prop}
  \label{thm:bosonisation}
  Let \(\hvoa{2}\) be the rank 2 Heisenberg \va{} and let \(\theta,\psi\) be a
  basis of \(\heisvs{2}\) such that the non-vanishing pairings of \((-,-)\)
  are \((\psi,\psi)=-(\theta,\theta)=1\). Then there exists an embedding
  \(\bgva\to\lva{\theta+\psi}\) of \va{s} defined by
  \begin{equation}\label{eq:bosonisation}
    \beta(z)\mapsto \vop{\theta+\psi}{z},\quad \gamma(z)\mapsto \normord{\psi(z)\vop{-\theta-\psi}{z}}.
  \end{equation}
\end{prop}
\begin{proof}
  The map is a homomorphism of \va{s} because \(\vop{\theta+\psi}{z}\) and
  \(\normord{\psi(z)\vop{-\theta-\psi}{z}}\) satisfy the same \ope{s} as
  \(\beta(z)\) and \(\gamma(z)\) do. It is injective because it is non-trivial
  and the \(\beta\gamma\) ghost \va{} is simple.
\end{proof}

Since \(\bgva\) is a vertex subalgebra of \(\lva{\theta+\psi}\), we can decompose
\(\lva{\theta+\psi}\) modules into \(\bgva\) modules. Before we do so, note that the dual
lattice of \(\ang*{\theta+\psi}\) can be written as
\begin{equation}
  \ang*{\theta+\psi}^\ast=\{\lambda(\theta+\psi)+n\psi\st \lambda\in \CC,\ n\in
  \ZZ\}.
\end{equation}
\begin{prop}\label{thm:bgffrmod}
  As modules over \(\bgva\), the \(\lva{\theta+\psi}\) modules of \cref{thm:lattmods} decompose as
  \begin{align}\label{eq:fockdecomp}
    \Fock{[n\psi]}&\cong \sfaut^{n+1} \bgd{0}^-,
    \quad \text{for}\  n\in\ZZ,
    \nonumber\\
    \Fock{[\lambda(\theta+\psi)+n\psi]}
    &\cong\sfaut^{n+1}\bgd{\sqbrac*{\lambda}}
    ,\quad
  \text{for}\ \lambda\in\CC\setminus \ZZ, n\in\ZZ
  \end{align}
\end{prop}
\begin{proof}
  It was shown in \cite{RidBos14} that the characters of \(\sfaut^{n+1}
  \bgd{0}^-\) and \(\sfaut^{n+1}\bgd{\sqbrac*{\lambda}}\), with
  \(\sqbrac*{\lambda}\in\ZZ/\ZZ\), \(\sqbrac*{\lambda}\neq\sqbrac*{0}\) and
  \(n\in\ZZ\) form a basis of the span of characters of relaxed highest weight
  modules and their spectral flow twists. Thus,
  the proposition follows by comparing characters of the left and \rhs{s} of
  \eqref{eq:fockdecomp}. The embedding \eqref{eq:bosonisation} implies that
  \begin{align}
    J(z)=\normord{\beta(z)\gamma(z)}&\mapsto -\theta(z),\nonumber\\
    T(z)=-\normord{\beta(z)\partial \gamma(z)}&\mapsto 
    \frac{1}{2}\normord{\brac*{\psi(z)+\theta(z)}\brac*{\psi(z)+\theta(z)}}-\frac{1}{2}\partial\brac*{\psi(z)+\theta(z)}.
  \end{align}
  Thus
  \begin{align}
    \tr_{\Fock{[\lambda(\theta+\psi)+n\psi]}}\brac*{z^{J_0}q^{L_0-\frac{1}{12}}}&
    =\sum_{m\in \ZZ}\tr_{\Fock{(\lambda+m)(\theta+\psi)+n\psi}}\brac*{z^{J_0}q^{L_0-\frac{1}{12}}}
    =\sum_{m\in \ZZ}\frac{z^{\lambda+m}q^{\frac{1}{2}\brac*{n+1}\brac*{2(\lambda+m)+n}}}{\eta(q)^2}
    \nonumber\\
    &=\frac{z^\lambda
      q^{(n+1)\lambda+\frac{1}{2}(n+1)n}}{\eta(q)^2}\sum_{m\in\ZZ}z^m
    q^{(n+1)m}
    =\ch{\sfaut^{n+1}\bgd{\lambda}},
  \end{align}
  where the last equality was taken from Equation (5.1) of
  \cite{RidBos14}. This proves the proposition whenever the \rhs{} of
  \eqref{eq:fockdecomp} is simple, that is, when \(\lambda\) is not an
  integer. For the non-simple case the character implies 3 possibilities: the
  \rhs{} is either semi-simple or isomorphic to either \(\sfaut^{n+1}\bgd{0}^+\) or
  \(\sfaut^{n+1}\bgd{0}^-\). That the \rhs{} is isomorphic to
  \(\sfaut^{n+1}\bgd{0}^-\) then follows by checking the action of the
  \(\beta\) and \(\gamma\) fields on the highest weight vectors of the Fock
  space summands in \(\Fock{[n\psi]}\).
\end{proof}

\subsection{The \(bc\) ghosts}
\label{sec:bc}

The final free field algebra which we shall consider is the \(bc\) ghost
algebra, the fermionic analogue of the \(\beta\gamma\) ghosts. The \(bc\)
ghost algebras are a pair of infinite dimensional Lie superalgebras
\begin{equation}
  \bclie{\epsilon}=\bigoplus_{n\in \ZZ+\epsilon}\CC
  b_n\oplus\bigoplus_{n\in\ZZ+\epsilon}\CC\gamma_n\oplus\CC\wun,\quad \epsilon=0,\frac{1}{2},
\end{equation}
whose Lie brackets are
\begin{equation}
  \acomm{b_m}{c_n}=\delta_{m+n,0},\quad
  \acomm{b_m}{b_n}=\acomm{c_m}{c_n}=0,\quad m,n\in\ZZ+\epsilon.
\end{equation}
As in the previous sections on \(\osp\), \(\epsilon=0\) corresponds to the
\ns{} sector and \(\epsilon=\sfrac{1}{2}\) corresponds to the Ramond sector.

The \(bc\) ghost algebra \(\bclie{\epsilon}\) admits a triangular decomposition
\begin{equation}
  \bclie{\epsilon}^\pm=\bigoplus_{n\ge1}b_{\pm (n-\epsilon)}\oplus\bigoplus_{n\ge1}c_{\pm (n-\epsilon)},\quad 
  \bclie{0}^0=\CC b_0\oplus\CC c_0\oplus\CC\wun,\quad 
  \bclie{1/2}^0=\CC\wun.
\end{equation}

Up to isomorphism and parity reversal there is only one simple \(\bclie{0}^\ge=\bclie{0}^0\oplus\bclie{0}^+\)
module \(\CC\ket{\NS}\oplus\CC c_0\ket{\NS}\), on which \(\wun\) acts as the
identity and \(\bclie{0}^+\) acts trivially. The action of \(b_0\) satisfies \(b_0\ket{\NS}=0\) (this is
necessary for simplicity) and the global parity of the module is fixed by
assigning even parity to \(\ket{\NS}\). The Verma module,
\begin{equation}
  \NSFock=\Ind{\bclie{0}}{\bclie{0}^\ge}{\brac*{\CC\ket{\NS}\oplus\CC c_0\ket{\NS}}},
\end{equation}
together with its parity reversal are the only \ns{} Verma modules and both
are simple. They are
called \ns{} Fock spaces.

In the Ramond sector, via
\(\bclie{\sfrac{1}{2}}^\ge=\bclie{\sfrac{1}{2}}^0\oplus\bclie{\sfrac{1}{2}}^+\),
one has the Verma module 
\begin{equation}
  \RFock=\Ind{\bclie{\sfrac{1}{2}}}{\bclie{\sfrac{1}{2}}^\ge}{\CC\ket{\Ra}}
\end{equation}
where \(\CC\ket{\Ra}\) is the 1 dimensional \(\bclie{\sfrac{1}{2}}^\ge\)
module characterised by \(\ket{\Ra}\) having even parity and
\begin{equation}
  \wun\ket{\Ra}=\ket{\Ra},\quad \bclie{\sfrac{1}{2}}^+\ket{\Ra}=0.
\end{equation}
This Verma module and its parity reversal are simple and they are the only
Ramond Verma modules.

\begin{defn}
  The \emph{\(bc\) ghost \vsa{}} \(\bcva\) is the unique \vsa{} that is
  strongly generated by two odd fields \(b(z),c(z)\), has the defining \ope{s}
  \begin{equation}
    b(z)c(w)\sim -c(w)b(z)\sim\frac{1}{z-w},
  \end{equation}
  and satisfies no additional relations beyond those required by the \vsa{} axioms.
\end{defn}
As a \(\bclie{0}\) module \(\bcva\) is isomorphic to \(\NSFock\).

\subsection{Free field realisations of the universal \vosa{} \(\vosp{k}\)}
\label{sec:ospFFR}

In this section we consider two free field realisations of \(\vosp{k}\),
once as a subalgebra of \(\ffa=\hvoa{1}\otimes\bgva\otimes\bcva\) and once as
a subalgebra of \(\ffb=\hvoa{1}\otimes\lva{\theta+\psi}\otimes\bcva\).

\begin{prop}
  \label{thm:ffrhom}
  Let \(\xi\in\CC\setminus\{0\}\) and \(k=\frac{\xi^2-3}{2}\).
  There exists a \vosa{} homomorphism \(\vosp{k}\to \ffa\) uniquely characterised by the assignment
  \begin{align}
    \label{eq:ffrhom}
    e(z)&\mapsto \beta(z),\quad
    x(z)\mapsto b(z)+\normord{\beta(z) c(z)},\quad
    h(z)\mapsto \xi a(z)+\normord{b(z)c(z)}+2\normord{\beta(z)\gamma(z)},\nonumber\\
    y(z)&\mapsto (\xi^2-2)\partial
    c(z)+\xi\normord{a(z)c(c)}+\normord{\gamma(z)b(z)}+\normord{\beta(z)\gamma(z)c(z)},\nonumber\\
    f(z)&\mapsto-\xi\normord{a(z)\gamma(z)}-\normord{\beta(z)\gamma(z)\gamma(z)}
    -\normord{\gamma(z)b(z)c(z)}
    +\frac{1-\xi^2}{2}\normord{\partial c(z) c(z)}+\frac{3-\xi^2}{2}\partial\gamma(z).
  \end{align}
  We omit the tensor product symbols for brevity, identifying \(a\) with
  \(a\otimes\wun\otimes\wun\) and so on. The image of the Virasoro field under
  the above homomorphism is
  \begin{equation}
    T(z)\mapsto \frac{1}{2}\normord{a(z)a(z)} -\frac{1}{2\xi}\partial a(z)-\normord{\beta(z)\partial\gamma(z)}-\normord{b(z)\partial c(z)}.
  \end{equation}
\end{prop}
\begin{proof}
  This free field realisation appears to have first been considered by
  Bershadsky and Ooguri in \cite{BerHid89}.
  The validity of the proposition follows by checking, through direct
  calculation, that the generating fields of \(\vosp{k}\)
  and their images in \(\ffa\) satisfy the same \ope{s} and that the formulae
  for the Virasoro fields match. Therefore the above
  map is a well defined \vosa{} homomorphism.
\end{proof}
\begin{rmk}
  The admissible levels of \cref{thm:admlevels} at which \(\vosp{k}\) is not simple
  are realised by the free field realisation whenever \(\xi^2=\frac{u}{v}\),
  where \(u,v\) are the integer parameters of \cref{thm:admlevels}.
\end{rmk}

Since the above homomorphism is non-trivial, it follows that it is injective
for the levels at which \(\vosp{k}\) is simple. To show that the homomorphism
remains injective at the
levels for which \(\vosp{k}\) is not simple, 
we need to take a closer look at \(\ffa\) modules, which are
automatically also \(\vosp{k}\) modules due to the homomorphism
\(\vosp{k}\to\ffa\).
Let \(\NSfket{p}=\ket{p}\otimes\Omega\otimes\ket{\NS}\),
\(\NSfket{p;j}=\ket{p}\otimes u_j\otimes\ket{\NS}\),
\(\Rfket{p}=\ket{p}\otimes\Omega\otimes\ket{\Ra}\) and
\(\Rfket{p;j}=\ket{p}\otimes u_j\otimes\ket{\Ra}\) respectively be relaxed
highest weight vectors in
\begin{equation}
  \NSFFaFock{p}=\Fock{p}\otimes \bgh\otimes\NSFock,\quad
  \NSFFaFock{p;\sqbrac*{j}}=\Fock{p}\otimes \bgd{\sqbrac*{j}}\otimes \NSFock,\quad
  \RFFaFock{p}=\Fock{p}\otimes\bgh\otimes\RFock\quad 
  \RFFaFock{p;\sqbrac*{j}}=\Fock{p}\otimes\bgd{\sqbrac*{j}}\otimes\RFock,
\end{equation}
where \(\Omega\) is the highest weight vector of \(\bgh\) and  \(u_j\) a relaxed
highest weight vector of \(\bgd{\sqbrac*{j}}\) of weight \(j\).
Let
\begin{align}
  \lambda_p^\NS&=\xi p,\qquad 
  \lambda_{p;j}^\NS=\xi p+2j,\qquad 
  \lambda_p^\Ra=\xi p-\frac{1}{2},\qquad 
  \lambda_{p;j}^\Ra=\xi p-\frac{1}{2}+2j,\nonumber\\
  s_p&=\xi p+\frac{1}{2},\qquad 
  q_p=\frac{\brac*{\xi p+\frac{1}{2}}^2-1}{2}.
\end{align}
 From the \vosa{} homomorphism
\eqref{eq:ffrhom} it then follows that
\begin{align}
  h_0\NSfket{p}&=
  \lambda_p^\NS\NSfket{p},&
  \Sigma \NSfket{p}& 
  =s_p\NSfket{p}\nonumber\\
  h_0 c_0\NSfket{p}&
  =\brac*{\lambda^\NS_p-1} c_0\NSfket{p},&
  \Sigma c_0\NSfket{p}&
  =-s_p c_0\NSfket{p}\nonumber\\
  h_0 \NSfket{p;j}& 
  =\lambda_{p;j}^\NS\NSfket{p;j},
  &\Sigma \NSfket{p;j}&
  =s_p \NSfket{p;j}\nonumber\\
  h_0 c_0\NSfket{p;j}&
  = \brac*{\lambda_{p;j}^\NS-1}c_0\NSfket{p;j},
  &\Sigma c_0\NSfket{p;j}&
  =-s_p c_0\NSfket{p;j}\nonumber\\
  h_0\Rfket{p}& 
  =\lambda_p^\Ra\Rfket{p},&
  Q\Rfket{p}& 
  =q_p\Rfket{p},\nonumber\\
  h_0\Rfket{p;j}&
  =\lambda^\Ra_{p;j}\Rfket{p;j},&
  Q\Rfket{p;j}&
  =q_p\Rfket{p;j},
\end{align}
where \(\Sigma=x_0y_0-y_0x_0+\sfrac{1}{2}\) and \(Q=\sfrac{h_0^2}{2}+e_0f_0+f_0e_0\).
Note that while the eigenvalue of \(h_0\) depends on the \(\beta\gamma\)
weight of the relaxed highest vector, the eigenvalues of \(\Sigma\) and \(Q\)
do not.

\begin{lem}\leavevmode
  \label{thm:ffainj}
  \begin{enumerate}
  \item\label{itm:ffainj1} The \vosa{} homomorphism of \cref{thm:ffrhom} is injective for all
    \(k\in\CC\setminus\{-\sfrac{3}{2}\}\) and therefore any singular vectors
    of \(\vosp{k}\) will have non-trivial image. 
  \item\label{itm:ffainj2} For admissible levels \(k=\frac{u-3v}{2v}\), the
    singular vector at \(\osp\) weight \((u-1)\alpha\)
    and conformal weight \((u-1)\frac{v}{2}\) is unique, up to rescaling, in
    \(\ffa\), where \(\ffa\) is regarded as a \(\vosp{k}\) module.
  \end{enumerate}
\end{lem}
\begin{proof}
  To show part \ref{itm:ffainj1}, consider the formulae for the \(h_0\)
  eigenvalues of \(\NSfket{p}\). It is clear that \(\NSfket{p}\) is an
  \(\aosp{0}\) highest weight vector and as one varies 
  \(p\), there are no restrictions on the \(\osp\) weights that can be
  obtained. However, for the affine minimal model \vosa{s} \(\MinMod{u}{v}\)
  not all \(\osp\) weights are allowed as highest weights. For example,
  \((u-1)\alpha\), the weight of the singular vector, is not allowed. Thus,
  the image of the homomorphism must be isomorphic to \(\vosp{k}\) and not just
  a quotient thereof.

  We show part \ref{itm:ffainj2} by contradiction. In order to do so we
  prepare some results from \cite{IoOsp01} on the submodule structure and singular vectors of
  \(\aosp{0}\) Verma modules.
  We denote by \(\Ver{\lambda}\) the Verma at admissible level
  \(k_{u,v}\) generated by a highest weight vector of \(\osp\)
  weight \(\lambda\alpha\) and by \(\Irr{\lambda}\) its simple quotient by its
  unique maximal proper submodule. Specialising \cite[Theorem 3.1]{IoOsp01} to
  the case at hand shows that the last few terms of the BGG resolution of
  \(\Irr{0},\ \Irr{-1}\) and \(\Irr{u-1}\) are given by
  \begin{align} 
    \cdots&\lra\Ver{-u-1}\oplus\Ver{2u-1}\lra\Ver{-u}\oplus\Ver{u}\lra\Ver{-1}\oplus\Ver{u-1}\lra\Ver{0}\lra\Irr{0}\lra0,\nonumber\\
    \cdots&\lra\Ver{-u-1}\oplus\Ver{2u-1}\lra\Ver{-u}\oplus\Ver{u}\lra\Ver{-1}\lra\Irr{-1}\lra0,\nonumber\\
    \cdots&\lra\Ver{-u-1}\oplus\Ver{2u-1}\lra\Ver{-u}\oplus\Ver{u}\lra\Ver{u-1}\lra\Irr{u-1}\lra0.
  \end{align}
  These resolutions imply that \(\Ver{0}\) has two
  independent singular vectors\footnote{Independent here meaning that the singular vector is a
    descendant of the highest weight vector only and not of other singular vectors.}
  whose \(\osp\) and conformal weights are \((-\alpha,0)\) and
  \(((u-1)\alpha,(u-1)\frac{v}{2})\), respectively. These two singular vectors
  each generate a Verma submodule, and both of these submodules share a pair of independent singular
  vectors whose \(\osp\) and conformal weights are \((-u\alpha,(u-1)\frac{v}{2})\) and
  \((u\alpha,(u+1)\frac{v}{2})\), respectively. These two additional singular vectors
  each generate a Verma submodule, and again both of these submodules share a pair of independent singular
  vectors whose \(\osp\) and conformal weights are \((-(u+1)\alpha,(u+1)\frac{v}{2})\) and
  \(((2u-1)\alpha,(u+1)\frac{v}{2})\), respectively, and so on. As an
  \(\aosp{}\) module \(\vosp{k_{u,v}}\) is isomorphic to
  \(\Ver{0}/\Ver{-1}\) and, as stated in \cref{thm:admlevels}, this
  quotient has a unique singular vector at  \(\osp\) and conformal weight
  \(((u-1)\alpha,(u-1)\frac{v}{2})\). In part \ref{itm:ffainj1}, we showed
  that \(\vosp{k_{u,v}}\) is isomorphic to the \(\aosp{}\) submodule of \(\ffa\)
  generated by the vector \(\NSfket{0}\). Thus,
  \(\UEA{\aosp{}}\NSfket{0}\cong \Ver{0}/\Ver{-1}\).
  A direct calculation 
  shows that \(x_0 c_0\NSfket{0}=\NSfket{0}\) and therefore,
  \(\UEA{\aosp{}}c_0\NSfket{0}/\UEA{\aosp{}}\NSfket{0}\) is an \(\aosp{}\)
  highest weight module of weight \(-\alpha\) and is therefore isomorphic to a
  quotient of \(\Ver{-1}\).

  Due to the weights at which the singular
  vectors of \(\Ver{-1}\) appear, the \(((u-1)\alpha,(u-1)\frac{v}{2})\)
  weight spaces of \(\Ver{-1}\) and \(\Irr{-1}\) have the same
  dimension. This implies that neither \(\Ver{-1}\) nor any quotient of
  \(\Ver{-1}\) contains a singular vector in the
  \(((u-1)\alpha,(u-1)\frac{v}{2})\) weight space. Further, since both
  \(\ffa\) and \(\Ver{0}\) have the same 
  characters as \(\aosp{}\) modules, their weight spaces also have the same
  dimensions. Thus, the weight spaces of \(\ffa\) have the same
  dimensions as those of the direct sum \(\brac*{\Ver{0}/\Ver{-1}} \oplus
  \Ver{-1}\). Finally, this implies that the dimensions of the
  \(((u-1)\alpha,(u-1)\frac{v}{2})\) weight spaces of \(\ffa\) and
  \(\UEA{\aosp{0}}c_0\NSfket{0}\) are equal, since the dimension of the
  \(((u-1)\alpha,(u-1)\frac{v}{2})\) weight spaces of \(\Ver{-1}\) and
  \(\Irr{-1}\) are.

  Assume that the \(\aosp{0}\) singular vector of weight
  \(((u-1)\alpha,(u-1)\frac{v}{2})\) in \(\ffa\) is not unique, then the
  quotient module
  \(\UEA{\aosp{0}}c_0\NSfket{0}/\UEA{\aosp{0}}\NSfket{0}\) would have to
  contain at least one singular vector at this weight, because the weight space
  of \(\ffa\) of weight
  \(((u-1)\alpha,(u-1)\frac{v}{2})\) 
  is contained in
  \(\UEA{\aosp{0}}c_0\NSfket{0}\) by the above dimension counting
  arguments. However, neither
  \(\Ver{-1}\) nor any of its quotients contain a singular vector at this
  weight, so in particular,
  neither can the quotient module \(\UEA{\aosp{0}}c_0\NSfket{0}/\UEA{\aosp{0}}\NSfket{0}\).
\end{proof}

The second free field realisation of \(\vosp{k}\) is constructed by embedding the
\(\beta\gamma\) \va{} of the first realisation into a lattice algebra as in \cref{thm:bosonisation}.

\begin{lem}\leavevmode
  \label{thm:bosinj}
  \begin{enumerate}
  \item\label{itm:bosinj1} The composition of the algebra homomorphisms in
    \cref{thm:bosonisation} and
  \cref{thm:ffrhom} is injective. Therefore any singular vectors of
  \(\vosp{k}\) will have non-trivial image in \(\ffb\).
  \item\label{itm:bosinj2} For admissible levels \(k=\frac{u-3v}{2v}\), the singular vector at \(\osp\) weight \((u-1)\alpha\)
    and conformal weight \((u-1)\frac{v}{2}\) is unique, up to rescaling, in
    \(\ffb\).
  \end{enumerate}
\end{lem}
\begin{proof}
  Part \ref{itm:bosinj1} follows from the fact that the two algebra
  homomorphisms in \cref{thm:bosonisation} and \cref{thm:ffrhom} are both
  injective and thus so is their composition.

  We show part \ref{itm:bosinj2} by contradiction. As a module over \(\ffa\)
  the \vsa{}
  \(\ffb\) is isomorphic to \(\Fock{0}\otimes \sigma\bgd{0}^-\otimes
  \NSFock\). Further, since \(\sigma\bgd{0}^-\) satisfies the non-split exact sequence
  \begin{equation}
    0\lra\bgh\lra\sigma\bgd{0}^-\lra\sigma\bgh\lra0,
  \end{equation}
  \(\ffb\) satisfies the non-split exact sequence:
  \begin{equation}
    0\lra\ffa\lra\ffb\lra\Fock{0}\otimes\sigma\bgh\otimes\NSFock\lra0.
  \end{equation}
  Since the singular vector at \(\osp\) and conformal weights
  \((u-1)\alpha\) and \((u-1)\frac{v}{2}\) is unique in \(\ffa\), it is unique in
  \(\ffb\) if and only if there is no \(\aosp{}\) singular vector at those
  weights in \(\Fock{0}\otimes\sigma\bgh\otimes\NSFock\). As shall shortly
  become apparent,
  no such singular vector exists because
  \(\Fock{0}\otimes\sigma\bgh\otimes\NSFock\) admits no \(\aosp{}\) singular vectors at
  any weight. A necessary condition for a vector to be singular is that it is
  annihilated by \(f_1\). We will show that no non-trivial such vector exists
  by using the fact that \(\gamma_1\) acts injectively on \(\sigma\bgh\) (this
  is a consequence of \(\gamma_0\) acting injectively on \(\bgh\)) to show that \(f_1\) acts injectively on
  \(\Fock{0}\otimes\sigma\bgh\otimes\NSFock\).
  The expansion of \(f_1\) in free field generators is
  \begin{align}
    f_1&=-\xi a_0 \gamma_1
    -\brac*{2\sum_{m\ge 2}\beta_{-m}\gamma_m+2\sum_{m\ge 0} \gamma_{-m}\beta_m+\beta_{-1}\gamma_1}\gamma_1
    -\normord{b\,c}_0\gamma_1-\frac{3-\xi^2}{2}\gamma_1+C\nonumber\\
    &=-\brac*{\xi a_0+2\normord{\beta\,\gamma}_0+\normord{b\,c}_0-\beta_{-1}\gamma_1}\gamma_1-\frac{3-\xi^2}{2}\gamma_1+C,
  \end{align}
  where \(C\) denotes all summands not containing \(\gamma_1\), while
  \(\normord{b\,c}_0\) denotes the zero mode of the normally ordered product
  \(\normord{b(z)c(z)}\) and \(\normord{\beta\,\gamma}_0\) denotes the zero
  mode of the normally ordered product \(\normord{\beta(z)\gamma(z)}\). Next,
  we refine the grading of \(\Fock{0}\otimes\sigma\bgh\otimes\NSFock\) by
  Heisenberg, \(\beta\gamma\) and \(bc\) weights to also include the
  eigenvalue of \(\beta_{-1}\gamma_1\) (on \(\sigma\bgh\) the eigenvalues
  of \(\beta_{-1}\gamma_1\) are all integers). When acting on a homogeneous
  vector with respect to this refined grading, the terms collected in \(C\) in the
  expansion of \(f_1\) either do not change the eigenvalue of
  \(\beta_{-1}\gamma_1\) or increase the eigenvalue by a positive integers if
  they contain factors of \(\beta_{-1}\). Conversely, the summands containing \(\gamma_1\) shift the
  eigenvalue of \(\beta_{-1}\gamma_1\) by \(-1\). Let \(w\) be a
  homogeneous vector with respect to the refined grading, then
  \begin{align}
    f_1 w&= \brac*{\xi
      a_0+2\normord{\beta\,\gamma}_0+\normord{b\,c}_0-\beta_{-1}\gamma_1}\gamma_1w
    -\frac{3-\xi^2}{2}\gamma_1w+Cw\nonumber\\
    &=
    \brac*{m+2-\frac{u+v}{2v}} \gamma_1 w+C w,
  \end{align}
  where \(m\in \ZZ\) is the eigenvalue of \(\xi
  a_0+2\normord{\beta\,\gamma}_0+\normord{b\,c}_0-\beta_{-1}\gamma_1\) and we
  have used the fact that \(\xi^2=\frac{u}{v}\). Since \(v\) and
  \(\frac{u+v}{2}\) are coprime, \(\frac{u+v}{2v}\) is not an integer and the
  coefficient of \(\gamma_1 w\) is therefore non-zero. Further, \(\gamma_1 w\)
  and \(C w\) have different \(\beta_{-1}\gamma_1\) eigenvalues, so they
  cannot be linearly dependent. Thus, \(f_1 w\neq 0\) for any homogeneous
  vector of the refined grading and \(f_1\) acts injectively.
\end{proof}

The second free field realisation implies that modules over \(\ffb\) are also
modules over \(\vosp{k}\) by restriction. We introduce the following notation
\begin{align}
  \NSFFbFock{p}&=\Fock{p}\otimes \Fock{\sqbrac*{0}}\otimes\NSFock, &
  \NSFFbFock{p;\sqbrac*{\lambda},n}&=\Fock{p}\otimes \Fock{\sqbrac*{\lambda\brac*{\psi+\theta}+n\psi}}\otimes\NSFock,\nonumber\\
  \RFFbFock{p}&=\Fock{p}\otimes \Fock{\sqbrac*{0}}\otimes\RFock, &
  \RFFbFock{p;\sqbrac*{\lambda},n}&=\Fock{p}\otimes \Fock{\sqbrac*{\lambda\brac*{\psi+\theta}+n\psi}}\otimes\RFock.
\end{align}

We now turn screening operators --- original introduced by Dotsenko and Fateev in
the context of the Coulomb gas formalism \cite{DotScr84} --- as a means of
constructing \(\vosp{k}\) \sv{s}.
\begin{defn}
  Let \(\VOA{V}\) be a \vosa{} together a free field realisation, \(\VOA{V}\ira\VOA{W}\), in a free
  field \vosa{} \(\VOA{W}\). A \emph{screening field} is a field corresponding
  to a vector in a module over \(\VOA{W}\), whose \ope{s} with the fields of
  \(\VOA{V}\) have singular parts that are total derivatives. It suffices to
  check this for the generating fields of \(V\).
\end{defn}

\begin{prop}
  The field
  \begin{equation}
    \label{eq:scr1}
    \scrf{1}{z}=\brac*{\beta(z)c(z)-b(z)}\vop{\frac{-a}{\xi}}{z}
  \end{equation}
  is a screening field for the free field realisations \(\vosp{k}\ira \ffa\)
  and also for \(\vosp{k}\ira\ffb\) when \(\beta(z)\) is realised as
  \(\beta(z)=\vop{\theta+\psi}{z}\).
  The field
  \begin{equation}
    \label{eq:scr2}
    \scrf{2}{z}=\brac*{\beta(z)c(z)-b(z)}\vop{\xi a-\frac{1+\xi^2}{2}(\theta+\psi)}{z}
  \end{equation}
  is a screening field for \(\vosp{k}\ira\ffb\), where again \(\beta(z)=\vop{\theta+\psi}{z}\).
\end{prop}
\begin{proof}
  The proposition follows by direct computation, that is, computing the
  \ope{s} of the generators of \(\vosp{k}\) with the screening fields
  \(\scrf{1}{z}\) and \(\scrf{2}{z}\) and verifying that they are indeed total
  derivatives. For \(\scrf{1}{z}\) the non-zero \ope{s} are
  \begin{align}
    y(z)\scrf{1}{w}&\sim\partial_w\frac{\xi^2\vop{\frac{-a}{\xi}}{w}}{z-w},\qquad
    f(z)\scrf{1}{w}\sim\partial_w\frac{\xi^2c(w)\vop{\frac{-a}{\xi}}{w}}{z-w},
  \end{align}
  while for \(\scrf{2}{z}\) they are
  \begin{align}
     y(z)\scrf{2}{w}&\sim\partial_w\frac{\vop{\xi
        a-\frac{1+\xi^2}{2}\brac*{\theta+\psi}}{w}}{z-w},
    \
    f(z)\scrf{2}{w}\sim\partial_w\frac{\frac{1+\xi^2}{2}b(w)\vop{\xi
        a-\frac{3+\xi^2}{2}\brac*{\theta+\psi}}{w}
      +
      \frac{1-\xi^2}{2}c(w)\vop{\xi
        a-\frac{1+\xi^2}{2}\brac*{\theta+\psi}}{w}
    }{z-w}.
  \end{align}
\end{proof}

\begin{rmk}
  While screening operators of the form \eqref{eq:scr1} are well known, those
  of the form \eqref{eq:scr2} are less commonly considered. They have appeared
  in the physics literature as formal non-integer powers of screening
  operators, that is, through identities of the form
  ``\(\scrf{2}{z}=\scrf{1}{z}^{-\xi^2}\)'' see for example \cite{BerSLn89}. In
  \cite{Petadsl95} such non-integer powers were interpreted as formal
  quantities which were to be expanded within normally ordered
  products in a way that avoids non-integer powers.
  
  Here, instead, we enlarge the \(\beta\gamma\) \va{}, by embedding it into a
  lattice \va{}, through a free field realisation. In this larger \va{}
  non-integer powers of \(\beta(z) = \vop{\theta+\psi}{z}\) can be defined as 
  \(\beta(z)^\lambda=\vop{\lambda\brac*{\theta+\psi}}{z}\) which is well
  defined because \(\theta+\psi\) has norm 0 and thus
  \(\vop{\lambda\brac*{\theta+\psi}}{z}\vop{\mu\brac*{\theta+\psi}}{z}=
  \vop{\brac*{\lambda+\mu}\brac*{\theta+\psi}}{z}\) for all \(\lambda,\mu\in \CC\).
\end{rmk}

The interest in screening fields, for a free field realisation
\(\VOA{V}\ira\VOA{W}\), stems from the fact that their residues, when well
defined, commute with the action of \(\VOA{V}\). These residues, referred to
as \emph{screening operators}, thus define \(\VOA{V}\) module
homomorphisms. This implies that the image of any singular vector will again
be singular (or zero). Thus screening fields and operators are an invaluable
tool for constructing singular vectors in \(\VOA{V}\) modules.

Consider the composition of \(\ell\) copies of the screening field \(\scrf{1}{z}\):
\begin{gather}
  \label{eq:scrflds}
  \scrf{1}{z_1}\cdots\scrf{1}{z_\ell}=\brac*{\beta(z_1)c(z_1)-b(z_1)}\cdots\brac*{\beta(z_\ell)c(z_\ell)-b(z_\ell)}
  \ee^{-\sfrac{\ell}{\xi}a}
  \prod_{1\leq
    i<j\leq\ell}\brac*{z_i-z_j}^{\xi^{-2}}
  \prod_{i=1}^\ell z_i^{-\frac{a_0}{\xi}}\nonumber\\
  \cdot
  \prod_{m\geq1}\sqbrac*{
    \exp\brac*{\frac{-1}{\xi}\frac{a_{-m}}{m}\sum_{i=1}^\ell z_i^m}
    \exp\brac*{\frac{1}{\xi}\frac{a_{m}}{m}\sum_{i=1}^\ell z_i^{-m}}
    }
\end{gather}
The action of screening fields and operators is most easily evaluated using
methods coming from the theory of symmetric functions, however, products of
fermions such as those above are skew symmetric. This obstacle can be 
overcome by factoring out the Vandermonde determinant \(\van{z}=\prod_{1\le i<j\le\ell}\brac*{z_i-z_j}\):
\begin{equation}
  \prod_{1\le i<j\le \ell}\brac*{z_i-z_j}^{\xi^{-2}}
  =\van{z}\prod_{1\le i<j\le \ell}\brac*{z_i-z_j}^{\xi^{-2}-1}
  =\van{z}\prod_{1\le i\neq j\le \ell}\brac*{z_i-z_j}^{\frac{\xi^{-2}-1}{2}},
\end{equation}
where we have suppressed a complex phase in the second equality which will later be
absorbed into integration cycles. This allows us to rewrite the product of
screening fields \eqref{eq:scrflds} in the form
\begin{multline}
  \scrf{1}{z_1}\cdots\scrf{1}{z_\ell}=
  \ee^{-\sfrac{\ell}{\xi}a}
  \prod_{1\leq i\neq j\leq\ell}\brac*{1-\frac{z_i}{z_j}}^{\frac{\xi^{-2}-1}{2}}
  \prod_{i=1}^\ell z_i^{-\frac{a_0}{\xi}+\brac*{\ell-1}\frac{\xi^{-2}-1}{2}}\\
  \cdot
  \van{z}\brac*{\beta(z_1)c(z_1)-b(z_1)}\cdots\brac*{\beta(z_\ell)c(z_\ell)-b(z_\ell)}
  \prod_{m\geq1}\sqbrac*{
    \exp\brac*{\frac{-1}{\xi}\frac{a_{-m}}{m}\sum_{i=1}^\ell z_i^m}
    \exp\brac*{\frac{1}{\xi}\frac{a_{m}}{m}\sum_{i=1}^\ell z_i^{-m}}
  },
  \label{eq:vandscrs}
\end{multline}
where the skew symmetry of the fermions is now countered by that of the Vandermonde determinant.

To define screening operators as integrals of products of screening fields,
there need to exist cycles over which to integrate. The obstruction to the
existence of such cycles lies in the multivaluedness of the second product on
the \rhs{} of \eqref{eq:vandscrs}. If the exponent
\(-\frac{a_0}{\xi}+\brac*{\ell-1}\frac{\xi^{-2}-1}{2}\) evaluates to an
integer, when acting on a \ns{} free field module, then there exists such a
cycle \(\cyc{\ell}{\xi}\), generically unique in homology up to normalisation
and constructed in \cite{TsuFoc86}. For example, on the free field module \(\Fock{p}\otimes
\bgd{\sqbrac*{j}}\otimes \NSFock\), \(a_0\) acts as \(p\) and the exponent evaluates to
\(-\frac{p}{\xi}+\brac*{\ell-1}\frac{\xi^{-2}-1}{2}\). The cycles
\(\cyc{\ell}{\xi}\) are homologically equivalent to the cycles over which one
integrates in the theory of symmetric polynomials to define inner products ---
see \cite[Sec. 3]{TsuExt13} for details. The actual construction of the cycles
\(\cyc{\ell}{\xi}\) is rather subtle and we refer the interested reader to
\cite{TsuFoc86} for the complete picture.

For completeness, we mention that when acting on a Ramond free field module,
the cycles \(\cyc{\ell}{\xi}\) exist when
\(-\frac{a_0}{\xi}+\brac*{\ell-1}\frac{\xi^{-2}-1}{2}\) evaluates to a half
integer (to compensate for the half integer exponents of free fermion
fields). Screening operators between Ramond modules shall, however, not be
needed in what follows.

Analogously, composing \(\ell\) copies of the screening field \(\scrf{2}{z}\)
yields
\begin{multline}
  \scrf{2}{z_1}\cdots\scrf{2}{z_\ell}=
  \ee^{\ell \xi a-\ell\frac{1+\xi^2}{2}\brac*{\theta+\psi}}
  \prod_{1\leq i\neq j\leq\ell}\brac*{1-\frac{z_i}{z_j}}^{\frac{\xi^{2}-1}{2}}
  \prod_{i=1}^\ell z_i^{\xi a_0+\brac*{\ell-1}\frac{\xi^{2}-1}{2}-\frac{1+\xi^2}{2}\brac*{\theta_0+\psi_0}}
  \\
  \cdot\van{z}\brac*{\beta(z_1)c(z_1)-b(z_1)}\cdots\brac*{\beta(z_\ell)c(z_\ell)-b(z_\ell)}
  \prod_{m\geq1}\sqbrac*{
    \exp\brac*{\xi\frac{a_{-m}}{m}\sum_{i=1}^\ell z_i^m}
    \exp\brac*{-\xi\frac{a_{m}}{m}\sum_{i=1}^\ell z_i^{-m}}
    }
  \\
  \cdot\prod_{m\geq1}\sqbrac*{
    \exp\brac*{-\frac{1+\xi^2}{2}\frac{\theta_{-m}+\psi_{-m}}{m}\sum_{i=1}^\ell z_i^m}
    \exp\brac*{\frac{1+\xi^2}{2}\frac{\theta_{m}+\psi_m}{m}\sum_{i=1}^\ell z_i^{-m}}
    },
    \label{eq:vandscrs2}
\end{multline}

\begin{defn}
  Let \(\ell\) and \(m\) be integers, and \(\ell\ge1\).
  \begin{enumerate}
  \item For \(p=\frac{1}{2\xi}\brac*{\ell-2m-1}-\frac{\xi}{2}\brac*{\ell-1}\),
    the screening operators \(\scrs{1}{\ell}:\NSFFaFock{p}\to
    \NSFFaFock{p-\frac{\ell}{\xi}}\) and \(\scrs{1}{\ell}:\NSFFaFock{p;j}\to
    \NSFFaFock{p-\frac{\ell}{\xi};j}\) 
    are defined by
    \begin{equation}
      \scrs{1}{\ell}=\int_{\cyc{\ell}{\xi}}\scrf{1}{z_1}\cdots\scrf{1}{z_\ell}\dd
      z_1\cdots \dd z_\ell,
    \end{equation}
    meaning that the cycle \(\cyc{\ell}{\xi}\) exists.
  \item For \(p=\frac{1}{2\xi}\brac*{\ell+2m-1}-\frac{\xi}{2}\brac*{\ell-1}\),
    the screening operators \(\scrs{2}{\ell}:\NSFFbFock{p}\to
    \NSFFbFock{p+\ell\xi;\sqbrac*{-\ell\frac{1+\xi^2}{2}},0}\) and \(\scrs{2}{\ell}:\NSFFbFock{p;\sqbrac*{\lambda},n}\to
    \NSFFbFock{p+\ell \xi;\sqbrac*{\lambda-\ell\frac{1+\xi^2}{2}},n}\) are defined by
    \begin{equation}
      \scrs{2}{\ell}=\int_{\cyc{\ell}{\xi^{-1}}}\scrf{2}{z_1}\cdots\scrf{2}{z_\ell}\dd
      z_1\cdots \dd z_\ell,
    \end{equation}
    meaning that the cycle \(\cyc{\ell}{\xi^{-1}}\) exists.
  \end{enumerate}
  We normalise these cycles such that
  \begin{equation}
    \int_{\cyc{\ell}{\xi^{\pm1}}}\prod_{1\leq i\neq
      j\leq\ell}\brac*{1-\frac{z_i}{z_j}}^{\frac{\xi^{\mp2}-1}{2}}\frac{\dd
      z_1\cdots \dd z_\ell}{z_1\cdots z_\ell}=1.
  \end{equation}
\end{defn}
We shall lighten notation in what follows by suppressing the cycle
\(\cyc{\ell}{\xi^{\pm1}}\) in all integrals.

\begin{rmk}
  As mentioned previously, the 
  \(\van{z}\brac*{\beta(z_1)c(z_1)-b(z_1)}\cdots\brac*{\beta(z_\ell)c(z_\ell)-b(z_\ell)}\) factor
  and the factors of exponential functions
  that appear on the \rhs{} of \eqref{eq:vandscrs} and \eqref{eq:vandscrs2}
  are both invariant under
  permuting the \(z_i\). The action of the screening operators
  \(\scrs{1}{\ell}\) and \(\scrs{1}{\ell}\) can thus be evaluated using the well studied family of
  inner products of symmetric polynomials defined by
  \begin{equation}\label{eq:intinprod}
    \jprod{f,g}{\ell}{t}=\int \prod_{1\leq i\neq
      j\leq\ell}\brac*{1-\frac{z_i}{z_j}}^{1/t} 
    f(z_1^{-1},\dots,z_\ell^{-1})g(z_1,\dots,z_\ell)
    \frac{\dd z_1\cdots \dd z_\ell}{z_1\cdots z_\ell},
  \end{equation}
  where \(f\) and \(g\) are symmetric polynomials and
  \(t\in\CC\setminus\{0\}\). One of the characterising properties of Jack symmetric polynomials
  \(\fjack{\lambda}{t}{z}\) is that they are orthogonal with respect to these inner products
  labelled by \(t\). For readers unfamiliar with Jack symmetric polynomials we
  recommend Macdonald's comprehensive book \cite{MacSym95}, while for readers familiar with Jack
  symmetric polynomials, a short summary of their properties and the notation
  used here can be found in \cite[App.~A]{RidRel15}.
\end{rmk}

\subsection{Free field correlation functions}
\label{sec:ffcorr}

In this section we derive identities involving the correlation
functions of  the Heisenberg, \(\beta\gamma\) and \(b c\) fields. These will be
needed for evaluating the action of the zero modes of 
singular vectors in \cref{sec:0modecalc}.

Let \(\Fock{p}^*\) be the graded dual of the \hw{} $\hvoa{r}$-module
\(\Fock{p}\), \(p\in \heisvs{r}\). Then, \(\Fock{p}^*\) is a \lw{} right $\hlie{r}$-module generated by a \lwv{} $\brab{p}$ satisfying 
\begin{equation}
  \braketb{p}{p} = 1, \quad \brab{p} \hlie{r}^- = 0.
\end{equation}
It is convenient to extend the domain of the functionals in \(\Fock{p}^*\) to all Fock spaces \(\Fock{q}\), \(q\in\heisvs{r}\), but to have them act trivially unless \(q = p\).

\begin{defn}
  Let \(B\) be any combination of normally ordered products of free bosons
  \(a(z)\), vertex operators \(\vop{p}{z}\) and their derivatives.  The
  \emph{free boson correlation function} in $\Fock{p}$ and \(\Fock{q}^\ast\) is then defined to be $\bracketb{q}{B}{p}$.
\end{defn}
The operator product expansion formulae \eqref{eq:vertope} then imply the
following well known correlation function formulae.
\begin{prop}
  The correlation function of \(k\) vertex operators is given by
  \begin{equation}
    \bracketb{q}{\vop{p_1}{z_1}\cdots \vop{p_k}{z_k}}{p}=\delta_{q,p_1+\cdots+p_k+p}
    \prod_{\mathclap{1\le i<j\le k}}\:(z_i-z_j)^{(p_i, p_j)}\cdot\prod_{i=1}^kz_i^{(p,p_i)}.
  \end{equation}
\end{prop}

Let \(\dbgh\) be the graded dual of the highest weight \(\bglie\)
module \(\bgh\). Then, \(\dbgh\) is a \lw{} right \(\bglie\) module
generated by a \lw{} vector \(\phi\) satisfying
\begin{equation}
  \braket{\phi}{\Omega}=1,\quad 
  \phi\, \bglie^{-}=0.
\end{equation}
Further, let \(\bgd{q}^\ast\) be the graded dual of the highest weight \(\bglie\)
module \(\bgd{q}\). Then, \(\bgd{q}^\ast\) is a \lw{} right \(\bglie\) module
generated by a \lw{} vector \(\phi_q\) satisfying
\begin{equation}
  \braket{\phi_q}{u_q}=1,\quad 
  \phi_q\, \bglie^{-}=0.
\end{equation}

\begin{defn}
  Let \(B\) be any combination of normally ordered products of \(\bgva\)
  fields. The \emph{bosonic ghost correlation functions} in \(\bgh\) and \(\bgd{q}\) are
    then respectively defined to be
    \begin{align}
      \corrfn{B}=\bracket{\phi}{B}{\Omega},\qquad
      \corrfn{B}^q=\bracket{\phi_q}{B}{u_q}.
    \end{align}
\end{defn}

\begin{lem}\label{thm:bgcor}
  Let \(n\) be a non-negative integer, \(q\in \CC\) and let \(\poch{q}{n}\) be the rising
  factorial \(\prod_{i=1}^n\brac*{q+i-1}\). The \(\beta,\gamma\) fields satisfy
  \begin{align}
    \corrfn{\gamma_1^n \beta(z_1)\cdots \beta(z_n)}&=n!\nonumber\\
    \corrfn{\gamma_0^n \beta(z_1)\cdots \beta(z_n)}^q&=\poch{q}{n}\prod_{i=1}^{n}z_i^{-1}.
  \end{align}
\end{lem}
\begin{proof}
  The formulae follow by induction after noting that
  \(\comm{\gamma_n}{\beta(z)}=\wun z^{-n}\), that \(\gamma_1\Omega=0\) and that
  \(\gamma_0\beta_0 u_q = q u_q\).
\end{proof}

Let \(\dNSFock\) be the graded dual of the highest weight \(\bclie{0}\)
module \(\NSFock\). Then, \(\dNSFock\) is a \lw{} right \(\bclie{0}\) module
generated by a \lw{} vector \(\bra{\NS}\) satisfying
\begin{equation}
  \braket{\NS}{\NS}=1,\quad \bracket{\NS}{b_0c_0}{\NS}=1,\quad
  \bra{\NS}\bclie{0}^{-}=0.
\end{equation}
Further, let \(\dRFock\) be the graded dual of the highest weight \(\bclie{\frac{1}{2}}\)
module \(\RFock\). Then, \(\dRFock\) is a \lw{} right \(\bclie{\frac{1}{2}}\) module
generated by a \lw{} vector \(\bra{\Ra}\) satisfying
\begin{equation}
  \braket{\Ra}{\Ra}=1,\quad  \bra{\Ra}\bclie{\frac{1}{2}}^{-}=0.
\end{equation}

\begin{defn}
  Let \(B\) be any combination of normally ordered products of \(\bcva\)
  fields. The \emph{fermionic ghost correlation functions} in \(\NSFock\) and \(\RFock\) are
    then defined to be
    \begin{align}
      \NScorrfn{B}{+}=\bracket{\NS}{B}{\NS},\qquad
      \NScorrfn{B}{-}=\bracket{\NS}{b_0Bc_0}{\NS},\qquad
      \Rcorrfn{B}=\bracket{\Ra}{B}{\Ra}.
    \end{align}
\end{defn}

Correlation functions involving fermions can often be efficiently expressed
using pfaffians and this remains true for \(b c\) ghosts. The determinant of a
skew-symmetric matrix \(A=-A^\intercal\) can always
be written as the square of a polynomials in the coefficients of \(A\). This
polynomial, up to a choice of sign, is the pfaffian \(\pf(A)\) of \(A\). 
\begin{defn} \label{def:pf}
  Let \(A\) be a \(2n\times2n\) skew-symmetric matrix, so that \(A\) is uniquely determined by its upper-triangular entries \(A_{i,j}\), $i<j$. 
  We shall write $A = (A_{i,j})_{1 \le i < j \le 2n}$ to indicate a skew-symmetric matrix $A$ with given upper-triangular entries.
  After defining the pfaffian of the \(0\times0\) matrix to be $1$,
  the pfaffian of $A$ can be recursively defined by
    \begin{equation}\label{eq:recpf}
      \pf(A)=\sum_{\substack{j=1\\j\neq i}}^{2n}(-1)^{i+j+\theta(j-i)}
      A_{i,j}\pf(A_{\hat{\imath},\hat{\jmath}}),
    \end{equation}
    where the row index $i$ may be chosen arbitrarily, \(A_{\hat{\imath},\hat{\jmath}}\) denotes the matrix \(A\) with the
    \(i\)-th and \(j\)-th rows and columns removed, and
    \begin{equation}
      \theta(x)=
      \begin{cases*}
        1& if \(x>0\), \\
        0& if \(x<0\)
      \end{cases*}
    \end{equation}
    is the Heaviside step function.  In particular, $i=1$ gives the simplified formula
    \begin{equation} \label{eq:recpf'}
      \pf(A) = \sum_{j=2}^{2n} (-1)^j A_{1,j} \pf(A_{\hat{1}, \hat{\jmath}}).
    \end{equation}
\end{defn}

\begin{lem}\label{thm:bccorrs}
  Let \(n\) be a non-negative integer and \(m\in \ZZ\).
  The \(\beta\gamma\) and \(b c\) fields satisfy the following correlation function formulae.
  \begin{enumerate}
  \item In the \ns{} sector,
    \begin{subequations}
      \begin{align}
        \NScorrfn{\prod_{i=1}^{2n}\brac*{c(z_i)z_i^m-b(z_i)}}{+}&=
        (-1)^n\pf\brac*{\frac{z_i^m+z_j^m}{z_i-z_j}}_{1\le i<j\le 2n},
        \\
        \NScorrfn{\prod_{i=1}^{2n}\brac*{c(z_i)z_i^m-b(z_i)}}{-}&=
        (-1)^n\pf\brac*{\frac{z_i^{m+1}z_j^{-1}+z_i^{-1}z_j^{m+1}}{z_i-z_j}}_{1\le i<j\le 2n}.
      \end{align}
    \end{subequations}
  \item In the Ramond sector,
    \begin{align}
      \Rcorrfn{\prod_{i=1}^{2n}\brac*{c(z_i)z_i^m-b(z_i)}}&=
      (-1)^n\pf\brac*{\frac{z_i^{m+\frac{1}{2}}z_j^{-\frac{1}{2}}+z_i^{-\frac{1}{2}}z_j^{m+\frac{1}{2}}}{z_i-z_j}}_{1\le i<j\le 2n}.
    \end{align}
  \end{enumerate}
\end{lem}
\begin{proof}
  The above identities all follow by induction using the recursive definition
  of pfaffians. We give the details for the first identity and leave the rest
  as an exercise for the reader.

  The base step of the induction follows from the fact that for \(n=0\) the
  correlator is empty and therefore equal to 1. For the inductive step, assume
  the identity is true for \(n-1\) and consider
  \begin{align}
    &\NScorrfn{\prod_{i=1}^{2n}\brac*{c(z_i) z_i^{m}-b(z_i)}}{+}=
    \NScorrfn{\brac*{c(z_1) z_1^{m}-b(z_1)}\prod_{i=2}^{2n}\brac*{c(z_i)
        z_i^{m}-b(z_i)}}{+}\nonumber\\
    &\quad=
    \sum_{k\ge1} z_1^{m-k}\NScorrfn{c_k\prod_{i=2}^{2n}\brac*{c(z_i) z_i^{m}-b(z_i)}}{+} 
    -\sum_{k\ge0} z_1^{-k-1}\NScorrfn{b_k\prod_{i=2}^{2n}\brac*{c(z_i) z_i^{m}-b(z_i)}}{+}
    \nonumber\\
    &\quad=
    \sum_{j=2}^{2n}\brac*{-1}^{j-1}\NScorrfn{\prod_{\substack{i=2\\ i\neq
          j}}^{2n}\brac*{c(z_i) z_i^{-1}-b(z_i)}}{+}
    \sum_{k\ge0}\brac*{z_1^{m-k-1}z_j^{k}+z_1^{-k-1}z_j^{k+m}}\nonumber\\
    &\quad=
    \sum_{j=2}^{2n}\brac*{-1}^{j}\brac*{-1\frac{z_1^m+z_j^m}{z_1-z_j}}
    \NScorrfn{\prod_{\substack{i=2\\ i\neq
          j}}^{2n}\brac*{c(z_i) z_i^{-1}-b(z_i)}}{+}
    =\pf\brac*{-\frac{z_i^m+z_j^m}{z_i-z_j}}_{1\le i< j\le 2n}\nonumber\\
    &\quad=
    \brac*{-1}^n\pf\brac*{\frac{z_i^m+z_j^m}{z_i-z_j}}_{1\le i< j\le 2n},
  \end{align}
  where the third equality follows from the anti-commutation relations
  \(\acomm{c_k}{b(z)}=\wun z^{k-1}\) and \(\acomm{b_k}{c(z)}=\wun z^k\) and the fifth from
  the recursive definition of pfaffians.
\end{proof}

The above correlation functions involve fermions and so it is unsurprising
that they are skew-symmetric with respect to permuting the \(z_i\)
variables. This skew-symmetry can be compensated for by multiplying the above
correlation functions with the Vandermonde determinant. In addition to being symmetric, the products
of the above correlation functions with the Vandermonde determinant also
vanish whenever 3 variables coincide. Symmetric polynomials
with this property form an ideal of the ring of symmetric
polynomials. This ideal and generalisations thereof were studied by
by Feigin and Jimbo and Miwa and Mukhin \cite{Feidif02}. This ideal is
spanned by Jack symmetric functions at parameter value \(t=-3\) whose
associated partitions satisfy certain admissibility conditions. We refer the
reader to \cite[Section 3]{BloSVir16} for a detailed account of how to apply
Feigin and Jimbo and Miwa and Mukhin's work to correlation functions and for
the definitions of the notation used below. See in particular the notation
introduced in Definitions 3.1 and 3.4 for the admissible partitions
\(\admp{m}{n_1,n_2}\) and \(\uniqp{m}{n}\), and the involution of partitions
\([k-\uniqp{m}{n}]\) in Lemma 3.5 of \cite{BloSVir16}, as they shall be used
repeatedly in what follows.

Let \(B\) be some combination of \(\beta\gamma\) and \(bc\) fields, then we
denote the combined correlation functions by
  \(\NScorrfn{B }{q;\pm}=\NScorrfn{\corrfn{B}^q}{\pm}\), 
  \(\Rcorrfn{B }^{q}=\Rcorrfn{\corrfn{B}^q}\)
and we abbreviate \(\NScorrfn{\corrfn{B}}{\pm}\) as \(\NScorrfn{B}{\pm}\).
\begin{lem}\label{thm:correlid}
  Let \(n\) be a non-negative integer and \(q\in \CC\).
  \begin{enumerate}
  \item In the \ns{} sector,
    \begin{subequations}\label{eq:gamma1correls}
      \begin{align}
        \prod_{i=1}^{2n}z_i^{-2n+2} \van{z}
        \NScorrfn{\gamma_1^n\prod_{i=1}^{2n}\brac*{\beta(z_i)c(z_i)-b(z_i)}}{+}
        &=n!
        \prod_{i=1}^{2n}z_i^{-2n+2} \van{z}
        \NScorrfn{\prod_{i=1}^{2n}\brac*{c(z_i)-b(z_i)}}{+}\nonumber\\
        &=n! \brac*{-2}^n\fjack{\admp{2n}{0,0}}{-3}{z^{-1}},\\
        \prod_{i=1}^{2n+1}z_i^{-2n} \van{z}
        \NScorrfn{\gamma_1^nb_0\prod_{i=1}^{2n+1}\brac*{\beta(z_i)c(z_i)-b(z_i)}}{+}
        &=n!
        \prod_{i=1}^{2n+1}z_i^{-2n} \van{z}
        \NScorrfn{b_0\prod_{i=1}^{2n+1}\brac*{c(z_i)-b(z_i)}}{+}\nonumber\\
        &=
        -n! \brac*{-2}^{n+1}\fjack{\admp{2n}{2,0}}{-3}{z^{-1}},
      \end{align}
    \end{subequations}
    and
    \begin{subequations}\label{eq:corrformulae}
      \begin{align}
        &\prod_{i=1}^{2n}z_i^{-2n+1} \van{z} \NScorrfn{\gamma_0^n\prod_{i=1}^{2n}\brac*{\beta(z_i+w)c(z_i+w)-b(z_i+w)}}{q;\pm}\nonumber\\
        &\qquad= \poch{q}{n}\brac{-1}^n\prod_{i=1}^{2n}\brac*{z_i+w}^{-1}\sum_{i=0}^nw^{n-i}c^{(2n)}_i\fjack{[2n-1-\uniqp{2n}{i}]}{-3}{z^{-1}},\\
        &\prod_{i=1}^{2n+1}z_i^{-2n-1} \van{z}
        \NScorrfn{\gamma_0^{n+1}b_0\prod_{i=1}^{2n+1}\brac*{\beta(z_i+w)c(z_i+w)-b(z_i+w)}}{q;+}\nonumber\\
        &\qquad=
        \poch{q}{n+1}\brac*{-1}^{n+1}\prod_{i=1}^{2n+1}\brac*{z_i+w}^{-1}\sum_{i=0}^{n+1}w^{n+1-i}c^{(2n+1)}_i\fjack{[2n+1-\uniqp{2n+1}{i}]}{-3}{z^{-1}},\\
        &\prod_{i=1}^{2n+1}z_i^{-2n-1} \van{z}
        \NScorrfn{\gamma_0^{n}c_0\prod_{i=1}^{2n+1}\brac*{\beta(z_i+w)c(z_i+w)-b(z_i+w)}}{q;-}\nonumber\\
        &\qquad =
        \poch{q}{n}\brac*{-1}^{n+1}\prod_{i=1}^{2n+1}\brac*{z_i+w}^{-1}\sum_{i=0}^nw^{n+1-i}c^{(2n+1)}_i\fjack{[2n+1-\uniqp{2n+1}{i}]}{-3}{z^{-1}},
      \end{align}
    \end{subequations}
    where the \(c^{(k)}_i\in \CC\) are non-zero constants.
  \item In the Ramond sector,
    \begin{subequations}\label{eq:corrformulaeR}
      \begin{align}
        &\prod_{i=1}^{2n}z_i^{-2n+2}\van{z}\Rbgcorrfn{\gamma_0^n\prod_{i=1}^{2n}\brac*{\beta(z_i+w)c(z_i+w)-b(z_i+w)}}{q}\nonumber\\
        &\qquad
        =\poch{q}{n}(-2)^n\prod_{i=1}^{2n}\brac{z_i+w}^{-\frac{1}{2}}\fjack{\admp{2n}{0,0}}{-3}{z^{-1}},\\
        &\prod_{i=1}^{2n+1} z_i^{-2n+1}
        \van{z}\Rbgcorrfn{\gamma_0^n\prod_{i=1}^{2n+1}\brac*{\beta(z_i+w)c(z_i+w)-b(z_i+w)}\cdot c(w)}{q}\nonumber\\
        &\qquad
        =-\poch{q}{n}(-2)^nw^{\frac{1}{2}}\prod_{i=1}^{2n+1}\brac{z_i+w}^{-\frac{1}{2}}\fjack{\admp{2n+1}{0,0}}{-3}{z^{-1}}.
      \end{align}
    \end{subequations}
  \end{enumerate}
\end{lem}
\begin{proof}
  The main stepping stones to the above identities are the identities,
  proved in Proposition 3.8 of \cite{BloSVir16},
  \begin{align}\label{eq:pfaffjacks}
    \van{z_1,\dots,z_{2n}}\pf\brac*{\frac{1}{z_i-z_j}}_{1\leq i< j\leq 2n}
    &=\fjack{\admp{2n}{0,0}}{-3}{z_1,\dots,z_{2n}},\nonumber\\
    \van{z_1,\dots,z_{2n}}\pf\brac*{\frac{z_i+z_j}{z_i-z_j}}_{1\leq i< j\leq 2n}
    &=\fjack{\admp{2n}{1,0}}{-3}{z_1,\dots,z_{2n}};
  \end{align}
  the translation invariance of \(\jack{\admp{2n}{0,0}}{-3}\), that is,
  \begin{equation}
    \fjack{\admp{2n}{0,0}}{-3}{z_1+w,\dots,z_{2n}+w}=\fjack{\admp{2n}{0,0}}{-3}{z_1,\dots,z_{2n}};
  \end{equation}
  and the Taylor expansion
  \begin{equation}\label{eq:jacktaylor}
    \fjack{\admp{2n}{1,0}}{-3}{z_1+w,\dots,z_{2n}+w}=\sum_{i=0}^n c^{(2n)}_i
    \fjack{\uniqp{2n}{i}}{-3}{z_1,\dots,z_{2n}} w^{n-i}.
  \end{equation}
  The non-vanishing of the expansion coefficients \(c^{(2n)}_i\) was shown in
  the Remark directly after Proposition 3.9 of \cite{BloSVir16}. Verifying the
  above formulae for correlation functions now follows in exactly the same way
  as it did in Proposition 3.9 of \cite{BloSVir16}, so we give the details for
  first two cases of \eqref{eq:corrformulae} only and leave the rest as an
  exercise for the reader.

  When multiplying out the product
  \(\prod_{i=1}^{2n}\brac*{\beta(z_i)c(z_i)-b(z_i)}\) in
  \begin{equation}
    \NScorrfn{\gamma_0^n\prod_{i=1}^{2n}\brac*{\beta(z_i)c(z_i)-b(z_i)}}{q;\pm}
  \end{equation}
  the only summands which can contribute are those containing an equal number
  of \(b\) and \(c\) fields, that is, \(n\) of each. These summands will also contain
  \(n\) \(\beta\) fields and so applying the second identity of
  \cref{thm:bgcor} implies
  \begin{align}
    \NScorrfn{\gamma_0^n\prod_{i=1}^{2n}\brac*{\beta(z_i)c(z_i)-b(z_i)}}{q;\pm}&=
    \poch{q}{n}\NScorrfn{\prod_{i=1}^{2n}\brac*{c(z_i) z_i^{-1}-b(z_i)}}{\pm}
    =\poch{q}{n}(-1)^n\pf\brac*{\frac{z_i^{-1}+z_j^{-1}}{z_i-z_j}}_{1\le i<j\le 2n}\nonumber\\
    &=\poch{q}{n}(-1)^n\prod_{i=1}^{2n}z_i^{-1}\pf\brac*{\frac{z_i+z_j}{z_i-z_j}}_{1\le i<j\le 2n},
  \end{align}
  where the second equality follows from the identities of
  \cref{thm:bccorrs}. Multiplying by the Vandermonde determinant and applying
  the second identity of \eqref{eq:pfaffjacks} then gives
  \begin{align}
    \van{z}\NScorrfn{\gamma_0^n\prod_{i=1}^{2n}\brac*{\beta(z_i)c(z_i)-b(z_i)}}{q;\pm}=
    \poch{q}{n}(-1)^n\prod_{i=1}^{2n}z_i^{-1}\fjack{\admp{2n}{1,0}}{-3}{z}.
  \end{align}
  Finally, translating all variables by \(w\) (recall that the Vandermonde
  determinant is translation invariant), applying the Taylor
  expansion \eqref{eq:jacktaylor} and multiplying by \(\prod_{i=1}^{2n}z_i^{-2n+1}\) gives
  \begin{align}\label{eq:evencorrel}
    &\prod_{i=1}^{2n}z_i^{-2n+1}\van{z}\NScorrfn{\gamma_0^n\prod_{i=1}^{2n}\brac*{\beta(z_i+w)c(z_i+w)-b(z_i+w)}}{q;\pm}\nonumber\\
    &\quad=
    \poch{q}{n}(-1)^n\prod_{i=1}^{2n}\brac*{z_i+w}^{-1}\prod_{i=1}^{2n}z_i^{-2n+1}\fjack{\admp{2n}{1,0}}{-3}{z_1+w,\dots,z_{2n}+w}\nonumber\\
    &\quad=
    \poch{q}{n}\brac{-1}^n\prod_{i=1}^{2n}\brac*{z_i+w}^{-1}\sum_{i=0}^nw^{n-i}c^{(2n)}_i\fjack{[2n-1-\uniqp{2n}{i}]}{-3}{z^{-1}}.
  \end{align}

  The formulae for an odd number of \(\beta(z_i)c(z_i)-b(z_i)\) factors follow
  from the even case by specialisation. Consider
  \begin{align}
    &\prod_{i=0}^{2n+1}z_i^{-2n-1}\van{z_0,\dots,z_{2n+1}}\NScorrfn{\gamma_0^{n+1}\prod_{i=0}^{2n+1}\brac*{\beta(z_i+w)c(z_i+w)-b(z_i+w)}}{q;+}\nonumber\\
    &\quad=
    \prod_{i=1}^{2n+1}\brac*{1-\frac{z_i}{z_0}}\prod_{i=1}^{2n+1}z_i^{-2n+1}\van{z_1,\dots,z_{2n+1}}\nonumber\\
    &\qquad\NScorrfn{\gamma_0^{n+1}\brac[\bigg]{\sum_{\substack{k\ge0\\l\ge1}}\beta_kc_l\brac*{z_0+w}^{-k-l-1}-
        \sum_{k\ge0}b_k\brac*{z_0+w}^{-k-1}}\prod_{i=1}^{2n+1}\brac*{\beta(z_i+w)c(z_i+w)-b(z_i+w)}}{q;+}.
  \end{align}
  Multiplying by \(z_0\) and then taking the limit \(z_0\rightarrow\infty\)
  (or setting \(z_0^{-1}=0\) gives
  \begin{align}
    \prod_{i=1}^{2n+1}z_i^{-2n+1}\van{z_1,\dots,z_{2n+1}}
    \NScorrfn{\gamma_0^{n+1}b_0\prod_{i=1}^{2n+1}\brac*{\beta(z_i+w)c(z_i+w)-b(z_i+w)}}{q;+},
  \end{align}
  while multiplying the \rhs{} of \eqref{eq:evencorrel} by \(z_0\) and then
  taking the limit \(z_0\rightarrow\infty\) gives
  \begin{align}
    &\poch{q}{n+1}\brac{-1}^{n+1}\prod_{i=1}^{2n+1}\brac*{z_i+w}^{-1}\sum_{i=0}^{n+1}w^{n+1-i}c^{(2n+2)}_i\fjack{[2n+1-\uniqp{2n+2}{i}]}{-3}{0,z_1^{-1},\dots,z_{2n+1}^{-1}}\nonumber\\
    &\quad=
    \poch{q}{n+1}\brac{-1}^{n+1}\prod_{i=1}^{2n+1}\brac*{z_i+w}^{-1}\sum_{i=0}^{n+1}w^{n+1-i}c^{(2n+2)}_i\fjack{[2n+1-\uniqp{2n+1}{i}]}{-3}{z_1^{-1},\dots,z_{2n+1}^{-1}},
  \end{align}
  where we have used the fact that the \((2n+2)\)th part of the partitions
  \([2n+1-\uniqp{2n+1}{i}]\) is 0, so setting \(z_0^{-1}=0\) amounts to
  dropping the \((2n+2)\)th part of the partition and considering the corresponding Jack
  polynomials in \(2n+1\) variables.
\end{proof}

\section{The proof of \cref{thm:svimage}}
\label{sec:0modecalc}

In this section we prove \cref{thm:svimage} by evaluating the zero mode
of the singular vector on suitable candidate relaxed highest weight vectors
and thereby calculating the image of the singular vector in the Zhu
algebras. The evaluation of these zero modes boils down to evaluating inner
products of symmetric polynomials of the form
\begin{equation}\label{eq:zeromodeinprod}
  \jprod{\fjack{\kappa}{-3}{z},\prod_{i=1}^{n}\brac*{1+\frac{z_i}{w}}^{\lambda}}{n}{t},
\end{equation}
where \(\lambda\) is a complex number depending on free fields weights, \(\kappa\) is an admissible partition, so that the Jack polynomial
\(\fjack{\kappa}{-3}{z}\) is well defined, \(t\) is positive rational and the
inner product is the integral inner product \eqref{eq:intinprod} with respect
to which Jack polynomials at parameter value \(t\) are orthogonal. In order to
evaluate this inner product both the left and right argument need to be
expanded in Jack polynomials at parameter value \(t\). The right argument
expands as, see \cite[Eq.~(A.28)]{RidRel15} with \(X=-\lambda t\),
\begin{equation}
  \prod_{i=1}^{n}\brac*{1+\frac{z_i}{w}}^{\lambda}=
  \sum_{\mu}(-w)^{|\mu|}\prod_{b\in\mu}\frac{-\lambda t+ta^\prime(b)-\ell^\prime(b)}{t(a(b)+1)+\ell(b)}\fjack{\mu}{t}{z},
\end{equation}
where \(\mu\) runs over all partitions of integers of length at most
\(n\), \(b\) runs over all the cells in the Young diagram of \(\mu\), and
\(a(b), \ell(b), a^\prime(b)\) and \(\ell^\prime(b)\) are the arm length, leg
length, arm colength and leg colength, respectively, of \(b\). The left
argument satisfies an upper triangular decomposition
\(\fjack{\kappa}{-3}{z}=\fjack{\kappa}{t}{z}
    +\sum_{\mu < \kappa}c_{\kappa,\mu} \fjack{\mu}{t}{z},\)
where the \(c_{\kappa,\mu}\) are rational functions of \(t\) and \(\mu <\kappa\)
denotes the dominance partial ordering of partitions. The inner product
therefore evaluates to
\begin{equation}\label{eq:zeromodeinprodzeros}
  \jprod{\fjack{\kappa}{-3}{z},\prod_{i=1}^{n}\brac*{1+\frac{z_i}{w}}^{\lambda}}{n}{t}=(-w)^{|\kappa|}\sum_{\mu
    < \kappa} c_\mu \prod_{b\in\mu}\frac{-\lambda
    t+ta^\prime(b)-\ell^\prime(b)}{t(a(b)+1)+\ell(b)},
  \quad
  c_\mu=c_{\kappa,\mu}\jprod{\fjack{\mu}{t}{z},\fjack{\mu}{t}{z}}{n}{t}.
\end{equation}
Note in particular, that the coefficients \(c_\mu\) do not depend on the exponent
\(\lambda\). So the inner product \eqref{eq:zeromodeinprod} is a polynomial in
\(\lambda\) and because the numerators of the summands on the \rhs{} of \eqref{eq:zeromodeinprodzeros} depend only
on arm and leg colengths, the cells common to all partitions dominated by
\(\kappa\) will give rise to factors common to the numerators of all summands
and hence zeros in the parameter \(\lambda\) common to all summands. Finding these
common zeros is key to the proofs below and we shall find them using Lemma 3.2
of \cite{BloSVir16}.

\subsection{The case when \(v>u\) and \(k<-1\)}
\label{sec:vgu}

Throughout this section we assume that \(u,v\) are integers satisfying
\(v>u\ge2\), \(u-v\in2\ZZ\) and \(\gcd\brac*{u,\tfrac{u-v}{2}}\) and that
\(\xi=\sqrt{\frac{u}{v}}\). We will
prove \cref{thm:svimage} under the assumption \(v>u\) by using the screening
operator \(\scr{1}\) to construct the singular vector of
\(\vosp{k_{u,v}}\) and then evaluating the action of its zero mode on
relaxed highest weight vectors to deduce its image in the Zhu algebras.
Note that for \(v>u\), the exponent \(\frac{\xi^{-2}-1}{2}=\frac{v-u}{2u}\) in
\eqref{eq:vandscrs} is positive rational and will form the Jack
polynomial parameter.

\begin{lem}
  The singular vector of \(\vosp{k_{u,v}}\), as a subalgebra of
  \(\ffa\), is given by
  \begin{equation}
    \vsv{u,v}=\scrs{1}{u-1}\NSfket{\sfrac{(u-1)}{\xi}},
  \end{equation}
  where \(\NSfket{\sfrac{(u-1)}{\xi}}\) is the highest weight vector of
  \(\NSFFaFock{\sfrac{(u-1)}{\xi}}\).
\end{lem}

The prove this lemma we consider both symmetric polynomials and their infinite-variable limits, the symmetric functions.  For easy visual distinction, we denote the infinite alphabet of
variables for symmetric functions by \(y=(y_1,y_2,\dots)\) and the finite alphabet of
variables for symmetric polynomials by \(z=(z_1,\dots,z_n)\).
We shall use both the infinite- and finite-variable inner products
\(\cjprod{\cdot,\cdot}{t}\) and \(\jprod{\cdot,\cdot}{n}{t}\), referring to
\cite[App.~A]{RidRel15} for our conventions, as well as 
the identity (see \cite[Eq.~(A.16)]{RidRel15})
\begin{equation}\label{eq:cauchykernel}
    \prod_{m\ge 1}\exp\brac[\Bigg]{\frac{1}{t} \frac{\fpowsum{m}{z}\fpowsum{m}{y}}{m}} = \sum_{\lambda} \fjack{\lambda}{t}{z} \fdjack{\lambda}{t}{y},
\end{equation}
where $\powsum{m}$ is the $m$-th power sum and the \(\fdjack{\lambda}{t}{y}\)
are the symmetric functions dual (with respect to \(\cjprod{\cdot,\cdot}{t}\))
to the Jack symmetric functions \(\fjack{\lambda}{t}{y}\).  

A simple, but very useful, observation concerning 
the ring of symmetric functions \(\Lambda\) 
is that it is isomorphic, as a commutative algebra, to the \uea{} of either
the positive or negative subalgebra, $\hlie{1}^+$ or $\hlie{1}^-$, of the rank
1 Heisenberg
algebra \(\hlie{1}\). We denote the
corresponding isomorphisms by
\begin{equation}
  \begin{aligned}
    \symiso{+}{\gamma}\colon\Lambda&\longrightarrow \CC[a_1,a_2,\dots],\\
    \powsum{m}&\longmapsto \gamma a_{m},
  \end{aligned}
\qquad
  \begin{aligned}
    \symiso{-}{\gamma}\colon\Lambda&\longrightarrow \CC[a_{-1},a_{-2},\dots],\\
    \powsum{m}&\longmapsto \gamma a_{-m},
  \end{aligned}
  \qquad \gamma\in\CC\setminus\set{0}.
\end{equation}
We shall use these isomorphisms
to identify inner products involving Heisenberg generators with the symmetric function
inner product \(\cjprod{\cdot,\cdot}{t}\). For example, one easily verifies in the power sum basis, hence for arbitrary
\(f,g\in \Lambda\), that
\begin{equation}\label{eq:isoprod}
  \cjprod{f,g}{t}=\bracketb{q}{\symiso{+}{t/\gamma}(f)\symiso{-}{\gamma}(g)}{q},
\end{equation}
where the \rhs{} is evaluated in the Fock space \(\Fock{q}\), for any \(q\in\CC\) and any $\gamma \in \CC \setminus \set{0}$.

\begin{proof}
  Since \(\NSfket{(u-1)\sqrt{\sfrac{v}{u}}}\) is an \(\aosp{}\) highest weight
  vector and \(\scrs{1}{u-1}\) is a module homomorphism, the image of
  \(\NSfket{(u-1)\sqrt{\sfrac{v}{u}}}\) must be either a singular vector or
  zero. 
  By \cref{thm:ffainj} the singular vector in \(\ffa\) is unique and
  therefore it is sufficient to show that
  \(\scrs{1}{u-1}\NSfket{(u-1)\sqrt{\sfrac{v}{u}}}\) is non-zero. We do this
  by explicitly evaluating certain matrix elements and verifying that they are
  non-zero. 

  Evaluating formula \eqref{eq:vandscrs} for \(\scrs{1}{u-1}\) on
  \(\NSfket{\sfrac{(u-1)}{\xi}}\) yields
  \begin{multline}
    \scrs{1}{u-1}\NSfket{\sfrac{(u-1)}{\xi}}=
    \int \prod_{i\neq j}\brac*{1-\frac{z_i}{z_j}}^{\frac{\xi^{-2}-1}{2}}\prod_{i=1}^{u-1}z_i^{2-\frac{u+v}{2}}
    \van{z}\brac*{\beta(z_1)c(z_1)-b(z_1)}\cdots\brac*{\beta(z_{u-1})c(z_{u-1})-b(z_{u-1})}\\
    \cdot
    \prod_{m\ge
      1}\exp\brac*{\frac{-1}{\xi}\frac{a_{-m}}{m}\fpowsum{m}{z}}\NSfket{0}\frac{\dd
    z_1\cdots \dd z_{u-1}}{z_1\cdots z_{u-1}}.
  \end{multline}
  To further evaluate this formula, we distinguish between \(u\) and \(v\) even or odd.

  Suppose first that \(u\) and \(v\) are odd. Then, by the identities
  \eqref{eq:gamma1correls} in \cref{thm:correlid},
  \begin{align}
    &\prod_{i=1}^{u-1}z_i^{2-\frac{u+v}{2}}\van{z}\NScorrfn{\frac{\gamma_{1}^{\sfrac{(u-1)}{2}}}{\frac{u-1}{2}!}
      \brac*{\beta(z_1)c(z_1)-b(z_1)}\cdots\brac*{\beta(z_{u-1})c(z_{u-1})-b(z_{u-1})}}{+}\nonumber\\
    &\qquad=\prod_{i=1}^{u-1}z_i^{\frac{u-v}{2}-1}\fjack{\admp{u-1}{0,0}}{-3}{z^{-1}}=
    \fjack{\kappa}{-3}{z^{-1}},
  \end{align}
  where \(\kappa\) is the admissible partition
  \(\kappa=\admp{u-1}{\frac{v-u}{2}+1,\frac{v-u}{2}+1}\) and \(\NScorrfn{\ }{+}\)
  denotes the combination of the \(\beta\gamma\) and \(bc\) correlation functions.
  The non-vanishing of \(\scrs{1}{u-1}\NSfket{\sfrac{(u-1)}{\xi}}\) then
  follows by evaluating the following matrix element as an integral of correlators:
  \begin{multline}
    \NSfbra{0}\frac{\gamma_{1}^{\frac{u-1}{2}}}{\frac{u-1}{2}!2^{\frac{u-1}{2}}}
    \symiso{+}{-\xi}
    \brac*{\fjack{\kappa}{\frac{2}{\xi^{-2}-1}}{y}}\scrs{1}{u-1}\NSfket{\sfrac{(u-1)}{\xi}}
    =\int \prod_{i\neq
      j}\brac*{1-\frac{z_i}{z_j}}^{\frac{\xi^{-2}-1}{2}}\fjack{\kappa}{-3}{z^{-1}}
    \fjack{\kappa}{\frac{2}{\xi^{-2}-1}}{z}\frac{\dd z_1\cdots \dd
      z_{u-1}}{z_1\cdots z_{u-1}}\\
    =
    \jprod{\fjack{\kappa}{-3}{z},\fjack{\kappa}{\frac{2}{\xi^{-2}-1}}{z}}{u-1}{\frac{2}{\xi^{-2}-1}}=
    \jprod{\fjack{\kappa}{\frac{2}{\xi^{-2}-1}}{z},\fjack{\kappa}{\frac{2}{\xi^{-2}-1}}{z}}{u-1}{\frac{2}{\xi^{-2}-1}}\neq 0,
  \end{multline}
  where the third equality uses the fact that Jack symmetric polynomials
  decompose upper triangularly (with leading coefficient equal to 1) in terms of Jack symmetric polynomials at
  different values of the parameter 
  and that Jack symmetric polynomials at parameter value
  \(\frac{2}{\xi^{-2}-1}\) are orthogonal with respect to the inner product
  \(\jprod{\ ,\ }{u-1}{\frac{2}{\xi^{-2}-1}}\).

  The case where \(u\) and \(v\) are even follows similarly by evaluating the matrix element
  \begin{multline}
    \NSfbra{0}\frac{\gamma_{1}^{\frac{u-2}{2}}}{\frac{u}{2}!2^{\frac{u}{2}}}b_0
    \symiso{+}{-\xi}
    \brac*{\fjack{\kappa}{\frac{2}{\xi^{-2}-1}}{y}}\scrs{1}{u-1}\NSfket{\sfrac{(u-1)}{\xi}}
    =\jprod{\fjack{\kappa}{-3}{z},\fjack{\kappa}{\frac{2}{\xi^{-2}-1}}{z}}{u-1}{\frac{2}{\xi^{-2}-1}}\\
    =  \jprod{\fjack{\kappa}{\frac{2}{\xi^{-2}-1}}{z},\fjack{\kappa}{\frac{2}{\xi^{-2}-1}}{z}}{u-1}{\frac{2}{\xi^{-2}-1}}\neq 0,
  \end{multline}
  where \(\kappa=\admp{u-1}{\frac{v-u}{2}+2,\frac{v-u}{2}}\).
\end{proof}

\begin{lem}\label{thm:nszhupoly1}
  Up to an irrelevant scale factor, the polynomial \(g(h,\Sigma)\) of
  \cref{thm:zhusvshape} is given by
  \begin{equation}\label{eq:centralfactor2}
    g(h,\Sigma)=
    \prod_{(i,j)\in \NSkac{u,v}}
    \brac*{\Sigma-s_{i,j}},\quad
    s_{i,j}=\frac{i}{2}-\frac{u\,j}{2v}.
  \end{equation}
  In particular, \(g(h,\Sigma)\) does not depend on \(h\).
\end{lem}
\begin{proof}
  As stated in \cref{thm:zhusvshape}, the degree of 
  the polynomial \(g(h,\Sigma)\) is bounded by \(\frac{(u-1)(v-1)}{2}\).
  We shall show
  that the factors \(\brac*{\Sigma-s_{i,j}}\) of \eqref{eq:centralfactor2}
  divide \(g(h,\Sigma)\) by showing that \(g(h,s_{i,j})\) vanishes independently
  of \(h\). Since these factors saturate the maximal degree of \(g\) and \(g\)
  is non-zero by
  \cref{thm:svnonzeroimage}, they
  uniquely determine \(g\) up to scaling.

  Consider first the case when \(u\) and \(v\) are odd. The zero
  mode of the singular vector maps the relaxed highest weight vectors
  \(\NSfket{p;q}\) and \(c_0\NSfket{p;q}\) to multiples of
  \(\beta_0^{\frac{u-1}{2}}\NSfket{p;q}\) and
  \(\beta_0^{\frac{u-1}{2}}c_0\NSfket{p;q}\), respectively. This mapping is non-zero if and only if the following matrix elements are
  also non-zero.
  \begin{align}\label{eq:vgunsoddsvcorrel}
    \NSffbracket{p;q}{ \gamma_0^{\frac{u-1}{2}}\scrs{1}{u-1}\vop{\frac{u-1}{\xi}a}{w}}{p;q}
    &=\int
    \FNScorrfn{\gamma_0^{\frac{u-1}{2}}\prod_{i=1}^{u-1}\brac*{\beta(z_i+w)c(z_i+w)-b(z_i+w)}}{q}{\pm}\nonumber\\
    &\qquad\bracketb{pa}{\prod_{i=1}^{u-1}\vop{\frac{-a}{\xi}}{z_i+w}\cdot\vop{\frac{u-1}{\xi}a}{w}}{pa}
    \dd z_1\cdots\dd z_{u-1}\nonumber\\
    &=\int \prod_{i\neq
      j}\brac*{1-\frac{z_i}{z_j}}^{\frac{\xi^{-2}-1}{2}}\prod_{i=1}^{u-1}\brac*{1+\frac{z_i}{w}}^{-\sfrac{p}{\xi}} 
    \prod_{i=1}^{u-1}z_i^{2-\frac{u+v}{2}} \nonumber\\
    &\qquad \van{z}
    \FNScorrfn{\gamma_0^{\frac{u-1}{2}}\prod_{i=1}^{u-1}\brac*{\beta(z_i+w)c(z_i+w)-b(z_i+w)}}{q}{\pm}
    \frac{\dd z_1\cdots\dd z_{u-1}}{z_1\cdots z_{u-1}}\nonumber\\
    &= w^{\frac{1-u}{2}}\sum_{i=0}^{\frac{u-1}{2}}w^{-i}c_i
    \jprod{\fjack{\sqbrac*{\frac{v+u}{2}-2-\uniqp{u-1}{i}}}{-3}{z},\prod_{i=1}^{u-1}\brac*{1+\frac{z_i}{w}}^{-\frac{p}{\xi}-1}}{u-1}{\frac{2}{\xi^{-2}-1}},
  \end{align}
  the second equality follows from the identities 
  \eqref{eq:corrformulae} in \cref{thm:correlid}.
  By the argument outlined at the beginning of \cref{sec:0modecalc}, we need
  to identify the cells common to all partitions dominated by
  \(\sqbrac*{\frac{v+u}{2}-2-\uniqp{u-1}{i}},\ i=0,\dots,\frac{u-1}{2}\). The
  parts of the admissible partitions
  \(\sqbrac*{\frac{v+u}{2}-2-\uniqp{u-1}{i}}\) are all bounded below by the
  parts of
  \(\sqbrac*{\frac{v+u}{2}-2-\admp{u-1}{1,0}}=\admp{u-1}{\frac{v-u}{2}+1,\frac{v-u}{2}}\)
  and so by \cite[Lem.~3.2.(3)]{BloSVir16}, the cells common to all partitions
  dominated by the \(\sqbrac*{\frac{v+u}{2}-2-\uniqp{u-1}{i}}\) form the
  length \(u-1\)
  partition
  \(\rho=\sqbrac*{\frac{v-1}{2},\frac{v-1}{2}-1,\frac{v-1}{2}-1,\dots,\frac{v-u}{2}+1,\frac{v-u}{2}+1,\frac{v-u}{2}}\),
  that is, \(\rho_1=\frac{v-1}{2}=\rho_2+1\) and \(\rho_{i}=\rho_{i-2}+1\) for \(i=3,\dots,u-1\).
  Therefore, by \eqref{eq:zeromodeinprodzeros}, the zeros of 
  \begin{align}
    \prod_{b\in\rho}\brac*{\frac{2}{\xi^{-2}-1}\brac*{\frac{p}{\xi}+1}+\frac{2}{\xi^{-2}-1}a^\prime(b)-\ell^\prime(b)}
  \end{align}
  are common to all the summands of \eqref{eq:vgunsoddsvcorrel}.
  Up
  to irrelevant scale factors, this polynomial can be rewritten as
  \begin{multline}
    \prod_{b\in\rho}\brac*{\xi
      p+\frac{1}{2}-\frac{1}{2}\brac*{\ell'(b)+1}+\frac{\xi^2}{2}\brac*{\ell'+2a'(b)+2}}
    =
    \prod_{b\in\rho}\brac*{\pm s_p - s_{\ell'(b)+1,\ell'(b)+2a'(b)+2}}\\
    =\prod_{j=1}^{\frac{v-1}{2}}\brac*{\pm s_p - s_{1,2j}}
    \prod_{i=1}^{\frac{u-3}{2}}\prod_{j=1}^{\frac{v-1}{2}-i}\brac*{\pm s_p -
      s_{2i,2i+2j-1}}\brac*{\pm s_p - s_{2i+1,2i+2j}}
    \prod_{j=1}^{\frac{v-u}{2}}\brac*{\pm s_p - s_{u-1,u+2j-2}},
  \end{multline}
  where we recall that \(s_p=\pm\brac*{\xi p+\frac{1}{2}}\) are the respective eigenvalues of
  \(\Sigma\) acting on the relaxed highest weight vectors \(\NSfket{p;q}\) and
  \(c_0\NSfket{p;q}\).
  So for \((i,j)\in\NSkac{u,v}\), \(i<j\) we have that \(g(h,s_{i,j})=0\) and
  \(g(h,-s_{i,j})=0\). Since \(-s_{i,j}=s_{u-i,v-j}\), this implies
  \(g(h,s_{i,j})=0\) for all \((i,j)\in\NSkac{u,v}\). Note we have used the
  fact that \(v>u\), otherwise the ordered pairs \((i,j), (u-i,v-j)\), with
  \(i<j\), 
  would not cover the entire \ns{} Kac table.

  Next, consider the case of even \(u\) and \(v\). The action of the zero
  mode of the singular vector acting on the relaxed highest weight vectors
  \(\NSfket{p;q}\) and \(c_0\NSfket{p;q}\) is non-zero if and only if the following matrix elements are
  also non-zero.
  \begin{align}\label{eq:vgunevensvcorrel}
    \NSffbracket{p;q}{\gamma_0^{\frac{u}{2}}b_0\scrs{1}{u-1}\vop{\frac{u-1}{\xi}}{w}}{p;q},\quad
    \NSffbracket{p;q}{b_0\gamma_0^{\frac{u-2}{2}}c_0\scrs{1}{u-1}\vop{\frac{u-1}{\xi}}{w}c_0}{p;q}
  \end{align}
Up to irrelevant factors involving \(q\), both matrix elements evaluate to
\begin{equation}
  w^{\frac{3u}{2}-2}\sum_{i=1}^{\frac{u}{2}-1}
  w^{-i}c_i^{(u-1)}
  \jprod{\fjack{\sqbrac*{\frac{v+u}{2}-2-\uniqp{u-1}{i}}}{-3}{z},\prod_{i=1}^{u-1}\brac*{1+\frac{z_i}{w}}^{-\frac{p}{\xi}-1}}{u-1}{\frac{2}{\xi^{-2}-1}}.
\end{equation}
As in the case of odd \(u,v\) we identify zeros by finding cells common to all Young
diagrams of partitions dominated by
\(\sqbrac*{\frac{v+u}{2}-2-\uniqp{u-1}{i}}\). By
\cite[Lem.~3.2.(4)]{BloSVir16},
\(\rho=\sqbrac*{\frac{v}{2}-1,\frac{v}{2}-1,\frac{v}{2}-2,\frac{v}{2}-2,\dots,\frac{v-u}{2}+1,\frac{v-u}{2}+1,\frac{v-u}{2}}\) 
is the partition of length \(u-1\) formed by these cells.
%these cells combine to form the length \(u-1\) partition
The zeros of
\begin{multline}
    \prod_{b\in\rho}\brac*{\xi
      p+\frac{1}{2}-\frac{1}{2}\brac*{\ell'(b)+1}+\frac{\xi^2}{2}\brac*{\ell'+2a'(b)+2}}
    =
    \prod_{b\in\rho}\brac*{\pm s_p - s_{\ell'(b)+1,\ell'(b)+2a'(b)+2}}\\
    =\prod_{i=1}^{\frac{u}{2}-1}\prod_{j=1}^{\frac{v}{2}-i}\brac*{\pm s_p -
      s_{2i-1,2i+2j-2}}\brac*{\pm s_p - s_{2i,2i+2j-1}}
    \prod_{j=1}^{\frac{v-u}{2}}\brac*{\pm s_p - s_{u-1,u+2j-2}}.
  \end{multline}
  are therefore also zeros of the original matrix elements \eqref{eq:vgunevensvcorrel}.
  As in the case for \(u\) and \(v\) odd, this implies that \(g(h,s_{i,j})=0\)
  for all \((i,j)\in \NSkac{u,v}\).
\end{proof}

\begin{lem}\label{thm:rzhupoly1}
  Up to an irrelevant scale factor, the polynomial \(g_\parity(h,Q)\) of
  \cref{thm:zhusvshape} is given by
  \begin{equation}\label{eq:Rcentralfactor}
    g_\parity(h,Q)= 
    \prod_{[(i,j)]\in \rRkac{u,v}}\brac*{Q-q_{i,j}},\quad
    q_{i,j}=\frac{\brac*{u\,j-v\,i}^2-4v^2}{8v^2}.
  \end{equation}
    In particular, \(g_\parity(h,Q)\) does not depend on \(h\).
\end{lem}
\begin{proof}
  As stated in \cref{thm:zhusvshape}, the degree of 
  the polynomial \(g_\parity(h,Q)\) is bounded by \(\frac{(u-1)(v-1)-1}{2}\).
  We shall show
  that the factors \(\brac*{Q-q_{i,j}}\) of \eqref{eq:Rcentralfactor}
  divide \(g_\parity(h,Q)\) by showing that \(g_\parity\) vanishes independently
  of \(h\) at appropriate values of \(Q\). Since these factors saturate the maximal degree of \(g_\parity\) and \(g_\parity\)
  is non-zero by
  \cref{thm:svnonzeroimage}, they
  uniquely determine \(g_\parity\) up to scaling.

  Consider first the case where \(u\) and \(v\) are odd. The
  zero mode of the singular vector maps the relaxed highest weight vector
  \(\Rfket{p;j}\) to a multiple of
  \(\beta_0^{\frac{u-1}{2}}\Rfket{p;j}\). This action is
  non-zero if and only if the following matrix element is non-zero:
  \begin{multline}\label{eq:rzeromodeclac}
    \Rffbracket{p;j}{\gamma_0^{\frac{u-1}{2}}\scrs{1}{u-1}\vop{\frac{u-1}{\xi}a}{w}}{p;j}=
    \\\int \prod_{i=1}^{u-1}\brac*{1-\frac{z_i}{z_j}}^{\frac{\xi^{-2}-1}{2}}\prod_{i=1}^{u-1}\brac*{1+\frac{z_i}{w}}^{-1}
    \prod_{i=1}^{u-1}z_i^{2-\frac{u+v}{2}}
    \van{z}
    \Rbgcorrfn{\gamma_0^{\frac{u-1}{2}}\prod_{i=1}^{u-1}\brac*{\beta(z_i)c(z_i)-b(z_i)}}{j}
    \frac{\dd z_1\cdots\dd z_{u-1}}{z_1\cdots z_{u-1}}
  \end{multline}
  Using the identities \eqref{eq:corrformulaeR} of \cref{thm:correlid}, this evaluates to
  \begin{equation}
    \poch{j}{\frac{u-1}{2}}\brac*{-2}^{\frac{u-1}{2}}\jprod{
      \fjack{\admp{u-1}{\frac{v-u}{2}+1,\frac{v-u}{2}+1}}{-3}{z},\prod_{i=1}^{u-1}\brac*{1+\frac{z_i}{w}}^{-\frac{p}{\xi}-\frac{1}{2}}}{u-1}{\frac{2}{\xi^{-2}-1}}.
  \end{equation}
  As in the proof of \cref{thm:nszhupoly1}, 
  cells common to all Young diagrams of partitions dominated by
  \(\admp{u-1}{\frac{v-u}{2}+1,\frac{v-u}{2}+1}\) can be used to identify zero
  of the above inner product. By \cite[Lem.~3.2.(1)]{BloSVir16} these cells
  form the length \(u-1\) partition \(\rho=\sqbrac*{\frac{v-1}{2},\frac{v-1}{2},\frac{v-1}{2}-1,\frac{v-1}{2}-1,\dots,\frac{v-u}{2}+1,\frac{v-u}{2}+1}\)
  and the zeros of \eqref{eq:rzeromodeclac} include the zeros of
  \begin{align}
    \prod_{b\in\rho}\brac*{\xi p+\frac{\xi^2}{2}+\xi^2
      a^\prime(b)-\frac{1-\xi^2}{2}\ell'(b)}
    &=
    \prod_{b\in\rho}\brac*{\xi
      p+\frac{1}{2}-\frac{1}{2}\brac*{\ell'(b)+1}+\frac{\xi^2}{2}\brac*{\ell'(b)+2a'(b)+1}}\nonumber\\
    &=\prod_{b\in\rho}\brac*{s_p-s_{\ell'(b)+1,\ell'(b)+2a'(b)+1}}
    =\prod_{\substack{ (i,j)\in\Rkac{u,v}\\i\le j}}
    \brac*{s_p-s_{i,j}}.
  \end{align}
  Since the above zeros do not depend on the \(\beta\gamma\) weight, 
  they correspond to values of \(Q\) at which the
  polynomial \(g_\parity(h,Q)\) vanishes independently of \(h\).
  Recall that \(q_p=\frac{1}{2}\brac*{s_p^2-1}\) and assume
  \(g_\parity(h,d)=0\) for some \(d\in \CC\) independent of \(h\), then \(g_\parity(h,q_p)\) is
  divided by
  \begin{equation}
    q_p-d=\frac{1}{2}\brac*{s_p^2-1}-d=\frac{1}{2}\brac*{s_p-\tilde{d}}\brac*{s_p+\tilde{d}},
  \end{equation}
  where \(\pm\tilde{d}\) are the square roots of \(2d+1\).
  So if \(\brac*{s_p-s_{i,j}}\) divides \(g(h,q_p)\), then
  \(\brac*{s_p+s_{i,j}}\) must do so also. Thus the product
  \(\brac*{s_p-s_{i,j}}\brac*{s_p+s_{i,j}}=\brac*{2q_p-s_{i,j}^2+1}=2\brac*{q_p-q_{i,j}}\)
  divides \(g(h,q_p)\). So since \(Q-q_{i,j}\) divides \(g_\parity\) for all \(i\le j,\
  (i,j)\in\Rkac{u,v}\) and every class in \(\rRkac{u,v}\) has a representative
  of this form (assuming \(v>u\)), this proves the lemma for \(u\) and \(v\) odd.

  Next consider the case when \(u\) and \(v\) are even, then the action of the
  zero mode of the singular vector on a relaxed highest weight vector
  \(\Rfket{p;j}\) is
  non-zero if and only if the following matrix element is non-zero.
  \begin{multline}\label{eq:zeromodematrixel}
    \Rffbracket{p;j}{\gamma_0^{\frac{u}{2}-1}\scrs{1}{u-1}c(w)\vop{\frac{u-1}{\xi}a}{w}}{p;j}=
    \\\int \prod_{i=1}^{u-1}\brac*{1-\frac{z_i}{z_j}}^{\frac{\xi^{-2}-1}{2}}
    \prod_{i=1}^{u-1}z_i^{2-\frac{u+v}{2}}
    \van{z}
    \Rbgcorrfn{\gamma_0^{\frac{u-2}{2}}\prod_{i=1}^{u-1}\brac*{\beta(z_i+w)c(z_i+w)-b(z_i+w)}c(w)}{j}
    \frac{\dd z_1\cdots\dd z_{u-1}}{z_1\cdots z_{u-1}}.
  \end{multline}
  Up to an irrelevant scale factor in \(j\) this evaluates to
  \begin{equation}
    \jprod{
      \fjack{\admp{u-1}{\frac{v-u}{2}+1,\frac{v-u}{2}+1}}{-3}{z},\prod_{i=1}^{u-1}\brac*{1+\frac{z_i}{w}}^{-\frac{p}{\xi}-\frac{1}{2}}}{u-1}{\frac{2}{\xi^{-2}-1}}.
  \end{equation}
  By \cite[Lem.~3.2.(2)]{BloSVir16}, the partition formed by the cells common
  to the Young diagrams of all partitions dominated by
  \(\admp{u-1}{\frac{v-u}{2}+1,\frac{v-u}{2}+1}\) is
  \(\rho=\sqbrac*{\frac{v}{2}-1,\frac{v}{2}-2,\frac{v}{2}-2,\dots,\frac{v-u}{2}+1,\frac{v-u}{2}+1}\).
  Thus, the zeros of the matrix element \eqref{eq:zeromodematrixel} include
  those of
  \begin{equation}
    \prod_{b\in\rho}\brac*{s_p-s_{\ell'(b)+1,\ell'(b)+2a'(b)+1}}
    =\prod_{\substack{(i,j)\in\Rkac{u,v}\\i\le j}}\brac*{s_p-s_{i,j}}.
  \end{equation}
  The remainder of the proof then follows as in the case of \(u\) and \(v\) odd.
\end{proof}

\subsection{The case when \(u>v\) and \(k>-1\)}
Throughout this section we will assume that \(u,v\) are integers satisfying
\(u>v\ge1\), \(u-v\in2\ZZ\) and \(\gcd\brac*{u,\tfrac{(u-v)}{2}}\) and that
\(\xi=\sqrt{\frac{u}{v}}\). We will
prove \cref{thm:svimage} under the assumption \(u>v\) by using the screening
operator \(\scr{2}\) to construct the singular vector of
\(\vosp{k_{u,v}}\) and then evaluating the action of its zero mode on
relaxed highest weight vectors to deduce its image on the Zhu
algebras.
Note that for \(u>v\), the exponent \(\frac{\xi^2-1}{2}=\frac{u-v}{2v}\) in
\eqref{eq:vandscrs2} is positive rational and will form the Jack
polynomial parameter.
Throughout this section we will be using the second free field
realisation of \(\vosp{k}\) in which \(\beta(z)=\vop{\theta+\psi}{z}\).

\begin{lem}
  The singular vector of \(\vosp{k_{u,v}}\), as a subalgebra of
  \(\ffb\), is given by
  \begin{equation}\label{eq:scrsvb}
    \vsv{u,v}=\scrs{2}{v-1}\NSfket{\brac*{1-v}\xi;\brac*{u-\frac{u+v}{2v}}\brac*{\theta+\psi}}
  \end{equation}
  where \(\NSfket{\brac*{1-v}\xi;\brac*{u-\frac{u+v}{2v}}\brac*{\theta+\psi}}\) is the highest weight vector of
  \(\Fock{\brac*{1-v}\xi a}\otimes \Fock{\brac*{u-\frac{u+v}{2v}}\brac*{\theta+\psi}}\otimes\NSFock\).
\end{lem}
\begin{proof}
  Since the singular vector is unique by \cref{thm:bosinj}, it suffices to
  show that \eqref{eq:scrsvb} does not vanish. As in the previous section,
  we do this by verifying that certain matrix elements are
  non-zero.

  Expanding \eqref{eq:scrsvb} as an integral of multiple copies of the
  screening field \(\scrf{2}{z}\) gives
  \begin{multline}
    \int \prod_{i\neq
      j}\brac*{1-\frac{z_i}{z_j}}^{\frac{\xi^2-1}{2}}\prod_{i=1}^{v-1}z_i^{2-\frac{u+v}{2}}\van{z}
    \prod_{i=1}^{v-1}\brac*{\beta(z_i)c(z_i)-b(z_i)}\prod_{i=1}^{v-1}\vop{-\frac{u+v}{2v}\brac*{\theta+\psi}}{z_i}\\
    \prod_{m\ge
      1}\exp\brac*{\xi\frac{a_{-m}}{m}\fpowsum{m}{z}}\NSfket{0;\brac*{u-\frac{u+v}{2v}}\brac*{\theta+\psi}}
    \frac{\dd z_1\cdots \dd z_{v-1}}{z_1\cdots z_{v-1}}
  \end{multline}
  For \(u\) and \(v\) odd, we show that this is non-zero by computing the
  matrix element
  \begin{equation}
    \NSffbracket{0;\frac{u-1}{2}\brac*{\theta+\psi}}{
    \symiso{+}{\xi^{-1}}
    \brac*{\fjack{\kappa}{\frac{2}{\xi^{2}-1}}{y}}
    \scrs{2}{v-1}}
  {\brac*{1-v}\xi;\brac*{u-\frac{u+v}{2v}}\brac*{\theta+\psi}},
  \end{equation}
  where \(\kappa=\admp{v-1}{1+\frac{u-v}{2},1+\frac{u-v}{2}}\). The above
  matrix element then evaluates to
  \begin{align}
    &\int \prod_{i\neq
      j}\brac*{1-\frac{z_i}{z_j}}^{\frac{\xi^2-1}{2}}\fjack{\kappa}{\frac{2}{\xi^2-1}}{z}\prod_{i=1}^{v-1}z_i^{2-\frac{u+v}{2}}\van{z}\nonumber\\
    &\qquad\qquad\NSffbracket{\frac{u-1}{2}\brac*{\theta+\psi}}{\prod_{i=1}^{v-1}\brac*{\beta(z_i)c(z_i)-b(z_i)}\prod_{i=1}^{v-1}\vop{-\frac{u+v}{2v}\brac*{\theta+\psi}}{z_i}}{\brac*{u-\frac{u+v}{2v}}\brac*{\theta+\psi}}
    \frac{\dd z_1\cdots \dd z_{v-1}}{z_1\cdots z_{v-1}}
  \end{align}
  The combined \(\beta\gamma\) weight of the ket and the
  \(\vop{-\frac{u+v}{2v}\brac*{\theta+\psi}}{z_i}\) vertex operators is
  \(\frac{u-v}{2}\brac*{\theta+\psi}\), so in order to reach the bra's weight
  of \(\frac{u-1}{2}\brac*{\theta+\psi}\), the product of \(\beta\) and \(bc\)
  fields needs to contribute \(\frac{v-1}{2}\brac*{\theta+\psi}\), that is, the
  only summands of \(\prod_{i=1}^{v-1}\brac*{\beta(z_i)c(z_i)-b(z_i)}\) which
  contribute are those containing \(\beta\) exactly \(\frac{v-1}{2}\)
  times. Further, since \(\beta(z)=\vop{\theta+\psi}{z}\) and
  \(\vop{-\frac{u+v}{2v}\brac*{\theta+\psi}}{z}\) have regular operator
  product expansions with themselves and each other, the above matrix element
  simplifies to
  \begin{align}
    &\int \prod_{i\neq
      j}\brac*{1-\frac{z_i}{z_j}}^{\frac{\xi^2-1}{2}}\fjack{\kappa}{\frac{2}{\xi^2-1}}{z}\prod_{i=1}^{v-1}z_i^{2-\frac{u+v}{2}}\van{z}
    \NScorrfn{\prod_{i=1}^{v-1}\brac*{c(z_i)-b(z_i)}}{+}
    \frac{\dd z_1\cdots \dd z_{v-1}}{z_1\cdots z_{v-1}}\nonumber\\
    &\qquad=
    \int \prod_{i\neq
      j}\brac*{1-\frac{z_i}{z_j}}^{\frac{\xi^2-1}{2}}\fjack{\kappa}{\frac{2}{\xi^2-1}}{z}
    \fjack{\kappa}{-3}{z^{-1}}
    \frac{\dd z_1\cdots \dd z_{v-1}}{z_1\cdots z_{v-1}}\nonumber\\
    &\qquad=
    \jprod{\fjack{\kappa}{\frac{2}{\xi^2-1}}{z},\fjack{\kappa}{-3}{z}}{v-1}{\frac{2}{\xi^2-1}}
    =
    \jprod{\fjack{\kappa}{\frac{2}{\xi^2-1}}{z},\fjack{\kappa}{\frac{2}{\xi^2-1}}{z}}{v-1}{\frac{2}{\xi^2-1}}\neq0.
  \end{align}
  The case for \(u\) and \(v\) odd follows similarly by evaluating the
  matrix element
  \begin{equation}
    \NSffbracket{0;\frac{u}{2}\brac*{\theta+\psi}}{b_0
    \symiso{+}{\xi^{-1}}
    \brac*{\fjack{\kappa}{\frac{2}{\xi^{2}-1}}{y}}
    \scrs{2}{v-1}}
  {\brac*{1-v}\xi;\brac*{u-\frac{u+v}{2v}}\brac*{\theta+\psi}}
  \end{equation}
  and verifying that it is non-zero.
\end{proof}

\begin{lem}\label{thm:nszhupoly2}
  Up to an irrelevant scale factor, the polynomial \(g(h,\Sigma)\) of
  \cref{thm:zhusvshape} is given by
  \begin{equation}
    g(h,\Sigma)=
    \prod_{(i,j)\in \NSkac{u,v}}
    \brac*{\Sigma-s_{i,j}},\quad
    s_{i,j}=\frac{i}{2}-\frac{u\,j}{2v}.
  \end{equation}
    In particular, \(g(h,\Sigma)\) does not depend on \(h\).
\end{lem}
\begin{proof}
  This proof mirrors that of \cref{thm:nszhupoly1} but with the screening
  operator \(\scr{1}\) replaced by \(\scr{2}\), that is, we will show that the
  factors \(\brac*{\Sigma-s_{i,j}}\) all divide \(g(h,\Sigma)\), thus saturating the
  degree bound on \(g\) and determining \(g\) up to scaling.
  
  Recall that by \cref{thm:bgffrmod} we have
  \(\Fock{\sqbrac*{\lambda\brac*{\theta+\psi}-\psi}}\cong \bgd{\lambda}\) as a
  \(\beta\gamma\) module and so \(\NSfket{p;j\brac*{\theta+\psi}-\psi}\) and \(c_0\NSfket{p;j\brac*{\theta+\psi}-\psi}\) are
  \(\aosp{}\) relaxed highest weight vectors. The zero mode of the singular
  vector shifts the \(\aosp{}\) weights of such relaxed highest weight vectors
  by \((u-1)\alpha\). For \(u\) and \(v\) odd, this means a shift of
  \(\frac{u-1}{2}\brac*{\theta+\psi}\) in \(\beta\gamma\) weight, and for \(u\)
  and \(v\) even, it means a shift of \(\frac{u-2}{2}\brac*{\theta+\psi}\) in
  \(\beta\gamma\) weight and 1 unit of \(bc\) weight.
  
  Consider the case when \(u\) and \(v\) are odd, then the action of the zero
  mode of the singular vector on the relaxed highest weight vectors
  \(\NSfket{p;j\brac*{\theta+\psi}-\psi}\) and
  \(c_0\NSfket{p;j\brac*{\theta+\psi}-\psi}\)
  is non-zero if and only if the matrix elements
  \begin{align}
    &\NSffbracket{p;\brac*{j+\tfrac{u-1}{2}}\brac*{\theta+\psi}-\psi}{
      \scrs{2}{v-1}\vop{(1-v)\xi
        a+\brac*{u-\frac{u+v}{2v}}\brac*{\theta+\psi}}{w}}
    {p;j\brac*{\theta+\psi}-\psi}
  \end{align}
  and
  \begin{align}
    &\NSffbracket{p;\brac*{j+\tfrac{u-1}{2}}\brac*{\theta+\psi}-\psi}{b_0
      \scrs{2}{v-1}\vop{(1-v)\xi
        a+\brac*{u-\frac{u+v}{2v}}\brac*{\theta+\psi}}{w}c_0}
    {p;j\brac*{\theta+\psi}-\psi}
  \end{align}
  are non-zero. Evaluating the first matrix element gives
  \begin{align}
    &\NSffbracket{p;\brac*{j+\tfrac{u-1}{2}}\brac*{\theta+\psi}-\psi}{
      \scrs{2}{v-1}\vop{(1-v)\xi
        a+\brac*{u-\frac{u+v}{2v}}\brac*{\theta+\psi}}{w}}
    {p;j\brac*{\theta+\psi}-\psi}\nonumber\\
    &\quad=
    \int 
    \prod_{i\neq j}\brac*{1-\frac{z_i}{z_j}}^{\frac{\xi^2-1}{2}}
    \prod_{i=1}^{v-1}\brac*{1+\frac{z_i}{w}}^{\xi p}
    \prod_{i=1}^{v-1}z_i^{2-\frac{u+v}{2}}\van{z}\nonumber\\ &\qquad\qquad
    \NSfbra{\brac*{j+\tfrac{u-1}{2}}\brac*{\theta+\psi}-\psi}
    \prod_{i=1}^{v-1}\brac*{\beta(z_i+w) c(z_i+w)-b(z_i+w)}\nonumber\\
    &\qquad\qquad
    \prod_{i=1}^{v-1}\vop{-\frac{u+v}{2v}\brac*{\theta+\psi}}{z_i+w}
    \cdot
    \vop{\brac*{u-\frac{u+v}{2v}}\brac*{\theta+\psi}}{w}
    \NSfket{j\brac*{\theta+\psi}-\psi}
    \frac{\dd z_1\cdots\dd z_{v-1}}{z_1\cdots z_{v-1}}
  \end{align}
  The combined \(\beta\gamma\) weight of the ket and the vertex operators in
  the integrand is \(\brac*{j+\frac{u-v}{2}}\brac*{\theta+\psi}-\psi\), so
  in order to reach the bra's weight of
  \(\brac*{j+\tfrac{u-1}{2}}\brac*{\theta+\psi}\) the product of \(\beta,b\)
  and \(c\) fields needs to contribute \(\frac{v-1}{2}\brac*{\theta+\psi}\), that
  is, after multiplying out the product, only summands containing
  \(\frac{v-1}{2}\) copies of \(\beta\) will contribute to the matrix
  element. Thus, the matrix element inside the integrand evaluates to
  \begin{align}
    &\NScorrfn{\prod_{i=1}^{v-1}\brac*{c(z_i+w)(z_i+w)^{-1}-b(z_i+w)}}{+}
    w^{\frac{u+v}{2v}-u}\prod_{i=1}^{v-1}\brac*{z_i+w}^{\frac{u+v}{2v}}\nonumber\\
    &\quad=
    \NScorrfn{\prod_{i=1}^{v-1}\brac*{c(z_i+w)(z_i+w)^{-1}-b(z_i+w)}}{+}
    w^{\frac{v-u}{2}}\prod_{i=1}^{v-1}\brac*{1+\frac{z_i}{w}}^{\frac{u+v}{2v}}
  \end{align}
  Thus the total matrix element evaluates to
  \begin{align}
    &\NSffbracket{p;\brac*{j+\tfrac{u-1}{2}}\brac*{\theta+\psi}-\psi}{
      \scrs{2}{v-1}\vop{(1-v)\xi
        a+\brac*{u-\frac{u+v}{2v}}\brac*{\theta+\psi}}{w}}
    {p;j\brac*{\theta+\psi}-\psi}\nonumber\\
    &\quad=
    w^{\frac{v-u}{2}}\int 
    \prod_{i\neq j}\brac*{1-\frac{z_i}{z_j}}^{\frac{\xi^2-1}{2}}
    \prod_{i=1}^{v-1}\brac*{1+\frac{z_i}{w}}^{\xi
      p+\frac{u+v}{2v}}\nonumber\\
    &\qquad
    \prod_{i=1}^{v-1}z_i^{2-\frac{u+v}{2}}\van{z}
    \NScorrfn{\prod_{i=1}^{v-1}\brac*{c(z_i+w)(z_i+w)^{-1}-b(z_i+w)}}{+}
    \frac{\dd z_1\cdots\dd z_{v-1}}{z_1\cdots z_{v-1}}\nonumber\\
    &\quad=\poch{j}{\frac{v-1}{2}}(-1)^{\frac{v-1}{2}}
    w^{\frac{v-u}{2}}\int 
    \prod_{i\neq j}\brac*{1-\frac{z_i}{z_j}}^{\frac{\xi^2-1}{2}}
    \prod_{i=1}^{v-1}\brac*{1+\frac{z_i}{w}}^{\xi
      p+\frac{u+v}{2v}-1}\nonumber\\
    &\qquad
    \prod_{i=1}^{v-1}z_i^{-\frac{u-v}{2}}
    \sum_{i=0}^{\frac{v-1}{2}}w^{\frac{v-1}{2}-i}c_i^{(v-1)}\fjack{\sqbrac*{v-2-\uniqp{v-1}{i}}}{-3}{z^{-1}}
    \frac{\dd z_1\cdots\dd z_{v-1}}{z_1\cdots z_{v-1}}\nonumber\\
    &\quad=
    \sum_{i=1}^{\frac{v-1}{2}}w^{\frac{v-1}{2}-i}c_i^{(v-1)}
    \jprod{\fjack{\sqbrac*{\frac{u+v}{2}+2-\uniqp{v-1}{i}}}{-3}{z},\prod_{j=1}^{v-1}\brac*{1+\frac{z_j}{w}}^{\xi
        p+\frac{u+v}{2v}-1}}{v-1}{\frac{2}{\xi^2-1}},
  \end{align}
  where the second equality makes use of the identities of \eqref{eq:corrformulae}.
  As in the proof of \cref{thm:nszhupoly1}, we can look for zeros common to
  all summands to find zeros of the action of the zero mode of the singular
  vector by identifying cells common to all Young diagrams of partitions
  dominated by the \(\sqbrac*{\frac{u+v}{2}+2-\uniqp{v-1}{i}}\). By
  \cite[Lem.~3.2.(3)]{BloSVir16}
  these cells form the length \(v-1\) partition
  \(\rho=\sqbrac*{\frac{u-1}{2},\frac{u-1}{2}-1,\frac{u-1}{2}-1,\dots,\frac{u-v}{2}+1,\frac{u-v}{2}+1,\frac{u-v}{2}}\)
  and the zeros common to all summands above include those of
  \begin{align}
    \prod_{b\in\rho}\brac*{\xi
      p+\frac{\xi^2-1}{2}-a'(b)+\frac{\xi^2-1}{2}\ell'(b)}
    &=\prod_{b\in\rho}\brac*{s_p-s_{\ell'(b)+2a'(b)+2,\ell'(b)+1}}
    =\prod_{\substack{(i,j)\in \NSkac{u,v}\\ i>j}}\brac*{s_p-s_{i,j}}.
  \end{align}
  The analogous matrix element coming from the action of the zero mode of the
  singular vectors in the relaxed highest weight vector
  \(c_0\NSfket{p;j\brac*{\theta+\psi}-\psi}\) gives
  \begin{align}
    &\NSffbracket{p;\brac*{j+\tfrac{u-1}{2}}\brac*{\theta+\psi}-\psi}{b_0
      \scrs{2}{v-1}\vop{(1-v)\xi
        a+\brac*{u-\frac{u+v}{2v}}\brac*{\theta+\psi}}{w}c_0}{p;j\brac*{\theta+\psi}-\psi}\nonumber\\
    &\qquad=
    w^{\frac{v-u}{2}}\int 
    \prod_{i\neq j}\brac*{1-\frac{z_i}{z_j}}^{\frac{\xi^2-1}{2}}
    \prod_{i=1}^{v-1}\brac*{1+\frac{z_i}{w}}^{\xi
      p+\frac{u+v}{2v}}\nonumber\\
    &\qquad
    \prod_{i=1}^{v-1}z_i^{2-\frac{u+v}{2}}\van{z}
    \NScorrfn{\prod_{i=1}^{v-1}\brac*{c(z_i+w)(z_i+w)^{-1}-b(z_i+w)}}{-}
    \frac{\dd z_1\cdots\dd z_{v-1}}{z_1\cdots z_{v-1}}\nonumber\\
    &\quad=
    \sum_{i=1}^{\frac{v-1}{2}}w^{\frac{v-1}{2}-i}c_i^{(v-1)}
    \jprod{\fjack{\sqbrac*{\frac{u+v}{2}+2-\uniqp{v-1}{i}}}{-3}{z},\prod_{j=1}^{v-1}\brac*{1+\frac{z_j}{w}}^{\xi
        p+\frac{u+v}{2v}-1}}{v-1}{\frac{2}{\xi^2-1}}.
  \end{align}
  As above, the zeros common to all summands include those of
  \begin{align}
    \prod_{b\in\rho}\brac*{\xi
      p+\frac{\xi^2-1}{2}-a'(b)+\frac{\xi^2-1}{2}\ell'(b)} 
    &=\prod_{b\in\rho}\brac*{-s_p-s_{\ell'(b)+2a'(b)+2,\ell'(b)+1}}
    =\prod_{\substack{(i,j)\in \NSkac{u,v}\\ i>j}}\brac*{-s_p-s_{i,j}}.
  \end{align}
  Thus,
  \(g(h,s_{i,j})=0\) independent of \(h\) for all \((i,j)\in\NSkac{u,v}\).
  Note that we have used the fact that \(u>v\), otherwise the ordered pairs
  \((i,j)\) and \((u-i,v-j)\), with \(i>j\), would not cover the entire \ns{}
  Kac table.

  For \(u\) and \(v\) even, the action of the zero mode of the singular vector
  on the relaxed highest weight vectors
  \(\NSfket{p;j\brac*{\theta+\psi}-\psi}\) 
  and \(c_0\NSfket{p;j\brac*{\theta+\psi}-\psi}\) is non-zero if and only if the matrix elements
  \begin{align}
    &\NSffbracket{p;\brac*{j+\tfrac{u-2}{2}}\brac*{\theta+\psi}-\psi}{b_0
      \scrs{2}{v-1}\vop{(1-v)\xi
        a+\brac*{u-\frac{u+v}{2v}}\brac*{\theta+\psi}}{w}}
    {p;j\brac*{\theta+\psi}-\psi}\nonumber\\
    &\NSffbracket{p;\brac*{j+\tfrac{u}{2}}\brac*{\theta+\psi}-\psi}{
      \scrs{2}{v-1}\vop{(1-v)\xi
        a+\brac*{u-\frac{u+v}{2v}}\brac*{\theta+\psi}}{w}c_0}
    {p;j\brac*{\theta+\psi}-\psi}
  \end{align}
  are non-zero.
  By the same reasoning as in the case of odd \(u\) and \(v\) above, these
  matrix elements both evaluate to
  \begin{align}
    \sum_{i=0}^{\frac{v}{2}-1}w^{-\frac{u}{2}-i}c_i
    \jprod{\prod_{i=1}^{v-1} z_i^{\frac{u-v}{2}}
      \fjack{\sqbrac*{v-2-\uniqp{v-1}{i}}}{-3}{z},\prod_{i=1}^{v-1}\brac*{1+\frac{z_i}{w}}^{\xi
        p+\frac{\xi^2+1}{2}-1}}{v-1}{\frac{2}{\xi^{2}-1}}.
  \end{align}
  By \cite[Lem.~3.2.(4)]{BloSVir16}, the cells common to all Young diagrams dominated
  by \(\sqbrac*{v-2-\uniqp{v-1}{i}}\) form the length \(v-1\) partition
  \(\rho=\sqbrac*{\frac{u}{2}-1,\frac{u}{2}-1,\dots,\frac{u-v}{2}+1,\frac{u-v}{2}+1,\frac{u-v}{2}}\).
  The zeros common to all summands therefore include those of
  \begin{align}
    \prod_{b\in\rho}\brac*{\pm s_p-s_{\ell'(b)+2a'(b)+2,\ell'(b)+1}}.
  \end{align}
  Thus,
  \(g(h,s_{i,j})=0\) independent of \(h\) for all \((i,j)\in\NSkac{u,v}\).
\end{proof}

\begin{lem}\label{thm:rzhupoly2}
  Up to an irrelevant scale factor, the polynomial \(g_\parity(h,Q)\) of
  \cref{thm:zhusvshape} is given by
  \begin{equation}
    g_\parity(h,Q)= 
    \prod_{[(i,j)]\in\rRkac{u,v}}
    \brac*{Q-q_{i,j}},\quad
    q_{i,j}=\frac{\brac*{u\,j-v\,i}^2-4v^2}{8v^2}.
  \end{equation}
    In particular, \(g_\parity(h,Q)\) does not depend on \(h\).
\end{lem}
\begin{proof}
  This proof mirrors that of \cref{thm:rzhupoly1} but with the screening
  operator \(\scr{1}\) replaced by \(\scr{2}\).
  For \(u\) and \(v\) odd, the action of the zero mode of the singular vector
  on the candidate relaxed highest weight vector is non-zero if and only if the following
  matrix element is non-zero.
  \begin{align}\label{eq:ugvrsvcalc}
    &\bracketb{p;\brac*{j+\tfrac{u-1}{2}}\brac*{\theta+\psi}-\psi;\Ra}{
    \scrs{2}{v-1}\vop{(1-v)\xi
      a+\brac*{u-\frac{u+v}{2v}}\brac*{\theta+\psi}}{w}}
  {p;j\brac*{\theta+\psi}-\psi;\Ra}\nonumber\\
  &\quad=
  w^{\frac{1-u}{2}}(-2)^{\frac{v-1}{2}}\int \prod_{i\neq j}\brac*{1-\frac{z_i}{z_j}}^{\frac{\xi^2-1}{2}}
  \prod_{i=1}^{v-1}\brac*{1+\frac{z_i}{w}}^{\xi p}
  \prod_{i=1}^{v-1}z_i^{2-\frac{u+v}{2}}\van{z}\nonumber\\
  &\qquad\qquad \bra{\brac*{j+\tfrac{u-1}{2}}\brac*{\theta+\psi}-\psi;\Ra}
  \prod_{i=1}^{v-1}\brac*{\beta(z_i+w)c(z_i+w)-b(z_i+w)}\nonumber\\
  &\qquad\qquad
  \prod_{i=1}^{v-1}\vop{-\frac{u+v}{2v}\brac*{\theta+\psi}}{z_i+w}\cdot
  \vop{\brac*{u-\frac{u+v}{2v}}\brac*{\theta+\psi}}{w}
  \ket{j\brac*{\theta+\psi}-\psi;\Ra}
  \frac{\dd z_1\cdots \dd z_{v-1}}{z_1\cdots z_{v-1}}
  \nonumber\\
  &\quad=
  w^{\frac{1-u}{2}}(-2)^{\frac{v-1}{2}}
  \jprod{
    \fjack{\admp{v-1}{\frac{u-v}{2}+1,\frac{u-v}{2}+1}}{-3}{z},\prod_{i=1}^{v-1}\brac*{1+\frac{z_i}{w}}^{\xi
       p+\frac{\xi^2}{2}}}{v-1}{\frac{2}{\xi^{2}-1}},
  \end{align}
  where the matrix element in the integrand is evaluated as in the proof of
  \cref{thm:nszhupoly2} (that is, the only summands of the product of
  \(\beta\), \(b\) and \(c\) fields which contribute are this containing
  \(\frac{v-1}{2}\) copies of \(\beta\))
  and using the identities of \eqref{eq:corrformulaeR}.
  As above in the proofs of \cref{thm:nszhupoly1}, \cref{thm:rzhupoly1} and \cref{thm:nszhupoly2},
  we identify cells common to all Young diagrams of partitions dominated by
  \(\admp{v-1}{\frac{u-v}{2}+1,\frac{u-v}{2}+1}\) to find zeros. By
  \cite[Lem.3.2.(1)]{BloSVir16} the cells form the length \(v-1\) partition
  \(\rho=\sqbrac*{\frac{u-1}{2},\frac{u-1}{2},\frac{u-1}{2}-1,\frac{u-1}{2}-1,\dots,\frac{u-v}{2}+1,\frac{u-v}{2}+1}\).
  The zeros of the matrix element \eqref{eq:ugvrsvcalc} therefore include
  those of
  \begin{align}
   \prod_{b\in\rho}\brac*{\xi p+\frac{\xi^2}{2}-a'(b)+\frac{\xi^2-1}{2}\ell'(b)}
   &=
   \prod_{b\in\rho}\brac*{\xi p +\frac{1}{2}
     -\frac{1}{2}\brac*{\ell'(b)+2a'(b)+1}+\frac{\xi^2}{2}\brac*{\ell'(b)+1}}
   \nonumber\\
   &=
   \prod_{b\in\rho}\brac*{s_p-s_{\ell'(b)+2a'(b)+1,\ell'(b)+1}}
   =\prod_{\substack{(i,j)\in\Rkac{u,v}\\i>j}}\brac*{s_p-s_{i,j}},
  \end{align}
  Also
  as in the proof of \cref{thm:rzhupoly1}, \(\brac*{s_p-s_{i,j}}\) dividing
  \(g_\parity(h,q_p)\) implies that
  \(\brac*{s_p+s_{i,j}}=\brac*{s_p-s_{u-r,v-s}}\) does so as well and thus
  \(q_p-q_{i,j}\) divides \(g_\parity(h,q_p)\) for all \((i,j)\in
  \Rkac{u,v}\), \(i>j\). Since \(u>v\) every class of
  \(\rRkac{u,v}\) has a representative of this form, thus proving the lemma
  for odd \(u\) and \(v\).

  For \(u\) and \(v\) even, one must consider the action of the zero mode of
  the \(y_0\) descendant of the singular vector. This action vanishes on the
  relaxed highest weight vector \(\Rfket{p;j\brac*{\theta+\psi}-\psi}\) if
  and only if the matrix element
  \begin{equation}
    \Rffbracket{p;\brac*{j+\tfrac{u-2}{2}}\brac*{\theta+\psi}-\psi}{
      \scrs{2}{v-1}c(w)\vop{(1-v)\xi
        a+\brac*{u-\frac{u+v}{2v}}\brac*{\theta+\psi}}{w}}
    {p;j\brac*{\theta+\psi}-\psi}
  \end{equation}
  vanishes. This matrix element
  evaluates to
  \begin{align}
    &\Rffbracket{p;\brac*{j+\tfrac{u-2}{2}}\brac*{\theta+\psi}-\psi}{
      \scrs{2}{v-1}c(w)\vop{(1-v)\xi
        a+\brac*{u-\frac{u+v}{2v}}\brac*{\theta+\psi}}{w}}
    {p;j\brac*{\theta+\psi}-\psi}\nonumber\\
    &\quad=
    -w^{1-\frac{u}{2}}(-2)^{\frac{v-2}{2}}
    \jprod{
      \fjack{\admp{v-1}{\frac{u-v}{2}+1,\frac{u-v}{2}+1}}{-3}{z},\prod_{i=1}^{v-1}\brac*{1+\frac{z_i}{w}}^{\xi
       p+\frac{\xi^2}{2}}}{v-1}{\frac{2}{\xi^{2}-1}}
  \end{align}
  The usual trick of looking for cells common to all Young diagrams of
  partitions dominated by \(\admp{v-1}{\frac{u-v}{2}+1,\frac{u-v}{2}+1}\)
  gives the length \(v-1\) partition
  \(\rho=\sqbrac*{\frac{u}{2},\frac{u}{2}-1,\frac{u}{2}-1,\frac{u}{2}-2,\frac{u}{2}-2,\dots,\frac{u-v}{2}+1,\frac{u-v}{2}+1}\)
  by \cite[Lem.~3.2.(2)]{BloSVir16}.
  The zeros of the above matrix element therefore include those of
  \begin{align}
   \prod_{b\in\rho}\brac*{s_p-s_{\ell'(b)+2a'(b)+1,\ell'(b)+1}}
   =\prod_{\substack{(i,j)\in\Rkac{u,v}\\i>j}}\brac*{s_p-s_{i,j}}.
  \end{align}
  As in the case of odd \(u\) and \(v\) this implies that
  \(g_\parity(h,q_{i,j})=0\) for all \([(i,j)]\in \rRkac{u,v}\).
\end{proof}

\flushleft
%\bibliography{osp12spectra}
%\bibliographystyle{unsrt}

\end{document}